\documentclass[12pt, reqno, a4paper]{amsart}
\usepackage[margin=1.2in]{geometry}
\numberwithin{equation}{section}
\usepackage{amssymb,amsfonts,amsthm}
\usepackage[utf8]{inputenc}
\usepackage{listings}
\usepackage{bm}
\usepackage[hyperpageref]{backref}
\usepackage{esint}
\usepackage{tikz}
\usepackage{bigints}
\usepackage{color}
\usepackage[bookmarks]{hyperref}
\usepackage{siunitx}
\usepackage{hyperref}
\allowdisplaybreaks[4]

\newtheoremstyle{myremark}{10pt}{10pt}{}{}{\bfseries}{.}{.5em}{}

\newtheorem{remark}{Remark}
 \newtheorem{thm}{Theorem}
 
 \newtheorem{lemma}{Lemma}[section]
 
 \theoremstyle{definition}
 \newtheorem{defn}{Definition}[section]

\usepackage{hyperref}

\allowdisplaybreaks[4]

\begin{document}

\title[Fractional  boundary Hardy inequality]{Fractional  boundary Hardy inequality for the critical cases}

\author{Adimurthi, Prosenjit Roy, AND Vivek Sahu}

\address{ Department of Mathematics and Statistics,
Indian Institute of Technology Kanpur, Kanpur - 208016, Uttar Pradesh, India}

\email{Adimurthi: adiadimurthi@gmail.com}
\email{Prosenjit Roy: prosenjit@iitk.ac.in}
\email{Vivek Sahu: viiveksahu@gmail.com}



\date{}

\dedicatory{}

\begin{abstract}
We establish  generalised fractional boundary Hardy-type inequality, in the spirit of Caffarelli-Kohn-Nirenberg inequality for different values of $s$ and $p$ on various domains in  $\mathbb{R}^d, ~ d \geq 1$.  In particular, for Lipschitz bounded domains any values of   $s$ and  $p$ are admissible, settling all the cases in subcritical, supercritical and critical regime. In this paper we have solved the open problems posed by Dyda for the critical case $sp =1$. Moreover we have proved the embeddings of $W^{s,p}_{0}(\Omega)$ in subcritical, critical and supercritical uniformly without using Dyda's decomposition. Additionally, we extend our results to include a weighted fractional boundary Hardy-type inequality for the critical case.
\end{abstract}

\maketitle
%\tableofcontents
%---------------------------INTRODUCTION---------------------------------

\section{Introduction}
In 1920, G. H. Hardy established the inequality known as the ``Hardy inequality" for a sequence of nonnegative real numbers. This inequality relates the  $L^{p}$ norm of a sequence to the  $L^{p}$ norm of its cumulative sums. In  \cite{hardy1920}, he proved that
\begin{equation*}
    \sum_{n=1}^{\infty} \left( \frac{a_{1}+ \dots + a_{n}}{n} \right)^{p} \leq \left( \frac{p}{p-1} \right)^{p} \sum_{n=1}^{\infty} a^{p}_{n},
\end{equation*}
where  $p>1$ and  $\{a_{n}\}_{n \in \mathbb{N}}$ is a sequence of nonnegative real numbers. E. Landau showed in  \cite{Landau1926} that the constant  $\left( \frac{p}{p-1} \right)^{p}$ is sharp and the equality holds if and only if  $a_{n}=0$ for all  $n \in \mathbb{N}$. The integral analogue of Hardy inequality, now called the classical one-dimensional Hardy inequality, was announced by Hardy in 1925. In  \cite{Hardy1925}, Hardy proved the inequality 
\begin{equation}\label{1 dim hardy inequality}
    \int_{0}^{\infty} \left( \frac{1}{x} \int_{0}^{x} u(t) \, dt \right)^{p} dx \leq \left( \frac{p}{p-1} \right)^{p} \int_{0}^{\infty} u(x)^{p} \, dx,
\end{equation}
where  $p>1$ and  $u$ is a nonnegative measurable function on  $(0, \infty)$. The constant  $\left( \frac{p}{p-1} \right)^{p}$ is again sharp, and equality holds only when  $u=0$ almost everywhere (see  \cite[Theorem 1.2.1]{balinsky2015analysis}). For the case  $p=1$, the inequality  \eqref{1 dim hardy inequality} fails (by taking  $\chi_{(0,1)}$) even if the constant  $\left( \frac{p}{p-1} \right)^{p}$ is replaced by any arbitrary positive constant  $C$. If $u \in L^{p}((0,\infty))$ with $1<p< \infty$ and we define $ U(x) = \int_{0}^{x} u(t)\, dt$,
then $U \in W^{1,p}_{\mathrm{loc}}((0,\infty))$ with weak derivative $U'(x) = u(x)$ for a.e.\ $x>0$. In this formulation, Hardy inequality takes the equivalent form
\begin{equation*}
    \int_{0}^{\infty} \frac{|U(x)|^{p}}{x^{p}} \, dx \leq  \left( \frac{p}{p-1} \right)^{p}  \int_{0}^{\infty} |U'(x)|^{p} \, dx,
\end{equation*}
and the constant is again optimal. The generalization of Hardy inequality to higher dimensions is given by
\begin{equation}\label{hardy inequality in higher dim}
    \int_{\mathbb{R}^d} \frac{|u(x)|^{p}}{|x|^{p}} \, dx \leq   \left| \frac{p}{p-d} \right|^{p} \int_{\mathbb{R}^d} |\nabla u(x)|^{p} \, dx,
\end{equation}
which holds for all  $u \in C^{\infty}_{c}(\mathbb{R}^d)$ if  $1 < p<d$, and for all  $u \in C^{\infty}_{c}(\mathbb{R}^d \setminus \{ 0 \})$ if  $p>d$. The constant  $\left| \frac{p}{p-d} \right|^{p}$  in  \eqref{hardy inequality in higher dim} is sharp (see  \cite[Lemma 2.1]{azorero1998} for the case  $1<p<d$ and  \cite[Theorem 1.2.5]{balinsky2015analysis} for  $p>d$), but never achieved. This inequality is also known as the ``Hardy-Sobolev" inequality. For the critical case $p=d$, the inequality  \eqref{hardy inequality in higher dim} does not hold with  $\left| \frac{p}{p-d} \right|^{p}$ replaced by any positive constant  $C$ on the space  $C_{c}^{\infty}(\mathbb{R}^{d} \backslash \{ 0 \} )$ (and hence also fails on  $C^{\infty}_{c}(\mathbb{R}^d)$). This can be seen directly by considering the sequence of functions
\begin{equation*}
    u_{n}(x) = \begin{cases}
        0, & 0<|x| \leq \frac{1}{n} \ \text{or} \ |x|\geq 4   \\
        1, & \frac{2}{n} \leq |x| \leq 3 
    \end{cases}
    \end{equation*}
such that  $|\nabla u_{n}(x)| \leq C n$ for all  $|x| \in (\frac{1}{n}, \frac{2}{n}) $ and  $|\nabla u_{n}(x)| \leq C $ for all  $|x| \geq \frac{2}{n}$, for some  $C>0$ . Let  $\Omega$ be a bounded domain in $\mathbb{R}^d$,   $d \geq 2$, with  $0 \in \Omega$, we have
\begin{equation*}
    \int_{\Omega} \frac{|u(x)|^{p}}{|x|^{p}} \, dx \leq \left( \frac{p}{d-p} \right)^{p} \int_{\Omega} |\nabla u(x)|^{p} \, dx,
\end{equation*}
holds for all  $ u \in C^{\infty}_{c}(\Omega)$ if  $1 < p < d$, and the constant  $\left( \frac{p}{d-p} \right)^{p}$ is sharp but never achieved. For a bounded domain  $\Omega$ in $\mathbb{R}^d$, Adimurthi, N. Chaudhuri and M. Ramaswamy in  \cite{adi2002} improved the Hardy-Sobolev inequality by adding a term characterized by a singular logarithmic weight and established the following result: for   $1<p  \leq d$, there exists a constant  $C =C(d,p,R) >0$ such that
\begin{equation*}
    \int_{\Omega} |\nabla u(x)|^{p} \, dx \geq \left( \frac{d-p}{p}  \right)^{p} \int_{\Omega} \frac{|u(x)|^{p}}{|x|^{p}} \, dx + C \int_{\Omega} \frac{|u(x)|^{p}}{|x|^{p}} \left( \ln \frac{R}{|x|} \right)^{- \gamma}  \, dx, \ \forall \  u \in C^{\infty}_{c}(\Omega),
\end{equation*}
where  $R \geq \sup_{\Omega} (|x|e^{\frac{2}{p}})$. The inequality holds if and only if  $\gamma \geq  2$ when  $1<p<d$, and  $\gamma \geq d$ when  $p=d$. Therefore, the result of \cite{adi2002} includes the critical case  $p = d$. A different type of logarithmic correction was later obtained in \cite{DELPINO2010}.

\smallskip

The boundary Hardy–Sobolev inequality for the local case (see \cite{lewis1988}) is stated as follows: let  $\Omega$ be a bounded domain with Lipschitz boundary in $\mathbb{R}^{d}$, and let  $1<p< \infty$. Then there exists a constant  $C=C(d,p, \Omega)>0$ such that
\begin{equation}\label{boundary hardy inequality}
      \int_{\Omega} \frac{|u(x)|^{p}}{\delta^{p}_{\Omega}(x)} \, dx \leq C \int_{\Omega} |\nabla u(x) |^{p} \, dx, \quad \forall \ u \in C^{\infty}_{c}(\Omega),
\end{equation}
where  $\delta_{\Omega}$ denotes the distance to the boundary of  $\Omega$, that is,  $\delta_{\Omega}(x) :=  \underset{y \in \partial \Omega}{\min} |x-y|$. The best constant  $\left( \frac{p}{p-1} \right)^{p}$ is obtained for convex domains in \cite{matskewich1997}. Using an appropriate approximation argument, one can extend the boundary Hardy–Sobolev inequality \eqref{boundary hardy inequality} to the classical Sobolev space  $W^{1,p}_{0}(\Omega)$, the closure of  $C^{\infty}_{c}(\Omega)$ with respect to the norm  $\left( \|u\|^{p}_{L^{p}(\Omega)} + \| \nabla u \|^{p}_{L^{p}(\Omega)} \right)^{\frac{1}{p}}$. The available literature is huge on Hardy-Sobolev inequality, and we refer to some of the related work in  \cite{adimurthisekar2006, tertikas2003,Brezis1997, Brezis19971,  lu2,lu1, e49ec619-af19-3029-915b-1bad4f2a8331, Nguyen2019, vazquez2000}. For the generalization of Hardy-Sobolev type inequality to Orlicz spaces, we refer to the work of  \cite{alberico2023, kaushik2022, salort2022}, and to the Heisenberg group, we refer to the work of  \cite{luan2008hardy, xiao2019hardy, yener2016weighted}.

\smallskip

We are interested in a similar type of problem in the setting of fractional Sobolev spaces, which have been considered with great interest in the last few decades due to their connection with probability theory (for example, see  \cite{Bogdan2003CensoredSP, claudia}) and partial differential equations (see  \cite{di2012hitchhikers} for introduction to the subject). 

\begin{defn}\emph{[Fractional Boundary Hardy inequality (FBHI)]}
Let $p>1$ and $s \in (0,1)$. Then there exists a constant 
$C = C(d,p,s,\Omega) > 0$ such that   

\begin{equation}\label{fractionalhardy}
    \int_{\Omega} \frac{|u(x)|^{p}}{\delta_{\Omega}^{sp}(x)} \, dx \leq C \int_{\Omega} \int_{\Omega} \frac{|u(x)-u(y)|^{p}}{|x-y|^{d+sp}} \, dx \, dy, \quad \forall \ u \in C^{\infty}_{c}(\Omega). 
\end{equation} 
\end{defn}

\begin{defn} The cases $sp < 1$, $sp = 1$, and $sp > 1$ are referred to as the subcritical, critical, and supercritical cases for the FBHI, respectively, when $\Omega$ is a bounded Lipschitz domain.
\end{defn}

Z.-Q. Chen and R. Song in  \cite{chen2003} studied FBHI for the case $p=2$. Afterwards, B. Dyda in  \cite{Dyda2004} investigated the FBHI  \eqref{fractionalhardy} for the subcritical and supercritical cases with arbitrary $p>0$. This work is the main motivation for the present article.
More precisely, he established the following theorem.

\begin{thm}[\emph{[Dyda, 2004]}]\label{dyda}
    Let  $\Omega$ be an open set in $\mathbb{R}^d, ~  d \geq 1$, and let   $p>0.$ The integral inequality \eqref{fractionalhardy} holds true in each of the following cases:
\begin{itemize}
\item[(T$1$)]  $\Omega$ is a bounded Lipschitz domain and $sp>1$;
\item[(T$2$)]  $\Omega$ is the complement of the bounded Lipschitz domain,  $sp \neq 1$ and  $sp \neq d$;
\item[(T$3$)]  $\Omega$ is a domain above the graph of a Lipschitz function  $\mathbb{R}^{d-1} \to \mathbb{R}$ and  $sp \neq 1$;
\item[(T$4$)]  $\Omega$ is the complement of a point and  $sp \neq d$.
\end{itemize}
Moreover, the inequality \eqref{fractionalhardy} fails in the critical cases, that is,  $sp=1$ for (T$1$) and  $sp=d$ for (T$2$) and (T$4$).
\end{thm}

The critical case  $sp=1$ for (T$2$) is not addressed in  \cite{Dyda2004}, but it is not difficult to show (see Subsection \ref{failure FBHI}, Appendix \ref{appen}) that  \eqref{fractionalhardy} also fails. 

\smallskip

The best constant for  \eqref{fractionalhardy} when  $\Omega = \mathbb{R}^{d}_{+}, ~ sp \neq 1$, and  $p=2$ is established in  \cite{bogdan2011}, and this result is further generalised  for any  $p\geq 1$ in  \cite{Frank2010}. For convex domains and  $p>1$, the best constant for \eqref{fractionalhardy} was obtained in \cite{loss2010}. The critical case  $sp=d$ for  $\Omega= \mathbb{R}^d \setminus \{ 0 \}$ was studied by H.-M. Nguyen and M. Squassina in \cite{squassina2018}. They proved (see  \cite[Theorem 3.1 with  $a=1$]{squassina2018}) an appropriate version of  \eqref{fractionalhardy} by adding a logarithmic weight on the left-hand side. In fact, their results are much more general, as they establish a full range of fractional Caffarelli–Kohn–Nirenberg inequalities. However, the issues of optimality of the logarithmic weight function and the best constant in their main theorem are not addressed. Adimurthi, P. Jana and P. Roy in  \cite{AdiPurbPro2023} studied  (T$1$)   \eqref{fractionalhardy} for the critical case  $sp=1$ in dimension  $1$. For Fractional Sobolev-Poincar\'e and fractional Hardy inequalities in unbounded John domains, we refer to the work of R. Hurri-Syrj{\"a}nen and A. V. V{\"a}h{\"a}kangas in  \cite{MR3343059}. For recent developments on fractional Hardy-type inequalities and other related inequalities, we refer to  \cite{brasco2018, csato, Cabre2022, luckn, lu2, lu2004, dyda2022sharp, lu1, lucartan, frank2008,  Frank2010, lu2009}. In particular, we refer to \cite{luckn}, where the Caffarelli-Kohn-Nirenberg inequality is studied by establishing appropriate identities, which in turn help in obtaining results on the existence of extremal functions as well.
   
\smallskip

We establish appropriate inequalities for all the remaining critical cases of Theorem  \ref{dyda}. In fact, we establish a much more general form of inequalities with appropriate logarithmic corrections. More precisely, our main result is as follows: suppose  $d \geq 1$,  $p>1$, and  $s \in (0,1)$. We denote by $W^{s,p}_{0}(\Omega)$, the closure of $C^{\infty}_{c}(\Omega)$ with respect to the norm $\|\cdot\|_{W^{s,p}(\Omega)}$ (see \eqref{defn frac sob space} for the definition of $W^{s,p}(\Omega)$). We define
\begin{equation*}
    p^{*}_{s}:= \frac{dp}{d-sp},  \quad \text{if} \ sp<d.
\end{equation*}

\begin{thm}\label{mainresult}
    Let  $\Omega$ be an open set in $\mathbb{R}^{d}$. Then the following fractional boundary Hardy-type inequality
    \begin{multline}\label{maineqn}
        \left( \int_{\Omega} \frac{|u(x)|^{\tau}}{\delta_{\Omega}^{\alpha}(x) \ln^{\beta}(\rho(x))} \,  dx \right)^{\frac{1}{\tau}} \leq C \left( \int_{\Omega} \int_{\Omega}  \frac{|u(x)-u(y)|^{p}}{|x-y|^{d+sp}} \, dx \, dy + \int_{\Omega} |u(x)|^{p} \, dx \right)^{\frac{1}{p}}, \\ \forall \  u \in W^{s,p}_{0}(\Omega),
    \end{multline}
    where $C=C(d,p,s, \tau, \Omega, R)>0$, holds true in each of the following cases:
    \begin{itemize}
        \item[(A$1$)] $\Omega$ is a bounded Lipschitz domain such that  $\delta_{\Omega}(x) < R$ for all  $x \in \Omega$, for some  $R>0$, and  $sp=1$. In this case,  $\rho(x) = \frac{2R}{\delta_{\Omega}(x)} , ~ W^{s,p}_{0}(\Omega) = W^{s,p}(\Omega) $ and \\
        (a) if  $d=1$ and  $\tau \geq p$, then  $\alpha=1$ and  $\beta= \tau$. \\
        (b) if  $d>1$ and  $\tau =p$, then  $\alpha=1$ and  $ \beta=p$. \\
        (c) if  $d>1$ and  $\tau \in (p,p + \frac{1}{d}]$, then  $\alpha = d+ (1-d)\frac{\tau}{p} $ and  $\beta= \tau -1$. \\
        (d) if  $d>1$ and  $\tau \in (p+ \frac{1}{d}, p^{*}_{s} ]$, then  $\alpha = d+ (1-d)\frac{\tau}{p}$ and  $\beta= dp + (1-d) \tau$. \\
        (e) if  $d \geq 1$ and  $\tau <p$, then  $\alpha=1$ and  $\beta=p$.
        \item[(A$2$)]  $\Omega$ is a domain above the graph of a Lipschitz function  $\gamma : \mathbb{R}^{d-1} \to \mathbb{R},  ~ d>1$ and  $sp=1$ such that  $\operatorname{supp} u \subset \Omega \backslash \overline{\Omega}_{R}$ where  $\Omega_{R}$ is a domain above the graph of a Lipschitz function  $\gamma_{R}: \mathbb{R}^{d-1} \to \mathbb{R}$ such that  $\gamma_{R}(x') = \gamma(x')+R$ for some  $R>0$. In this case,  $\rho(x) = \frac{2R}{\delta_{\Omega}(x)} $ and \\
        (a) if  $\tau =p$, then  $\alpha=1$ and  $ \beta=p$. \\
        (b) if  $\tau \in (p,p + \frac{1}{d}]$, then  $\alpha = d+ (1-d)\frac{\tau}{p} $ and  $\beta= \tau -1$. \\
        (c) if  $\tau \in (p+ \frac{1}{d}, p^{*}_{s} ]$, then  $\alpha = d+ (1-d)\frac{\tau}{p}$ and  $\beta= dp + (1-d) \tau$. \\ 
        (d) if  $\tau <p$, then  $\alpha=1$ and  $\beta=p$.
        \item [(A$3$)] $\Omega$ is a complement of a bounded Lipschitz domain such that  $\overline{\mathbb{R}^d \backslash \Omega} \subset B_{R}(0)$ for some  $R>0, ~ d>1$ and  $sp=d$. In this case,  $\rho(x) = \max \big\{ \frac{2R}{\delta_{\Omega}(x)} , \frac{2 \delta_{\Omega}(x)}{R} \big\},   ~ \tau = p, ~ \alpha= d$ and  $\beta = p$.
    \end{itemize}
 The weight function is optimal for $d = 1$ and $\tau \ge p$, and for $d > 1$ with $\tau = p$ in the critical case $sp = 1$, in the cases (A$1$)(a), (A$1$)(b), and (A$2$)(a), among the class of functions depending on $\delta_{\Omega}$.
\end{thm}

\smallskip

If we consider (A$1$) of Theorem \ref{mainresult}, then \eqref{maineqn} can be rewritten as
\begin{equation}\label{Theorem 2 Remark}
     \left( \int_{\Omega} \frac{|u(x) - (u)_{\Omega}|^{\tau}}{\delta_{\Omega}^{\alpha}(x) \ln^{\beta}(\rho(x))} \, dx \right)^{\frac{1}{\tau}} \leq C \left( \int_{\Omega} \int_{\Omega}  \frac{|u(x)-u(y)|^{p}}{|x-y|^{d+sp}} \, dx \, dy \right)^{\frac{1}{p}} ,  \quad \forall \ u \in W^{s,p}(\Omega) ,
\end{equation}
where  $(u)_{\Omega}$ denotes the average of  $u$ over  $\Omega$, given by
\begin{equation*}
    (u)_{\Omega} = \frac{1}{|\Omega|} \int_{\Omega} u(y) \, dy ,
\end{equation*}
and  $|\Omega|$ denotes the Lebesgue measure of  $\Omega$. The above Hardy-type inequality is obtained by applying the fractional Poincaré inequality \eqref{poincare} in Theorem \ref{mainresult}.  For a general bounded Lipschitz domain $\Omega \subset \mathbb{R}^{d}$ with $sp=1$, we established in Lemma \ref{lemma proof of opt gen domain} the optimality of the weight function in Theorem  \ref{mainresult}, among the class of functions depending on $\delta_{\Omega}$. Initially, we showed that  $W^{s,p}_{0} (\Omega) = W^{s,p}(\Omega)$, which is proved in Lemma  \ref{lemma on Wsp} using a result due to B. Dyda and M. Kijaczko in  \cite[Lemma 13]{dyda2022}. In fact, for any bounded Lipschitz domain  $\Omega$, G. Leoni \cite[Theorem 6.78]{leonibook} proves that   $W^{s,p}_{0}(\Omega) = W^{s,p}(\Omega)$ when  $sp \leq 1$. Using this fact, the trace operator  $T : W^{s,p}_{0} (\Omega) \to L^{p} (\partial \Omega )$, defined by  $T(u) = u|_{\partial \Omega}$, is not well defined for $sp \leq 1$. When  $\Omega= B_{1}(0)$ is a unit ball in  $\mathbb{R}^{d}$, and we consider the radial case, one of the key ingredients  to prove the optimality of the weight function for  $B_{1}(0)$ in the previous theorem is to find out the  approximating sequence (with precise rates of convergence) in  $W^{s,p}_{0}$ for  $sp=1$, as established in Theorem  \ref{density theorem}.

\smallskip 

(A$1$) of Theorem \ref{mainresult} deals with the critical case $sp=1$ for bounded Lipschitz domains. In this article, we also generalize the FBHI \eqref{fractionalhardy} for bounded Lipschitz domains when $sp \neq 1$. This is established in the next two theorems. For the case  $sp \neq 1$ and $\tau=p$, the results are already available in  \cite{Dyda2004}. We extend their results to cover appropriate values of $\tau$. In particular, we will prove the following theorems: suppose  $d \geq 1, ~ p>1$, and  $s \in (0,1)$.

\begin{thm}\label{theorem sp<1}
Let $sp < 1$ and let $\Omega$ be a bounded Lipschitz domain in $\mathbb{R}^{d}$. Then the following fractional boundary Hardy-type inequality
   \begin{equation}\label{generalized1}
        \left(  \int_{\Omega} \frac{|u(x)|^{\tau}}{\delta^{\alpha}_{\Omega}(x) } \,  dx \right)^{\frac{1}{\tau}} \leq  C \left( \int_{\Omega} \int_{\Omega}  \frac{|u(x)-u(y)|^{p}}{|x-y|^{d+sp}} \, dx \, dy + \int_{\Omega} |u(x)|^{p} \, dx \right)^{\frac{1}{p}}, \  \forall  \ u \in W^{s,p}(\Omega),
   \end{equation}
 where $C=C(d,p,s, \tau,\Omega)>0$,  holds true in each of the following cases:
   \begin{itemize}
       \item[(B$1$)]  if  $\tau \in [p, p^{*}_{s}]$, then  $\alpha= d+ (sp-d)\frac{\tau}{p}$.
       \item[(B$2$)] if  $\tau <p$, then  $\alpha = sp$.
   \end{itemize}
\end{thm}

\smallskip

 The next theorem generalizes and provides a different proof of Dyda's result in \cite{Dyda2004} for bounded Lipschitz domains.

\begin{thm}\label{theorem sp>1}
   Let $sp > 1$ and let $\Omega$ be a bounded Lipschitz domain in $\mathbb{R}^{d}$. Then the following fractional boundary Hardy-type inequality
   \begin{equation}\label{generalized2}
        \left(  \int_{\Omega} \frac{|u(x)|^{\tau}}{\delta^{\alpha}_{\Omega}(x) } \, dx \right)^{\frac{1}{\tau}} \leq  C 
 \left( \int_{\Omega} \int_{\Omega}  \frac{|u(x)-u(y)|^{p}}{|x-y|^{d+sp}} \,  dx \, dy \right)^{\frac{1}{p}}, \quad \forall \ u \in W^{s,p}_{0}(\Omega),
   \end{equation}
  where $C=C(d,p,s, \tau, \Omega)>0$, holds true in each of the following cases:
   \begin{itemize}
       \item[(C$1$)]  if  $sp<d$, then  $\tau \in [p, p^{*}_{s}] $ and  $\alpha= d+ (sp-d)\frac{\tau}{p}$.
       \item[(C$2$)] if  $sp=d$, then  $\tau \geq p$ and  $\alpha = sp$.
       \item[(C$3$)] if  $sp>d$, then $\tau =p$ and  $\alpha=sp$.
   \end{itemize}
\end{thm}

Theorem \ref{theorem sp<1} and Theorem \ref{theorem sp>1} generalize the fractional boundary Hardy inequality for suitable $\tau$ depending on $s$ and $p$. However, the generalization is straightforward for the case $sp < d$. It suffices to establish the case $\tau = p$. For suitable $\tau$, this follows from interpolation and the fractional Sobolev inequality, i.e., H$\ddot{\text{o}}$lder's inequality, the fractional boundary Hardy-type inequality for the case $\tau = p$, and the fractional Sobolev inequality, which is done in Subsection \ref{proof of theorem sp>1} in the proof of Theorem \ref{theorem sp>1}.

\smallskip

\begin{remark}
    \normalfont For a bounded Lipschitz domain  $\Omega$ and  $\alpha>0$, there exists a constant  $C= C(d,p,s,\tau,\Omega)>0$ such that  $|u(x)|^{\tau} \leq C \frac{|u(x)|^{\tau}}{\delta^{\alpha}_{\Omega}(x)}$ for all  $x \in \Omega$ and for all  $ u \in W^{s,p}_{0}(\Omega)$. Therefore, using this and Theorem  \ref{theorem sp>1}, we obtain
\begin{equation*}
    \left( \int_{\Omega} |u(x)|^{\tau} \, dx \right)^{\frac{1}{\tau}} \leq C \left( \int_{\Omega} \frac{|u(x)-u(y)|^{p}}{|x-y|^{d+sp}} \, dx \, dy \right)^{\frac{1}{p}}, \quad \forall \  u \in W^{s,p}_{0}(\Omega),
\end{equation*}
 for all  $\tau \geq p$ when  $1<sp=d$, and  $\tau \in [p, p^{*}_{s} ] $ when  $1<sp<d$.
\end{remark}

\smallskip

In this article, we also consider a weighted fractional boundary Hardy-type inequality for the critical case in  $\mathbb{R}^{d}_{+}$, which was not addressed by Dyda and Kijaczko in \cite{dyda2022sharp}. They studied the inequality
\begin{equation*}
 \int_{\mathbb{R}^d_{+}} \frac{|u(x)|^{p}}{x^{sp - \beta_{1}- \beta_{2}}_{d}} \, dx \leq C \int_{\mathbb{R}^{d}_{+}} \int_{\mathbb{R}^{d}_{+}} \frac{|u(x)-u(y)|^{p}}{|x-y|^{d+sp}} x^{\beta_{1}}_{d} y^{\beta_{2}}_{d} \, dx \, dy, \quad \forall \ u \in C_{c}(\mathbb{R}^{d}_{+}),
\end{equation*}
and proved that this inequality holds whenever  $\beta_{1},  ~ \beta_{2}, ~ \beta_{1}+ \beta_{2} \in (-1, sp)$ and  $1+ \beta_{1}+ \beta_{2} \neq sp$. We will prove (see Theorem \ref{upperhalfplane} in Section \ref{weighted fractional boundary Hardy-type}) the corresponding version of the above inequality for the critical case  $1+ \beta_{1} + \beta_{2} = sp$ with an additional logarithmic weight. Theorem \ref{upperhalfplane} generalizes condition (A$2$) of Theorem \ref{mainresult} in the special case $\Omega= \mathbb{R}^{d}_{+}$. Moreover, upon substituting  $\beta_{1}= \beta_{2} = 0$ in Theorem \ref{upperhalfplane}, one recovers (A$2$) of Theorem \ref{mainresult} for $\Omega = \mathbb{R}^{d}_{+}$.

\smallskip

The article is organized as follows. In Section \ref{preliminaries}, we present preliminary lemmas and notation that will be used throughout the paper. In Section \ref{The flat boundary case}, we establish a fractional boundary Hardy-type inequality for the critical case  $sp=1$ in domains with flat boundaries, and for the case  $sp=d$ in the complement of a ball in $\mathbb{R}^d$. Section \ref{proof of main result} contains the proofs of Theorem \ref{mainresult} (without the optimality of the weight function), Theorem \ref{theorem sp<1}, and Theorem \ref{theorem sp>1}. In Section \ref{optimality of weight function}, we prove the approximation theorem and establish the optimality of the weight function depending on $\delta_{\Omega}$ in the critical case $sp=1$. Finally, in Section \ref{weighted fractional boundary Hardy-type}, we prove the weighted fractional boundary Hardy-type inequality for the critical case.

%-----------------NOTATIONS AND PRELIMINARIES------------------------------

\section{Notation and Preliminaries}\label{preliminaries}

In this section, we introduce the notation, definitions, and preliminary results that will be used throughout the article. Throughout this article we shall use the following notation: for a measurable set  $\Omega \subset \mathbb{R}^d, ~   (u)_{\Omega}$ will denote the average of the function  $u$ over  $\Omega$, i.e., 
\begin{equation*}
    (u)_{\Omega} :=  \frac{1}{|\Omega|} \int_{\Omega} u(x) \,  dx = \fint_{\Omega} u(x) \, dx. 
\end{equation*}
Here,  $|\Omega|$ denotes the Lebesgue measure of  $\Omega$. We denote by  $B_{r}(x)$, an open ball of radius  $r$ and center  $x$ in  $\mathbb{R}^d$, and  $\mathbb{S}^{d-1}$ represents the boundary of the open ball of radius  $1$ with center  $0$ in  $\mathbb{R}^{d}$. The parameter  $s$ will always be understood to lie in $(0,1)$. Any point  $x \in \mathbb{R}^d$ is represented as  $x=(x',x_{d})$, where  $x'=(x_{1}, \dots, x_{d-1}) \in \mathbb{R}^{d-1}$. For any  $f,  ~ g: \Omega ( \subset \mathbb{R}^d) \to \mathbb{R}$, we denote  $f \sim g$ if there exist constants  $C_{1}, ~ C_{2}>0$ such that  $C_{1}g(x) \leq f(x) \leq C_{2} g(x)$ for all  $x \in \Omega$. Also,  $C>0$ will denote a generic constant that may change from line to line. 

\smallskip

Let  $\Omega$ be an open set in $\mathbb{R}^{d}$, and let  $s \in (0,1)$. For any  $p \in [1, \infty)$, define the fractional Sobolev space
\begin{equation}\label{defn frac sob space}
W^{s,p}(\Omega) := \left\{    u \in L^{p}(\Omega) : \int_{\Omega} \int_{\Omega}  \frac{|u(x)-u(y)|^{p}}{|x-y|^{d+sp}} \, dx \, dy < \infty \right\},
\end{equation}
endowed with the norm
\begin{equation*}
    \|u\|_{W^{s,p}(\Omega)} := \left( \|u\|^{p}_{L^{p}(\Omega)} + [u]^{p}_{W^{s,p}(\Omega)} \right)^{\frac{1}{p}},
\end{equation*}
where 
\begin{equation*}
    [u]_{W^{s,p}(\Omega)} := \left( \int_{\Omega} \int_{\Omega}  \frac{|u(x)-u(y)|^{p}}{|x-y|^{d+sp}} \, dx \, dy \right)^{\frac{1}{p}} 
\end{equation*}
is the so-called Gagliardo seminorm. Denote by  $W^{s,p}_{0}(\Omega)$, the closure of  $C^{\infty}_{c}(\Omega)$ with respect to the norm  $\| \cdot \|_{W^{s,p}(\Omega)}$. For  $\beta_{1}, ~ \beta_{2} \in \mathbb{R}$,  define the weighted Gagliardo seminorm using the distance function as
\begin{equation*}
     [u]_{W^{s,p, \beta_{1}, \beta_{2}}(\Omega)} := \left( \int_{\Omega} \int_{\Omega}  \frac{|u(x)-u(y)|^{p}}{|x-y|^{d+sp}} \delta^{\beta_{1}}_{\Omega}(x) \delta^{\beta_{2}} _{\Omega}(y) \, dx \, dy \right)^{\frac{1}{p}},
\end{equation*}
and define the weighted fractional Sobolev space
\begin{equation}\label{weighted fractional sobolev norm}
    W^{s,p, \beta_{1}, \beta_{2}}(\Omega) = \left\{ u \in L^{p}(\Omega) : [u]_{W^{s,p, \beta_{1}, \beta_{2}}(\Omega)} < \infty \right\},
\end{equation}
with the norm as 
\begin{equation*}
    \|u\|_{W^{s,p, \beta_{1}, \beta_{2}}(\Omega)} := \left( \| u\|^{p}_{L^{p}(\Omega)} + [u]^{p}_{W^{s,p, \beta_{1}, \beta_{2}}(\Omega)} \right)^{\frac{1}{p}}.
\end{equation*}

\smallskip

One has, for any  $a_{1},  \dots , a_{m} \in \mathbb{R}$ and  $\gamma \geq1$,
\begin{equation}\label{sumineq}
    \sum_{\ell=1}^{m} |a_{\ell}|^{\gamma} \leq  \left( \sum_{\ell=1}^{m} |a_{\ell}| \right)^{\gamma} .
\end{equation}
\smallskip

{\bf A bounded domain with Lipschitz boundary}: Let  $\Omega$ be a bounded Lipschitz domain. Then for each  $x \in \partial \Omega$ there exists  $r_{x}>0$, an isometry  $T_{x}$ of  $\mathbb{R}^d$, and a Lipschitz function  $\phi_{x} : \mathbb{R}^{d-1} \to \mathbb{R}$ such that
\begin{equation*}
    T_{x}(\Omega) \cap B_{r_{x}}(T_{x}(x)) = \{ \xi : \xi_{d} > \phi_{x}(\xi') \} \cap B_{r_{x}}(T_{x}(x))  . 
\end{equation*}

\smallskip

{\bf Fractional Poincar\'e\ inequality}  \cite[Theorem 3.9]{edmunds2022}: Let $p \geq 1$ and $\Omega$ be a bounded open set in $\mathbb{R}^{d}$. Then there exists a constant  $C = C(d,p,s, \Omega)>0$ such that
\begin{equation}\label{poincare}
    \int_{\Omega} |u(x)-(u)_{\Omega}|^{p} \, dx \leq C [u]^{p}_{W^{s,p}(\Omega)},  \quad \forall  \ u \in W^{s,p}(\Omega) .
\end{equation} 
\smallskip

The next lemma provides the exact dependence of the constant in Poincar\'e inequality/ Sobolev inequality on the parameter  $\lambda$, where  $\lambda $ is the scaling parameter of domain  $\Omega$ in all directions.  Since we have to decompose our domain  $\Omega$ (in all the cases) as a disjoint union of small cubes of different sizes, the next lemma is one of the key ingredients in the proof of our main results.
\begin{lemma}[\emph{Fractional Sobolev inequality for $sp \leq d$}]\label{sobolev} 
   Let  $\Omega$ be a bounded Lipschitz domain in  $\mathbb{R}^{d}, ~ d \geq 1$. Let  $p > 1$ and  $s \in (0,1)$ be such that  $sp \leq d$. Define  $\Omega_{\lambda}: = \left\{ \lambda x : ~  x \in \Omega \right\}$ for  $\lambda>0$. There exists a constant  $C=C(d,p,s, \tau, \Omega)>0$ such that, for any  $u \in W^{s,p}(\Omega_{\lambda})$, we have
    \begin{equation}
         \left( \fint_{\Omega_{\lambda}} |u(x)-(u)_{\Omega_{\lambda}}|^{\tau } \, dx \right)^{\frac{1}{\tau}}  \leq C \left( \lambda^{sp-d} [u]^{p}_{W^{s,p}(\Omega_{\lambda})} \right)^{\frac{1}{p}}  ,
    \end{equation}
    for any  $\tau \in [p, p^{*}_{s}]$ when  $sp<d$, and  $\tau \geq p$ when  $sp=d$.
\end{lemma} 
\begin{proof}
  Let  $\Omega$ be a bounded Lipschitz domain in $\mathbb{R}^{d}$. Then from Theorem  $6.7$ of  \cite{di2012hitchhikers}, for any  $\tau \in [p,p^{*}_{s}]$ if  $sp<d$, and from Theorem  $6.10$ of  \cite{di2012hitchhikers}, for any  $\tau \geq p$ if  $sp=d$, we have
\begin{equation}\label{Sobineqq}
      \|u\|_{L^{\tau}(\Omega)} \leq C \|u\|_{W^{s,p}(\Omega)},
  \end{equation}
   where  $C=C(d,p,s, \tau, \Omega)$ is a constant. Applying  \eqref{Sobineqq} with  $u-(u)_{\Omega}$ and using  \eqref{poincare}, we obtain
    \begin{equation*}
        \|u-(u)_{\Omega}\|_{L^{\tau}(\Omega)} \leq C [u]_{W^{s,p}(\Omega)} .
    \end{equation*}
    Therefore,
    \begin{equation}\label{ineqlemma}
       \left(  \fint_{\Omega} |u(x)-(u)_{\Omega}|^{\tau } \, dx \right)^{\frac{1}{\tau}}  \leq C  [u]_{W^{s,p}(\Omega)}  . 
    \end{equation}
    Let us apply the above inequality to  $u(\lambda x)$ instead of  $u(x)$. This gives
    \begin{equation*}
        \left(  \fint_{\Omega}  \Big|u(\lambda x)-\fint_{\Omega} u(\lambda x) \, dx \Big|^{\tau } \, dx \right)^{\frac{1}{\tau}}  \leq C \left( \int_{\Omega} \int_{\Omega} \frac{|u(\lambda x) - u(\lambda y)|^{p}}{|x-y|^{d+sp}} \, dx \, dy  \right)^{\frac{1}{p}}  . 
    \end{equation*}
    Using the fact
    \begin{equation*}
        \fint_{\Omega} u(\lambda x) \, dx = \fint_{\Omega_{\lambda}} u(x) \, dx,
    \end{equation*}
    we have
    \begin{equation*}
       \left(  \fint_{\Omega} |u(\lambda x)-(u)_{\Omega_{\lambda}}|^{\tau } \, dx \right)^{\frac{1}{\tau}}  \leq C \left( \int_{\Omega} \int_{\Omega} \frac{|u(\lambda x) - u(\lambda y)|^{p}}{|x-y|^{d+sp}} \, dx \, dy  \right)^{\frac{1}{p}}  . 
    \end{equation*}
    By changing variables \(X = \lambda x\) and \(Y = \lambda y\), we obtain
    \begin{equation*}
         \left( \fint_{\Omega_{\lambda}} |u(x)-(u)_{\Omega_{\lambda}}|^{\tau } \, dx \right)^{\frac{1}{\tau}}  \leq C \left( \lambda^{sp-d} [u]^{p}_{W^{s,p}(\Omega_{\lambda})} \right)^{\frac{1}{p}}  . 
    \end{equation*}
    This finishes the proof of the lemma.
\end{proof}

\smallskip

For the study of fractional Poincar\'e inequalities when the domains are scaled in a certain direction (but not in all directions), we refer to the following lemma, which establishes a connection to the average of $u$ over two disjoint sets. This technical step is crucial in the proof of the main results.

\begin{lemma}\label{avg}
    Let $E$ and $F$ be disjoint sets in $\mathbb{R}^{d}$. Then for any $\tau > 1$, there exists a constant $C > 0$ such that
    \begin{equation}
        |(u)_{E} - (u)_{F}|^{\tau} \leq C \frac{|E \cup F|}{\min \{ |E|, |F| \} }  \fint_{E \cup F} |u(x)-(u)_{E \cup F}|^{\tau} \, dx  . 
    \end{equation}
\end{lemma}
\begin{proof}
Let us consider  $|(u)_{E}-(u)_{F}|^{\tau}$, we have
\begin{align*}
    |(u)_{E}-(u)_{F}|^{\tau} & = |(u)_{E} - (u)_{E \cup F} - (u)_{F} + (u)_{E \cup F}|^{\tau} \\ &
        \leq C|(u)_{E} - (u)_{E \cup F}|^{\tau} +  C | (u)_{F} + (u)_{E \cup F}|^{\tau} \\ &
        = C \left| \fint_{E} \left\{ u(x) - (u)_{E \cup F} \right\}  \, dx  \right|^{\tau} + C \left| \fint_{F} \left\{ u(x) - (u)_{E \cup F} \right\} \, dx \right|^{\tau}.
\end{align*}
    By using H$\ddot{\text{o}}$lder's inequality with  $ \frac{1}{\tau} +  \frac{1}{\tau'} = 1$, we have
    \begin{align*}
         |(u)_{E}-(u)_{F}|^{\tau} & \leq  C \fint_{E} |u(x) - (u)_{E \cup F} |^{\tau} \,  dx + C \fint_{F} | u(x) - (u)_{E \cup F} |^{\tau} \,  dx \\ &
       \leq  \frac{C}{\min \{ |E|, |F| \} } \int_{E \cup F} |u(x) - (u)_{E \cup F} |^{\tau} \,  dx \\ & 
       = C \frac{|E \cup F|}{\min \{ |E|, |F| \}} \fint_{E \cup F} |u(x) - (u)_{E \cup F} |^{\tau}  \, dx .
    \end{align*}
    This completes the proof of the lemma.
\end{proof}

\smallskip

The next lemma establishes an inequality when any function  $u \in W^{s,p}(\Omega)$ is multiplied by a test function. This lemma plays a crucial role in establishing our main results.
\begin{lemma}\label{testfunc}
    Let  $\Omega$ be an open set in  $\mathbb{R}^d$. Let us consider  $u \in W^{s,p}(\Omega)$ and  $\xi \in C^{0,1}(\Omega), ~ 0 \leq \xi \leq 1$. Then  $\xi u \in W^{s,p}(\Omega)$ and for some constant  $C=C(d,p,s,\Omega)>0$, 
    \begin{equation*}
        \| \xi u\|_{W^{s,p}(\Omega)} \leq C \|u\|_{W^{s,p}(\Omega)}.
    \end{equation*}
\end{lemma}
\begin{proof}
    See  \cite[Lemma 5.3]{di2012hitchhikers} for the proof.
\end{proof}

\smallskip

The next lemma establishes compactness result involving the fractional Sobolev spaces $W^{s,p}(\Omega)$ in bounded Lipschitz domains. This lemma is useful in proving Theorem  \ref{theorem sp>1}.

\begin{lemma}\label{compactness}
    Let  $s \in (0,1), ~ p \geq 1, ~ \tau \in [1,p], ~ \Omega$ be a bounded Lipschitz domain in $\mathbb{R}^{d}$ and  $\mathcal{F}$ be a bounded subset of  $L^{p}(\Omega)$. Suppose that
    \begin{equation*}
        \sup_{u \in \mathcal{F}} \int_{\Omega} \int_{\Omega} \frac{|u(x)-u(y)|^{p}}{|x-y|^{d+sp}} \, dx \, dy < + \infty.
    \end{equation*}
    Then  $\mathcal{F}$ is pre-compact in  $L^{\tau}(\Omega)$.
\end{lemma}
\begin{proof}
    See  \cite[Theorem 7.1]{di2012hitchhikers} for the proof.
\end{proof}

\smallskip

The next lemma establishes a basic inequality. A similar inequality is also used in \cite{squassina2018} (see \cite[Lemma 3.2]{squassina2018}), but our version is less restrictive.

\begin{lemma}\label{estimate}
    Let  $\tau >1$ and $c>1$. Then for all  $a, ~ b \in \mathbb{R}$, we have 
    \begin{equation}
        (|a| + |b|)^{\tau} \leq c|a|^{\tau} + (1-c^{\frac{-1}{\tau -1}})^{1-\tau} |b|^{\tau} .
    \end{equation}
\end{lemma}
\begin{proof}
    Let  $f(x) = (1+x)^{\tau}- cx^{\tau}$ for all  $x \geq 0$. Then  $f(0) =  1$ and  $f(x) \to - \infty$ as  $x \to \infty$. Let  $x_{0}$ be the point of maxima. Then  $f'(x_{0}) = 0$, i.e.,
    \begin{equation*}
        \tau (1+x_{0})^{\tau -1} - c \tau x^{\tau -1}_{0} = 0.
    \end{equation*}
    Therefore, we have  $(1+x_{0})^{\tau}= c^{ \frac{\tau}{\tau-1}} x^{\tau}_{0}$ and  $x_{0} = \frac{1}{c^{\frac{1}{\tau-1}}-1}$. As  $x_{0}$ is the point of maxima, we have  $ f(x) \leq f(x_{0})$ for all  $x \geq 0$. Therefore,
    \begin{equation*}
        (1+x)^{\tau} - cx^{\tau} \leq (1+x_{0})^{\tau} - cx^{\tau}_{0} 
        = (c^{ \frac{\tau}{\tau -1}}- c) x^{\tau}_{0} 
        = \frac{(c^{\frac{\tau}{\tau -1}} - c)} {(c^{ \frac{1}{\tau-1}} -1)^{\tau}} 
        = (1-c^{\frac{-1}{\tau -1}})^{1-\tau}.
    \end{equation*}
    Assume  $|b| \neq 0$ and take  $x= \frac{|a|}{|b|}$. Then, the above inequality proves the lemma.
\end{proof}

\smallskip

The next lemma establishes a basic inequality for all  $n \in \mathbb{N}$. This lemma is helpful in proving Theorem  \ref{mainresult}.

\begin{lemma}\label{large n ineq}
Let  $\tau>1$. Then for all  $n \in \mathbb{N}$, we have
     \begin{equation*}
       \frac{1}{(n)^{\tau-1}} - \frac{1}{ \left(n + \frac{1}{2} \right)^{\tau-1}} \geq \frac{1}{2} ( \tau-1) \left( \frac{2}{3} \right)^{\tau} \frac{1}{n^{\tau}}.
    \end{equation*}
\end{lemma}
\begin{proof}
 Let  $f(x) = (n+x)^{1- \tau}$ for $x \in \left[0, \frac{1}{2} \right]$ and for some  $n \in \mathbb{N}$. By the mean value theorem, there exists $\gamma \in \left(0, \tfrac{1}{2}\right)$ such that
    \begin{equation*}
        f'(\gamma) = \frac{f(\frac{1}{2})- f(0)}{\frac{1}{2}} = \frac{\left( n+ \frac{1}{2} \right)^{1- \tau} - (n)^{1-\tau}}{\frac{1}{2}}.
    \end{equation*}
   Therefore, we have
    \begin{equation*}
      \frac{1}{(n)^{\tau-1}} - \frac{1}{ \left(n + \frac{1}{2} \right)^{\tau-1}} = - \frac{1}{2} f'(\gamma) = \frac{1}{2} ( \tau-1) (n + \gamma)^{- \tau}  =  \frac{1}{2} ( \tau-1) n^{- \tau} \left( 1+ \frac{\gamma}{n} \right)^{- \tau}.
    \end{equation*}
    For any $\gamma \in \left(0, \tfrac{1}{2}\right)$ and $n \in \mathbb{N}$, we have $  1+ \frac{\gamma}{n} \leq 1+\frac{1}{2} = \frac{3}{2}$. Therefore,
    \begin{equation*}
      \left( 1+ \frac{\gamma}{n} \right)^{- \tau} \geq \left( \frac{2}{3} \right)^{\tau}, \quad \forall \  n \in \mathbb{N}.
    \end{equation*}
    Combining the above two estimates, we obtain for all $n \in \mathbb{N}$,
    \begin{equation*}
         \frac{1}{(n)^{\tau-1}} - \frac{1}{ \left(n + \frac{1}{2} \right)^{\tau-1}} \geq \frac{1}{2} ( \tau-1) \left( \frac{2}{3} \right)^{\tau} \frac{1}{n^{\tau}}.
    \end{equation*}
    This proves the lemma.
\end{proof}

%-----------------FLAT BOUNDARY CASE-------------------

\section{The flat boundary case}\label{The flat boundary case}
In this section, we establish our main results in the flat boundary case, where the boundary is taken to be $\{ x=(x',x_d) \in \mathbb{R}^{d} : x_{d} =0 \}$, and the functions are supported on $\Omega_{n}$. The domain $\Omega_{n}$ is defined below. For general bounded Lipschitz domains, we will employ a patching technique, which is carried out in the next section. Throughout this section, we will use the following definition for the exponent $\alpha$.

\begin{defn}\label{alpha defn}
Let $s \in (0,1)$, $p > 1$ with $sp \leq d$, and let $\tau$ satisfy $\tau \ge p$ if $sp=d$, and  $\tau \in [p, p^{*}_{s}]$ if $sp<d$. We define 
\begin{equation*}
    \alpha := d + (sp - d)\frac{\tau}{p}.
\end{equation*}
\end{defn}

Define the domain
\begin{equation*}
    \Omega_{n} := (-n,n)^{d-1} \times (0,1)
\end{equation*}
for some  $n \in \mathbb{N}$, where a generic point is written as $x=(x',x_{d}) \in \Omega_{n} $. For each  $k \leq -1$, set
\begin{equation*}
    A_{k} := \{ (x',x_{d}) : \  x' \in (-n,n)^{d-1}, \ 2^{k} \leq x_{d} < 2^{k+1} \} .
\end{equation*}
Then, we have
\begin{equation*}
    \Omega_{n}= \bigcup_{k = - \infty}^{-1} A_{k}  .
\end{equation*}
Again, we  further divide each  $A_{k}$ into disjoint cubes each of side length  $2^k$ (say  $A^{i}_{k}$) (see Figure \ref{fig:figure1}). Then
\begin{equation*}
    A_{k} = \bigcup_{i=1}^{ \sigma_{k}} A^{i}_{k}  ,
\end{equation*}
where  $\sigma_{k} :=  2^{(-k+1)(d-1)} n^{d-1}$. 

\begin{figure}[!ht]
\centering
\begin{tikzpicture}
    % Labels
    \node at (-3.5,1.6) {$A_{k+1}$};
    \node at (-3.4,0.8) {$A_{k}$};
    \node at (-2.3,0.8) {$A^{i}_{k}$};
    
    % Main rectangle
    \draw (-3,2) -- (3,2);
    \draw (-3,2) -- (-3,0);
    \draw (3,2) -- (3,0);
    
    % Dashed lines
    \draw[dashed] (-3,1.2) -- (3,1.2);
    \draw[dashed] (-3,0.5) -- (3,0.5);
    \draw[dashed] (-1.5,1.2) -- (-1.5,0);
    \draw[dashed] (1.5,1.2) -- (1.5,0);
    \draw[dashed] (-0.75,0.5) -- (-0.75,0);
    \draw[dashed] (0.75,0.5) -- (0.75,0);
  \draw[dashed] (-2.25,0.5) -- (-2.25,0);
  \draw[dashed] (2.25,0.5) -- (2.25,0);
    % Axes
    \draw[->] (-3.5,0) -- (3.5,0) node[right] {$\mathbb{R}^{d-1}$};
    \draw[->] (0,0) -- (0,2.3) node[above] {$\mathbb{R}$};
\end{tikzpicture}
\caption{The domain $\Omega_{n}$.}
\label{fig:figure1}
\end{figure}

\smallskip

The next lemma establishes an inequality that holds true for each  $A^{i}_{k}$. This lemma is helpful in proving Lemma  \ref{case 1 flat case} to Lemma  \ref{flat case sp>1 a}.

\begin{lemma}\label{sumineqlemma}
For any  $A^{i}_{k}$, there exists a constant  $C=C(d,p,s, \tau)>0$ such that
   \begin{equation}
 \int_{A^{i}_{k}} \frac{|u(x)|^{\tau}}{x^{\alpha}_{d}} \, dx \leq C [u]^{\tau}_{W^{s,p}(A^{i}_{k})} + C 2^{k(d-\alpha)} |(u)_{A^{i}_{k}}|^{\tau}  .
\end{equation} 
\end{lemma}
\begin{proof}
    Fix any   $A^{i}_{k}$. Then  $A^{i}_{k}$ is a translation of  $(2^{k}, 2^{k+1})^{d}$. Applying Lemma  \ref{sobolev} with  $\Omega = (1,2)^d$ and  $ \lambda = 2^k$, and using translation invariance, we have
\begin{equation*}
    \fint_{A^{i}_{k}} |u(x)-(u)_{A^{i}_{k}}|^{\tau} \,  dx \leq C 2^{k(sp-d) \frac{\tau}{p}} [u]^{\tau}_{W^{s,p}(A^{i}_{k})}  ,
\end{equation*}
where  $C= C(d,p,s, \tau)$ is a constant. Let  $x = (x',x_d) \in A^{i}_{k}$.  Then  $x_d \geq 2^k$, which implies  $ \frac{1}{x_{d}} \leq \frac{1}{2^{k}}$. Therefore, we have
\begin{align*}
    \int_{A^{i}_{k}} \frac{|u(x)|^{\tau}}{x^{\alpha}_{d}} \,  dx & \leq \frac{1}{2^{k \alpha}} \int_{A^{i}_{k}} |u(x)-(u)_{A^{i}_{k}} + (u)_{A^{i}_{k}}|^{\tau}  \, dx \\ &
    \leq \frac{2^{\tau-1}}{2^{k \alpha}} \int_{A^{i}_{k}} |u(x)-(u)_{A^{i}_{k}}|^{\tau}  \, dx + \frac{2^{\tau-1}}{2^{k \alpha}} \int_{A^{i}_{k}} |(u)_{A^{i}_{k}}|^{\tau} \, dx.
    \end{align*}
  Now, using the previous inequality, together with the definition of  $\alpha$ (see Definition \ref{alpha defn}), we obtain
    \begin{align*}
    \int_{A^{i}_{k}} \frac{|u(x)|^{\tau}}{x^{\alpha}_{d}}  \, dx & \leq C \frac{|A^{i}_{k}|}{2^{k \alpha}} \fint_{A^{i}_{k}} |u(x)-(u)_{A^{i}_{k}}|^{\tau}  \, dx + C \frac{|A^{i}_{k}|}{2^{k \alpha}} |(u)_{A^{i}_{k}}|^{\tau}  \\ &
    \leq C \frac{2^{kd}}{2^{k \alpha}} 2^{k(sp-d) \frac{\tau}{p}} [u]^{\tau}_{W^{s,p}(A^{i}_{k})} +  C 2^{k(d-\alpha)} |(u)_{A^{i}_{k}}|^{\tau}  
   \\ & = C [u]^{\tau}_{W^{s,p}(A^{i}_{k})} + C 2^{k(d-\alpha)} |(u)_{A^{i}_{k}}|^{\tau}  ,
\end{align*}   
where  $C$ is a constant depending on  $d, ~ p, ~ s$ and  $\tau$. This proves the lemma.
\end{proof}
\smallskip

The following lemma establishes a connection between the averages of two disjoint cubes  $A^{j}_{k+1}$ and  $A^{i}_{k}$, where  $A^{j}_{k+1}$ is chosen in such a way that  $A^{i}_{k}$ lies below the cube  $A^{j}_{k+1}$. This lemma is helpful in proving Lemma  \ref{case 1 flat case} to Lemma  \ref{flat case sp>1 a}.

\begin{lemma}\label{est2}
    Let  $A^{j}_{k+1}$ be a cube such that   $A^{i}_{k}$ lies below the cube  $A^{j}_{k+1}$. Then
    \begin{equation}
        |(u)_{A^{i}_{k}} - (u)_{A^{j}_{k+1}} |^{\tau} \leq C 2^{k(sp-d) \frac{\tau}{p}} [u]^{\tau}_{W^{s,p}(A^{i}_{k} \cup A^{j}_{k+1})},
    \end{equation}
   where $C$ is a constant independent of $i,~ j$, and $k$.
\end{lemma}
\begin{proof}
Let us consider  $|(u)_{A^{i}_{k}} - (u)_{A^{j}_{k+1}} |^{\tau}$  and using Lemma  \ref{avg} with  $E=A^{i}_{k}$ and  $F = A^{j}_{k+1}$, we get for some constant  $C$ (independent of  $i$, $j$, and $k$),
\begin{equation}\label{compineq1}
|(u)_{A^{i}_{k}} - (u)_{A^{j}_{k+1}} |^{\tau} \leq C \fint_{A^{i}_{k} \cup A^{j}_{k+1} }  |u(x) - (u)_{A^{i}_{k} \cup A^{j}_{k+1}}|^{\tau} \, dx.
\end{equation}
Choose an open set  $\Omega$ such that  $\Omega_{\lambda}$ is a translation of  $ A^{i}_{k} \cup A^{j}_{k+1}$ with scaling parameter  $\lambda=2^{k+1}$ (for example, if  $A^{j}_{k+1}= (2^{k+1}, 2^{k+2})^{d}$ and  $A^{i}_{k}= (2^{k+1},  2^{k+1} + 2^{k})^{d-1} \times (2^{k},2^{k+1})$ with  $k<-1$. Then  $\Omega = ((1, \frac{3}{2})^{d-1} \times (\frac{1}{2},1)) \cup (1,2)^{d} $ and  $\Omega_{\lambda} = A^{i}_{k} \cup A^{j}_{k+1}$ with  $\lambda=2^{k+1}$). Applying Lemma  \ref{sobolev} with this  $\Omega$ and  $\lambda=2^{k+1}$, and using translation invariance, we obtain
\begin{equation}\label{compineq2}
     \fint_{A^{i}_{k} \cup A^{j}_{k+1} }  |u(x) - (u)_{A^{i}_{k} \cup A^{j}_{k+1}}|^{\tau} \, dx \leq C 2^{k(sp-d) \frac{\tau}{p}} [u]^{\tau}_{W^{s,p}(A^{i}_{k} \cup A^{j}_{k+1})}.
\end{equation}
By combining  \eqref{compineq1} and  \eqref{compineq2}, we establish the lemma.
\end{proof}

\begin{remark}
\label{uniform-constant}
\normalfont Observe that for each cube $A^{j}_{k+1}$ there are exactly $2^{d-1}$ cubes $A^{i}_{k}$ lying below it. For different indices $i$, the associated sets $\Omega$ used in the proof differ, but since there are only $2^{d-1}$ possibilities, the constant appearing in Lemma \ref{sobolev} can be chosen uniformly over all such $\Omega$, and hence is independent of $i$. For different $j$, the configuration is simply a translation of the previous one, so the constant is independent of $j$ as well. Finally, by the scaling property in Lemma \ref{sobolev}, the constant does not depend on $k$. Thus the constant in Lemma \ref{est2} is uniform with respect to the indices $i$, $j$, and $k$.
\end{remark}

\smallskip

The next six lemmas, that is, Lemma \ref{case 1 flat case} to Lemma \ref{spd}, will be used in the next section for the proof of Theorem \ref{mainresult} (without optimality of the weight function depending on $\delta_{\Omega}$), Theorem \ref{theorem sp<1}, and Theorem \ref{theorem sp>1}.

\smallskip 

The next lemma proves (A$2$)(a) and (A$2$)(b) of Theorem \ref{mainresult} when $\Omega = \mathbb{R}^{d}_{+}$ with test functions supported on $\Omega_{n}$. For the case $d = 1$, it also proves (after applying it separately on the domains $(0,1)$ and $(1,2)$) (A$1$)(a) when $\Omega = (0,2)$, but only for test functions supported on $(0,1)$. In fact, the next theorem is the key theorem for the proof of parts (A$1$) and (A$2$) of Theorem \ref{mainresult}, because for a general Lipschitz domain, the proof follows by the usual patching technique using a partition of unity. This is done in the next section.

\begin{lemma}\label{case 1 flat case}
    Let $sp = 1$. If $\tau \geq p$ for $d = 1$ and $\tau \in [p, p^{*}_{s}]$ for $d > 1$, then there exists a constant $C = C(d, p, \tau) > 0$ such that
    \begin{equation}
     \left(  \bigintsss_{\Omega_{n}} \frac{|u(x)|^{\tau}}{x^{\alpha}_{d} \ln^{b} \left(\frac{2}{x_{d}} \right)}  \, dx \right)^{\frac{1}{\tau}} \leq  C 
 \|u\|_{W^{s,p}(\Omega_{n})}, \quad \forall \ u \in C^{1}_{c}(\Omega_n),
    \end{equation}
     where  $b= \tau-1$ when  $0 \leq \alpha <1$, and  $b= \tau$ when  $\alpha =1$.
\end{lemma}
\begin{proof}
For each  $x_{d} \in A^{i}_{k}$, we have  $\frac{2}{x_{d}} > \frac{1}{2^{k}}$, which  implies  $ \ln \left(\frac{2}{x_{d}} \right) > (-k) \ln 2$. Using this together with Lemma  \ref{sumineqlemma} and the fact that $\frac{1}{(-k)^{b}} \leq 1$, we obtain
\begin{align}\label{eqn0090}
    \bigintsss_{A^{i}_{k}} \frac{|u(x)|^{\tau}}{x^{\alpha}_{d} \ln^{b} \left(\frac{2}{x_{d}} \right)}  \, dx & \leq \frac{C}{(-k)^{b}} [u]^{\tau}_{W^{s,p}(A^{i}_{k})} + C \frac{2^{k(d-\alpha)}}{(-k)^{b}} |(u)_{A^{i}_{k}}|^{\tau} \nonumber \\ &
     \leq  C [u]^{\tau}_{W^{s,p}(A^{i}_{k})} + C \frac{2^{k(d-\alpha)}}{(-k)^{b}} |(u)_{A^{i}_{k}}|^{\tau}  .
\end{align}
Summing the above inequality from  $i=1$ to  $\sigma_{k}$, we obtain
\begin{equation*}
    \bigintsss_{A_{k}} \frac{|u(x)|^{\tau}}{x^{\alpha}_{d} \ln^{b} \left( \frac{2}{x_{d}} \right)}  \, dx \leq C   \sum_{i=1}^{\sigma_{k}} [u]^{\tau}_{W^{s,p}(A^{i}_{k})} + C \frac{2^{k(d-\alpha)}}{(-k)^{b}}  \sum_{i=1}^{\sigma_{k}} |(u)_{A^{i}_{k}}|^{\tau}.
\end{equation*}
Applying  \eqref{sumineq} with  $\gamma= \frac{\tau}{p}$, we have
\begin{equation}\label{seminorm ineq relation}
    \sum_{i=1}^{\sigma_{k}} [u]^{\tau}_{W^{s,p}(A^{i}_{k})} \leq \left( \sum_{i=1}^{\sigma_{k}} [u]^{p}_{W^{s,p}(A^{i}_{k})} \right)^{\frac{\tau}{p}} \leq [u]^{\tau}_{W^{s,p}(A_{k})} .
\end{equation}
Therefore, we have
\begin{equation*}
    \bigintsss_{A_{k}} \frac{|u(x)|^{\tau}}{x^{\alpha}_{d} \ln^{b} \left( \frac{2}{x_{d}} \right) } \, dx \leq C [u]^{\tau}_{W^{s,p}(A_{k})} +  C \frac{2^{k(d-\alpha)}}{(-k)^{b}}  \sum_{i=1}^{\sigma_{k}} |(u)_{A^{i}_{k}}|^{\tau}.
\end{equation*}
Again, summing the above inequality from  $k=m \in \mathbb{Z}^{-}$ to  $-1$, we get
\begin{equation}\label{eqnn2}
\sum_{k=m}^{-1} \bigintsss_{A_{k}} \frac{|u(x)|^{\tau}}{x^{\alpha}_{d} \ln^{b}\left( \frac{2}{x_{d}} \right)}  \, dx \leq  C \sum_{k=m}^{-1} [u]^{\tau}_{W^{s,p}(A_{k})} + C \sum_{k=m}^{-1} \frac{2^{k(d-\alpha)}}{(-k)^{b}}  \sum_{i=1}^{\sigma_{k}} |(u)_{A^{i}_{k}}|^{\tau}  .
\end{equation}
Let  $A^{j}_{k+1}$ be a cube such that  $A^{i}_{k}$ lies below the cube  $A^{j}_{k+1}$. Independently, using triangle inequality, we have 
\begin{equation*}
    |(u)_{A^{i}_{k}}|^\tau \leq  \left( |(u)_{A^{j}_{k+1}}| + |(u)_{A^{i}_{k}} - (u)_{A^{j}_{k+1}}| \right)^\tau.
\end{equation*}
For  $k \in \mathbb{Z}^{-} \backslash \{-1 \}$, applying Lemma  \ref{estimate} with  $c := \left( \frac{-k}{-k-(1/2)} \right)^{\tau-1}>1$ together with Lemma  \ref{est2} and  $sp=1$, we obtain
\begin{equation}\label{eqq0}
    |(u)_{A^{i}_{k}}|^{\tau} \leq \left( \frac{-k}{-k-(1/2)} \right)^{\tau-1} |(u)_{A^{j}_{k+1}}|^{\tau} + C (-k)^{\tau-1} 2^{k(1-d) \frac{\tau}{p}} [u]^{\tau}_{W^{s,p}(A^{i}_{k} \cup A^{j}_{k+1})}  .
\end{equation}
In the above estimate, we have used the fact that
\begin{equation*}
   (1-c^{\frac{-1}{\tau -1}})^{1-\tau} =  \left( \frac{1}{2} \right)^{1- \tau} (-k)^{\tau-1}.
\end{equation*}
Rewrite  \eqref{eqq0} as
\begin{equation*}
    \frac{|(u)_{A^{i}_{k}}|^{\tau}}{(-k)^{\tau-1}} \leq \frac{ |(u)_{A^{j}_{k+1}}|^{\tau}}{(-k-(1/2))^{\tau-1}} + C 2^{k(1-d) \frac{\tau}{p}} [u]^{\tau}_{W^{s,p}(A^{i}_{k} \cup A^{j}_{k+1})}  .
\end{equation*}
Multiplying the above inequality by  $2^{k(d-\alpha)}$ and using the definition of  $\alpha$ (see Definition \ref{alpha defn}), we get
\begin{equation*}
    \frac{2^{k(d-\alpha)}}{(-k)^{\tau-1}}  |(u)_{A^{i}_{k}}|^{\tau} \leq \frac{2^{k(d-\alpha)}}{ (-k-(1/2))^{\tau-1}} |(u)_{A^{j}_{k+1}}|^{\tau} + C [u]^{\tau}_{W^{s,p}(A^{i}_{k} \cup A^{j}_{k+1})}  .
\end{equation*}
Since there are $2^{d-1}$ such cubes $A^{i}_{k}$ lying below the cube $A^{j}_{k+1}$, summing the above inequality from $i = 2^{\,d-1}(j-1) + 1$ to $2^{\,d-1}j$, we obtain
\begin{align}\label{cubes below the cubes 1}
      \frac{2^{k(d-\alpha)}}{(-k)^{\tau-1}}  \sum_{i=2^{d-1}(j-1)+1}^{2^{d-1}j} |(u)_{A^{i}_{k}}|^{\tau} & \leq 2^{d-1} \frac{2^{k(d-\alpha)}}{ (-k-(1/2))^{\tau-1}} |(u)_{A^{j}_{k+1}}|^{\tau} \nonumber \\ & \quad
       + \ C  \sum_{i=2^{d-1}(j-1)+1}^{2^{d-1}j} [u]^{\tau}_{W^{s,p}(A^{i}_{k} \cup A^{j}_{k+1})}  .
\end{align}
Using the scaling identity $2^{k(d- \alpha)} = \frac{2^{(k+1)(d- \alpha)}}{2^{d- \alpha}}$, we further obtain
\begin{align*}
      \frac{2^{k(d-\alpha)}}{(-k)^{\tau-1}}  \sum_{i=2^{d-1}(j-1)+1}^{2^{d-1}j} |(u)_{A^{i}_{k}}|^{\tau} & \leq \frac{2^{(k+1)(d-\alpha)}}{2^{1-\alpha} (-k-(1/2))^{\tau-1}} |(u)_{A^{j}_{k+1}}|^{\tau} \nonumber \\ & \quad
       + \ C  \sum_{i=2^{d-1}(j-1)+1}^{2^{d-1}j} [u]^{\tau}_{W^{s,p}(A^{i}_{k} \cup A^{j}_{k+1})}  .
\end{align*}
Again, summing the above inequality from  $j=1$ to  $\sigma_{k+1}$ and using the fact that
\begin{equation}\label{summation relation}
  \sum_{j=1}^{\sigma_{k+1}} \left(  \sum_{i=2^{d-1}(j-1)+1}^{2^{d-1}j}  |(u)_{A^{i}_{k}}|^{\tau} \right)  = \sum_{i=1}^{\sigma_{k}}  |(u)_{A^{i}_{k}}|^{\tau},
\end{equation}
we obtain
\begin{align}\label{ineqn11}
\frac{2^{k(d-\alpha)}}{(-k)^{\tau-1}}  \sum_{i=1}^{\sigma_{k}} |(u)_{A^{i}_{k}}|^{\tau} & \leq \frac{2^{(k+1)(d-\alpha)}}{2^{1-\alpha} (-k-(1/2))^{\tau-1}} \sum_{j=1}^{\sigma_{k+1}} |(u)_{A^{j}_{k+1}}|^{\tau} \nonumber \\ & \quad
  + \ C \sum_{j=1}^{\sigma_{k+1}} \left( \sum_{i=2^{d-1}(j-1)+1}^{2^{d-1}j}  [u]^{\tau}_{W^{s,p}(A^{i}_{k} \cup A^{j}_{k+1})} \right).
\end{align}
Applying  \eqref{sumineq} with  $\gamma = \frac{\tau}{p}$ on the last term of the above equation, and using the fact that there are $2^{d-1}$ such cubes $A^{i}_{k}$ lying below the cube $A^{j}_{k+1}$, we have
\begin{align}\label{ineqn22}
    \sum_{j=1}^{\sigma_{k+1}} \left( \sum_{i=2^{d-1}(j-1)+1}^{2^{d-1}j}  [u]^{\tau}_{W^{s,p}(A^{i}_{k} \cup A^{j}_{k+1})} \right) & \leq \left( \sum_{j=1}^{\sigma_{k+1}} \left( \sum_{i=2^{d-1}(j-1)+1}^{2^{d-1}j}  [u]^{p}_{W^{s,p}(A^{i}_{k} \cup A^{j}_{k+1})} \right) \right)^{\frac{\tau}{p}} \nonumber \\ & \leq 2^{d-1} [u]^{\tau}_{W^{s,p}(A_{k} \cup A_{k+1})} .
\end{align}
Combining  \eqref{ineqn11} and  \eqref{ineqn22}, we obtain
\begin{equation*}
   \frac{2^{k(d-\alpha)}}{(-k)^{\tau-1}}  \sum_{i=1}^{\sigma_{k}} |(u)_{A^{i}_{k}}|^{\tau}  \leq \frac{2^{(k+1)(d-\alpha)}}{ 2^{1-\alpha} (-k-(1/2))^{\tau-1}} \sum_{j=1}^{\sigma_{k+1}} |(u)_{A^{j}_{k+1}}|^{\tau} + C  [u]^{\tau}_{W^{s,p}(A_{k} \cup A_{k+1})}  .
\end{equation*}
For simplicity, let  $a_{k} = \sum_{i=1}^{\sigma_{k}} |(u)_{A^{i}_{k}}|^{\tau}$. Then, the above inequality will become
\begin{equation*}
    \frac{2^{k(d-\alpha)}}{(-k)^{\tau-1}} a_{k} \leq \frac{2^{(k+1)(d-\alpha)}}{2^{1-\alpha} (-k-(1/2))^{\tau-1}} a_{k+1} + C  [u]^{\tau}_{W^{s,p}(A_{k} \cup A_{k+1})}  .
\end{equation*}
Summing the above inequality from  $k=m \in \mathbb{Z}^{-}$ to  $-2$, we get
\begin{equation*}
   \sum_{k=m}^{-2} \frac{2^{k(d-\alpha)}}{(-k)^{\tau-1}} a_{k} \leq \sum_{k=m}^{-2} \frac{2^{(k+1)(d-\alpha)}}{2^{1-\alpha} (-k-(1/2))^{\tau-1}} a_{k+1} + C  \sum_{k=m}^{-2} [u]^{\tau}_{W^{s,p}(A_{k} \cup A_{k+1})}  .
\end{equation*}
By changing sides, rearranging, and re-indexing, we get
\begin{multline*}
    \frac{2^{m(d-\alpha)}}{(-m)^{\tau-1}} a_{m} + \sum_{k=m+1}^{-2} \left\{ \frac{1}{(-k)^{\tau-1}} - \frac{1}{2^{1-\alpha}(-k+1/2)^{\tau-1}} \right\} 2^{k(d-\alpha)} a_{k}  \\ \leq C 2^{(-1)(d-\alpha)}a_{-1} 
    + \ C \sum_{k=m}^{-2} [u]^{\tau}_{W^{s,p}(A_{k} \cup A_{k+1})}  .
\end{multline*}
Now, using the lower bound from Lemma \ref{large n ineq} when $\alpha=1$, and the following lower bound when $0\leq \alpha<1$
\begin{equation*}
     \frac{1}{(-k)^{\tau-1}} - \frac{1}{2^{1-\alpha}(-k+1/2)^{\tau-1}} \geq \left(1-\frac{1}{2^{1-\alpha}} \right) \frac{1}{(-k)^{\tau-1}}, \quad \forall \ k \leq -1,
\end{equation*}
and choosing  $-m$ sufficiently large so that  $|(u)_{A^{j}_{m}}|=0$ for all  $j \in \{ 1, \dots, \sigma_{m} \}$, we obtain
\begin{equation*}
    \sum_{k=m}^{-2} \frac{2^{k(d-\alpha)}}{(-k)^{b}} a_{k} \leq  C a_{-1} +  C \sum_{k=m}^{-2} [u]^{\tau}_{W^{s,p}(A_{k} \cup A_{k+1})}  ,
\end{equation*}
where $b=\tau-1$ when $0 \leq \alpha<1$, and $b= \tau$ when $\alpha=1$. Therefore, adding   $2^{(\alpha -d)}a_{-1}$ on  both sides,  we have for some constant  $C= C(d,p, \tau)>0$, 
\begin{equation*}
    \sum_{k=m}^{-1} \frac{2^{k(d-\alpha)}}{(-k)^{b}} a_{k} \leq C  a_{-1} +  C \sum_{k=m}^{-2} [u]^{\tau}_{W^{s,p}(A_{k} \cup A_{k+1})}  .
\end{equation*}
Putting the value of  $a_{k}$ in the above inequality, we obtain
\begin{equation}\label{eqn10}
    \sum_{k=m}^{-1} \frac{2^{k(d-\alpha)}}{(-k)^{b}} \sum_{i=1}^{\sigma_{k}} |(u)_{A^{i}_{k}}|^{\tau} \leq C \sum_{j=1}^{\sigma_{-1}} |(u)_{A^{j}_{-1}}|^{\tau} + C \sum_{k=m}^{-2} [u]^{\tau}_{W^{s,p}(A_{k} \cup A_{k+1})}  .
\end{equation}
Combining  \eqref{eqnn2} and  \eqref{eqn10} yields
\begin{equation}\label{eqqnn}
    \sum_{k=m}^{-1} \bigintsss_{A_{k}} \frac{|u(x)|^{\tau}}{x^{\alpha}_{d} \ln^{b} \left( \frac{2}{x_{d}} \right)}  \, dx \leq  C \sum_{j=1}^{\sigma_{-1}} |(u)_{A^{j}_{-1}}|^{\tau} +  C \sum_{k=m}^{-2} [u]^{\tau}_{W^{s,p}(A_{k} \cup A_{k+1})}  .
\end{equation}
Also,
\begin{equation*}
      \sum_{j=1}^{\sigma_{-1}} |(u)_{A^{j}_{-1}}|^{\tau} =   \sum_{j=1}^{\sigma_{-1}} \Bigg| \frac{1}{|A^{j}_{-1}|} \int_{A^{j}_{-1}} u(x) \, dx \Bigg|^{\tau}
     \leq  C \sum_{j=1}^{\sigma_{-1}} \Bigg| \int_{A^{j}_{-1}} |u(x)|^{p} \, dx \Bigg|^{\frac{\tau}{p}}.
\end{equation*}
Here we have used H$\ddot{\text{o}}$lder's inequality with  $ \frac{1}{p} +  \frac{1}{p'} = 1$. Now applying  \eqref{sumineq} with  $\gamma= \frac{\tau}{p}$, we obtain
\begin{equation*}
    \sum_{j=1}^{\sigma_{-1}} \Bigg| \int_{A^{j}_{-1}} |u(x)|^{p} \, dx \Bigg|^{\frac{\tau}{p}} \leq   \Bigg| \sum_{j=1}^{\sigma_{-1}}\int_{A^{j}_{-1}} |u(x)|^{p} \, dx \Bigg|^{\frac{\tau}{p}}  \leq 
    \Bigg| \int_{\Omega_{n}} |u(x)|^{p} \, dx \Bigg|^{\frac{\tau}{p}} \leq  \|u\|^{\tau}_{L^{p}(\Omega_{n})}.
\end{equation*}
Therefore, by combining above two inequalities, we have
\begin{equation}\label{lp estimates}
     \sum_{j=1}^{\sigma_{-1}} |(u)_{A^{j}_{-1}}|^{\tau} \leq C   \|u\|^{\tau}_{L^{p}(\Omega_{n})}.
\end{equation}
Hence, from  \eqref{eqqnn} and  \eqref{lp estimates}, we obtain
\begin{equation*}
     \sum_{k=m}^{-1} \bigintsss_{A_{k}} \frac{|u(x)|^{\tau}}{x^{\alpha}_{d} \ln^{b} \left( \frac{2}{x_{d}} \right)}  \, dx \leq  C \left( \|u\|^{\tau}_{L^{p}(\Omega_{n})} +  \sum_{k=m}^{-2} [u]^{\tau}_{W^{s,p}(A_{k} \cup A_{k+1})} \right)  .
\end{equation*}
Using the fact that  $\tau \geq p$  and Subsection  \ref{some estimate} (see Appendix  \ref{appen}), we obtain
 \begin{equation*}
      \left( \bigintsss_{\Omega_{n}} \frac{|u(x)|^{\tau}}{x^{\alpha}_{d} \ln^{b}\left( \frac{2}{x_{d}} \right)}  \, dx \right)^{\frac{1}{\tau}} \leq C \left( [u]^{\tau}_{W^{s,p}(\Omega_{n})} + \|u\|^{\tau}_{L^{p}(\Omega_{n})} \right)^{\frac{1}{\tau}}  
       \leq C \left( [u]^{p}_{W^{s,p}(\Omega_{n})} + \|u\|^{p}_{L^{p}(\Omega_{n})} \right)^{\frac{1}{p}} .
 \end{equation*}
 This proves the lemma.
\end{proof}

\smallskip

The next lemma will improve the weight function presented in Lemma  \ref{case 1 flat case} when  $d>1$ and  $\tau \in (p + \frac{1}{d} , p^{*}_{s}]$. This lemma proves (A$2$) of Theorem  \ref{mainresult} when  $\Omega = \mathbb{R}^{d}_{+}$ with test functions restricted on  $\Omega_{n}$ with  $d>1$ and  $\tau \in (p+ \frac{1}{d}, p^{*}_{s}]$. 
\begin{lemma}\label{optlemma}
     Let  $d>1$ and  $sp=1$. Then there exists a constant  $C= C(d,p, \tau)>0$ such that, for any  $\tau \in [p, p^{*}_{s}] $, we have
    \begin{equation}
     \left(  \bigintsss_{\Omega_{n}} \frac{|u(x)|^{\tau}}{x^{\alpha}_{d} \ln^{\beta} \left(\frac{2}{x_{d}} \right)}  \, dx \right)^{\frac{1}{\tau}} \leq   C \|u\|_{W^{s,p}(\Omega_{n})}, \quad \ \forall \ u \in C^{1}_{c}(\Omega_{n}) ,
    \end{equation}
     where  $\beta=p$ when  $\tau=p, ~  \beta = \tau -1$ when  $\tau \in \left(p, p + \frac{1}{d}\right]$, and  $\beta= dp +  (1-d) \tau$ when  $\tau \in \left(p + \frac{1}{d}, p^{*}_{s} \right]$.
\end{lemma}
\begin{proof}
It is sufficient to prove for  $\tau \in \left(p + \frac{1}{d}, p^{*}_{s} \right]$. For  $\tau \in \left[p, p + \frac{1}{d}\right]$, it follows from Lemma  \ref{case 1 flat case}. Let  $\tau_{1} = p+ \frac{1}{d}$, using Lemma  \ref{case 1 flat case}, we have
\begin{equation*}
        \left(  \bigintsss_{\Omega_{n}} \frac{|u(x)|^{\tau_{1}}}{x^{ 1- \frac{1}{p^{*}_{s}}}_{d} \ln^{\tau_{1} -1} \left(\frac{2}{x_{d}} \right)}  \, dx \right)^{\frac{1}{\tau_{1}}} \leq  C \|u\|_{W^{s,p}(\Omega_{n})} .
    \end{equation*}
   Also, from  \eqref{Sobineqq}, we have
    \begin{equation*}
        \left(  \int_{\Omega_{n}} |u(x)|^{p^{*}_{s}}   \, dx \right)^{\frac{1}{p^{*}_{s}}} \leq  C \|u\|_{W^{s,p}(\Omega_{n})} .
    \end{equation*}
    For any  $\tau \in (\tau_{1}, p^{*}_{s})$, there exists  $\theta \in (0,1)$ such that  $\tau = \theta \tau_{1} + (1- \theta)p^{*}_{s}$. Let  $a_{1} ,b_{1} \geq 0$, then by using H$\ddot{\text{o}}$lder's inequality with  $\frac{1}{1/ \theta}+ \frac{1}{1/(1- \theta)}= 1$, we obtain
    \begin{align*}
        \bigintsss_{\Omega_{n}} \frac{|u(x)|^{\tau}}{ x_{d}^{a_{1}} \ln^{b_{1}}\left( \frac{2}{x_{d}} \right)}  \, dx & =  \bigintsss_{\Omega_{n}} \frac{|u(x)|^{\theta \tau_{1} + (1- \theta)p^{*}_{s}}}{ x_{d}^{a_{1}} \ln^{b_{1}}\left( \frac{2}{x_{d}} \right) } \, dx \\ & \leq  \left( \bigintsss_{\Omega_{n}} \frac{|u(x)|^{\tau_{1}}}{ x_{d}^{\frac{a_{1}}{\theta}} \ln^{\frac{b_{1}}{\theta}}\left( \frac{2}{x_{d}} \right)}  \, dx \right)^{\theta}  \left( \int_{\Omega_{n}} |u(x)|^{p^{*}_{s}} \, dx \right)^{1- \theta} .
    \end{align*}
    Let  $\frac{a_{1}}{\theta} = 1- \frac{1}{p^{*}_{s}}$ and  $\frac{b_{1}}{\theta} = \tau_{1} -1 $, we have
    \begin{align*}
         \bigintsss_{\Omega_{n}} \frac{|u(x)|^{\tau}}{ x_{d}^{a_{1}} \ln^{b_{1}}\left( \frac{2}{x_{d}} \right)}  \, dx  & \leq  \left( \bigintsss_{\Omega_{n}} \frac{|u(x)|^{\tau_{1}}}{ x^{1- \frac{1}{p^{*}_{s}}}_{d} \ln^{\tau_{1}-1} \left( \frac{2}{x_{d}} \right)}  \, dx \right)^{\theta} \left( \int_{\Omega_{n}} |u(x)|^{p^{*}_{s}} \, dx \right)^{1- \theta} \\ & \leq C \|u\|^{\theta \tau_{1}}_{W^{s,p}(\Omega_{n})}  \|u\|^{(1- \theta) p^{*}_{s}}_{W^{s,p}(\Omega_{n})}  \leq
         C \|u\|^{\tau}_{W^{s,p}(\Omega_{n})}.
    \end{align*}
    From the definition of  $\theta$, we obtain  $a_{1} = \alpha$ and  $b_{1} = \beta$. Therefore, we have
    \begin{equation*}
       \left( \bigintsss_{\Omega_{n}} \frac{|u(x)|^{\tau}}{ x_{d}^{\alpha} \ln^{\beta}\left( \frac{2}{x_{d}} \right)}  \, dx \right)^{\frac{1}{\tau}} \leq C \|u\|_{W^{s,p}(\Omega_{n})},
    \end{equation*}
    for  $\tau \in \left(p + \frac{1}{d} , p^{*}_{s} \right)$. Finally, when $\tau = p^{*}_{s}$, we have $\alpha=0$ and $\beta=0$, and the lemma follows directly from the fractional Sobolev inequality.
\end{proof}
\smallskip

The next lemma establishes the generalized FBHI \eqref{generalized1} in the subcritical case when the domain is $\mathbb{R}^{d}_{+}$ and the test functions are supported in $\Omega_{n}$. This lemma is useful in the proof of Theorem \ref{theorem sp<1}.

\begin{lemma}\label{flat case sp<1}
     Let  $p>1$ and  $ sp<1$. For any  $\tau \in [p, p^{*}_{s}] $, there exists a constant  $C= C(d,p,s, \tau) > 0$  such that
    \begin{equation}
     \left(  \int_{\Omega_{n}} \frac{|u(x)|^{\tau}}{x^{\alpha}_{d} }  \, dx \right)^{\frac{1}{\tau}} \leq  C 
 \|u\|_{W^{s,p}(\Omega_{n})}, \quad \forall \ u \in C^{1}_{c}(\Omega_n).
    \end{equation}
\end{lemma}
\begin{proof}
Let $p>1$ and $s \in (0,1)$. From Lemma  \ref{sumineqlemma}, we have
\begin{equation*}
     \int_{A^{i}_{k}} \frac{|u(x)|^{\tau}}{x^{\alpha}_{d}}  \, dx \leq C [u]^{\tau}_{W^{s,p}(A^{i}_{k})} + C 2^{k(d-\alpha)} |(u)_{A^{i}_{k}}|^{\tau}  .
\end{equation*}
Summing the above inequality from  $i=1$ to  $\sigma_{k}$ and using  \eqref{seminorm ineq relation}, we obtain
\begin{align*}
    \int_{A_{k}} \frac{|u(x)|^{\tau}}{x^{\alpha}_{d}}  \, dx & \leq C   \sum_{i=1}^{\sigma_{k}} [u]^{\tau}_{W^{s,p}(A^{i}_{k})} + C 2^{k(d-\alpha)}  \sum_{i=1}^{\sigma_{k}} |(u)_{A^{i}_{k}}|^{\tau}  \\ &
    \leq C [u]^{\tau}_{W^{s,p}(A_{k})} +  C 2^{k(d-\alpha)}  \sum_{i=1}^{\sigma_{k}} |(u)_{A^{i}_{k}}|^{\tau}.
\end{align*}
Again, summing the above inequality from  $k=m \in \mathbb{Z}^{-}$ to  $-1$, we obtain, for any $p>1$ and $s \in (0,1)$,
\begin{equation}\label{eqnn1}
\sum_{k=m}^{-1} \int_{A_{k}} \frac{|u(x)|^{\tau}}{x^{\alpha}_{d}}  \, dx \leq  C \sum_{k=m}^{-1} [u]^{\tau}_{W^{s,p}(A_{k})} + C \sum_{k=m}^{-1} 2^{k(d-\alpha)} \sum_{i=1}^{\sigma_{k}} |(u)_{A^{i}_{k}}|^{\tau}  .
\end{equation} 
Now assume $sp<1$. For any  $\tau \in [p, p^{*}_{s}]$, we have  $\alpha \in [0,sp]$. Let  $A^{j}_{k+1}$ be a cube such that  $A^{i}_{k}$ lies below the cube $A^{j}_{k+1}$. Independently, using triangle inequality, we have 
\begin{equation*}
    |(u)_{A^{i}_{k}}|^\tau \leq  \left( |(u)_{A^{j}_{k+1}}| + |(u)_{A^{i}_{k}} - (u)_{A^{j}_{k+1}}| \right)^\tau.
\end{equation*}
For  $k \in \mathbb{Z}^{-} \backslash \{-1 \}$, applying Lemma  \ref{estimate} with  $c :=c_{1}2^{1 - \alpha}>1$ where  $c_{1} = \frac{2}{1+ 2^{1- \alpha}} <1$ together with Lemma  \ref{est2}, we obtain
\begin{equation*}
    |(u)_{A^{i}_{k}}|^{\tau} \leq c_{1} 2^{1- \alpha} |(u)_{A^{j}_{k+1}}|^{\tau} + C  2^{k(sp-d) \frac{\tau}{p}} [u]^{\tau}_{W^{s,p}(A^{i}_{k} \cup A^{j}_{k+1})}  .
\end{equation*}
Multiplying the above inequality by  $2^{k(d-\alpha)}$ and using the definition of  $\alpha$ (see Definition \ref{alpha defn}), we get
\begin{equation*}
   2^{k(d-\alpha)} |(u)_{A^{i}_{k}}|^{\tau} \leq c_{1} 2^{k(d-\alpha) + (1-\alpha)}|(u)_{A^{j}_{k+1}}|^{\tau} + C [u]^{\tau}_{W^{s,p}(A^{i}_{k} \cup A^{j}_{k+1})}  .
\end{equation*}
Since there are $2^{d-1}$ such cubes $A^{i}_{k}$ lying below the cube $A^{j}_{k+1}$, summing the above inequality from $i = 2^{d-1}(j-1) + 1$ to $2^{d-1}j$, we obtain
\begin{align*}
     2^{k(d-\alpha)}  \sum_{i=2^{d-1}(j-1)+1}^{2^{d-1}j} |(u)_{A^{i}_{k}}|^{\tau} & \leq c_{1} 2^{d-1} 2^{k(d-\alpha) + (1- \alpha)} |(u)_{A^{j}_{k+1}}|^{\tau} \\ & \quad
       + \ C  \sum_{i=2^{d-1}(j-1)+1}^{2^{d-1}j} [u]^{\tau}_{W^{s,p}(A^{i}_{k} \cup A^{j}_{k+1})}  .
\end{align*}
Again, summing the above inequality from  $j=1$ to  $\sigma_{k+1}$ and using  \eqref{summation relation}, we obtain
\begin{align*}
2^{k(d-\alpha)}  \sum_{i=1}^{\sigma_{k}} |(u)_{A^{i}_{k}}|^{\tau} & \leq c_{1} 2^{(k+1)(d-\alpha)} \sum_{j=1}^{\sigma_{k+1}} |(u)_{A^{j}_{k+1}}|^{\tau} \nonumber \\ & \quad 
  + \ C \sum_{j=1}^{\sigma_{k+1}} \left( \sum_{i=2^{d-1}(j-1)+1}^{2^{d-1}j}  [u]^{\tau}_{W^{s,p}(A^{i}_{k} \cup A^{j}_{k+1})} \right).
\end{align*}
From  \eqref{ineqn22}, we obtain
\begin{equation*}
  2^{k(d-\alpha)} \sum_{i=1}^{\sigma_{k}} |(u)_{A^{i}_{k}}|^{\tau}  \leq c_{1} 2^{(k+1)(d-\alpha)} \sum_{j=1}^{\sigma_{k+1}} |(u)_{A^{j}_{k+1}}|^{\tau} + C  [u]^{\tau}_{W^{s,p}(A_{k} \cup A_{k+1})}  .
\end{equation*}
For simplicity, let  $a_{k} = \sum_{i=1}^{\sigma_{k}} |(u)_{A^{i}_{k}}|^{\tau}$. Then, the above inequality will become
\begin{equation*}
    2^{k(d-\alpha)} a_{k} \leq c_{1} 2^{(k+1)(d-\alpha)} a_{k+1} + C  [u]^{\tau}_{W^{s,p}(A_{k} \cup A_{k+1})}  .
\end{equation*}
Summing the above inequality from  $k=m \in \mathbb{Z}^{-}$ to  $-2$, we get
\begin{equation*}
   \sum_{k=m}^{-2} 2^{k(d-\alpha)} a_{k} \leq \sum_{k=m}^{-2} c_{1} 2^{(k+1)(d-\alpha)} a_{k+1} + C  \sum_{k=m}^{-2} [u]^{\tau}_{W^{s,p}(A_{k} \cup A_{k+1})}  .
\end{equation*}
By changing sides, rearranging, and re-indexing, we get
\begin{equation*}
   2^{m(d-\alpha)} a_{m} + (1-c_{1}) \sum_{k=m+1}^{-2} 2^{k(d-\alpha)} a_{k}  \leq  C a_{-1} 
    +  C \sum_{k=m}^{-2} [u]^{\tau}_{W^{s,p}(A_{k} \cup A_{k+1})}.
\end{equation*}
Adding  $2^{(\alpha -d)}a_{-1}$ on  both sides and using the fact $(1-c_{1}) 2^{m(d-\alpha)} a_{m} \leq 2^{m(d-\alpha)} a_{m}$, we have for some constant  $C= C(d,p, \tau)>0$,
\begin{equation*}
    (1-c_{1}) \sum_{k=m}^{-1} 2^{k(d-\alpha)} a_{k} \leq C a_{-1} +  C \sum_{k=m}^{-2} [u]^{\tau}_{W^{s,p}(A_{k} \cup A_{k+1})}  .
\end{equation*}
Putting the value of  $a_{k}$ in the above inequality, we obtain
\begin{equation}\label{eqnn99}
 (1-c_{1})   \sum_{k=m}^{-1} 2^{k(d-\alpha)} \sum_{i=1}^{\sigma_{k}} |(u)_{A^{i}_{k}}|^{\tau} \leq C \sum_{j=1}^{\sigma_{-1}} |(u)_{A^{j}_{-1}}|^{\tau} + C \sum_{k=m}^{-2} [u]^{\tau}_{W^{s,p}(A_{k} \cup A_{k+1})}  .
\end{equation}
Combining  \eqref{eqnn1} and  \eqref{eqnn99} yields
\begin{equation}\label{eqqnnn}
    \sum_{k=m}^{-1} \int_{A_{k}} \frac{|u(x)|^{\tau}}{x^{\alpha}_{d} }  \, dx \leq  C \sum_{j=1}^{\sigma_{-1}} |(u)_{A^{j}_{-1}}|^{\tau} +  C \sum_{k=m}^{-2} [u]^{\tau}_{W^{s,p}(A_{k} \cup A_{k+1})}  .
\end{equation}
From  \eqref{lp estimates}, we obtain
\begin{equation*}
     \sum_{k=m}^{-1} \int_{A_{k}} \frac{|u(x)|^{\tau}}{x^{\alpha}_{d} }  \, dx \leq  C \left( \|u\|^{\tau}_{L^{p}(\Omega_{n})} +  \sum_{k=m}^{-2} [u]^{\tau}_{W^{s,p}(A_{k} \cup A_{k+1})} \right)  .
\end{equation*}
Using the fact that  $\tau \geq p$ and Subsection  \ref{some estimate} (see Appendix  \ref{appen}), we obtain
 \begin{equation*}
      \left( \int_{\Omega_{n}} \frac{|u(x)|^{\tau}}{x^{\alpha}_{d} }  \, dx \right)^{\frac{1}{\tau}} \leq C \left( [u]^{\tau}_{W^{s,p}(\Omega_{n})} + \|u\|^{\tau}_{L^{p}(\Omega_{n})} \right)^{\frac{1}{\tau}}  
       \leq C \left( [u]^{p}_{W^{s,p}(\Omega_{n})} + \|u\|^{p}_{L^{p}(\Omega_{n})} \right)^{\frac{1}{p}} .
 \end{equation*}
 This finishes the proof of the lemma.
\end{proof}
\smallskip

The next lemma proves the generalized FBHI  \eqref{generalized2} for supercritical case when the domain is  $\mathbb{R}^{d}_{+}$ and the test functions supported in  $\Omega_{n}$. This lemma is useful in the proof of Theorem  \ref{theorem sp>1}.

\begin{lemma}\label{flat case sp>1 a}
     Let  $sp>1$. For any  $\tau \geq p$ when  $sp=d$, and  $\tau \in \left[p, p^{*}_{s} - \frac{p}{d-sp} \right) $ when  $sp<d$, there exists a constant  $C= C(d,p,s, \tau) > 0$  such that
    \begin{equation}
     \left(  \int_{\Omega_{n}} \frac{|u(x)|^{\tau}}{x^{\alpha}_{d} }  \, dx \right)^{\frac{1}{\tau}} \leq  C 
 [u]_{W^{s,p}(\Omega_{n})}, \quad \forall \ u \in C^{1}_{c}(\Omega_n).
    \end{equation}
\end{lemma}
\begin{proof}
For any  $\tau \in [p, p^{*}_{s} - \frac{p}{d-sp}) $ and  $sp<d$, we have  $\alpha \in (1, sp]$. Let  $A^{j}_{k+1}$ be a cube such that  $A^{i}_{k}$ lies below the cube  $A^{j}_{k+1}$. Independently, using triangle inequality, we have 
\begin{equation*}
    |(u)_{A^{j}_{k+1}}|^\tau \leq  \left( |(u)_{A^{i}_{k}}| + |(u)_{A^{i}_{k}} - (u)_{A^{j}_{k+1}}| \right)^\tau.
\end{equation*}
For  $k \in \mathbb{Z}^{-} \backslash \{-1 \}$, applying Lemma  \ref{estimate} with  $c :=c_{1}2^{\alpha-1}>1$, where  $c_{1} = \frac{2}{1+ 2^{\alpha-1}} <1$ together with Lemma  \ref{est2}, we obtain
\begin{equation*}
    |(u)_{A^{j}_{k+1}}|^{\tau} \leq c_{1} 2^{\alpha-1} |(u)_{A^{i}_{k}}|^{\tau} + C  2^{k(sp-d) \frac{\tau}{p}} [u]^{\tau}_{W^{s,p}(A^{i}_{k} \cup A^{j}_{k+1})}  .
\end{equation*}
Multiplying the above inequality by  $2^{(k+1)(d-\alpha)-d+1}$ and using the definition of  $\alpha$ (see Definition \ref{alpha defn}), we get
\begin{equation*}
   2^{(k+1)(d-\alpha)-d+1} |(u)_{A^{j}_{k+1}}|^{\tau} \leq c_{1} 2^{k(d-\alpha)}|(u)_{A^{i}_{k}}|^{\tau} + C [u]^{\tau}_{W^{s,p}(A^{i}_{k} \cup A^{j}_{k+1})}  .
\end{equation*}
Since there are $2^{d-1}$ such cubes $A^{i}_{k}$ lying below the cube $A^{j}_{k+1}$, we therefore sum the above inequality from $i = 2^{d-1}(j-1) + 1$ to $2^{d-1}j$, and obtain
\begin{align*}
     2^{d-1}2^{(k+1)(d-\alpha)-d+1}  |(u)_{A^{j}_{k+1}}|^{\tau} & \leq c_{1} 2^{k(d-\alpha)} \sum_{i=2^{d-1}(j-1)+1}^{2^{d-1}j}  |(u)_{A^{i}_{k}}|^{\tau} \\ & \quad
       + \ C  \sum_{i=2^{d-1}(j-1)+1}^{2^{d-1}j} [u]^{\tau}_{W^{s,p}(A^{i}_{k} \cup A^{j}_{k+1})}  .
\end{align*}
Again, summing the above inequality from  $j=1$ to  $\sigma_{k+1}$ and using  \eqref{summation relation}, we obtain
\begin{align}\label{ineqn1}
2^{(k+1)(d-\alpha)}  \sum_{j=1}^{\sigma_{k+1}}  |(u)_{A^{j}_{k+1}}|^{\tau} & \leq c_{1} 2^{k(d-\alpha)} 
\sum_{i=1}^{\sigma_{k}} |(u)_{A^{i}_{k}}|^{\tau} \nonumber \\ & \quad
  + \ C \sum_{j=1}^{\sigma_{k+1}} \left( \sum_{i=2^{d-1}(j-1)+1}^{2^{d-1}j}  [u]^{\tau}_{W^{s,p}(A^{i}_{k} \cup A^{j}_{k+1})} \right).
\end{align}
From  \eqref{ineqn22}, we obtain
\begin{equation*}
2^{(k+1)(d-\alpha)}  \sum_{j=1}^{\sigma_{k+1}}  |(u)_{A^{j}_{k+1}}|^{\tau} \leq c_{1} 2^{k(d-\alpha)} 
\sum_{i=1}^{\sigma_{k}} |(u)_{A^{i}_{k}}|^{\tau} + C  [u]^{\tau}_{W^{s,p}(A_{k} \cup A_{k+1})}  .
\end{equation*}
For simplicity, let  $a_{k} = \sum_{i=1}^{\sigma_{k}} |(u)_{A^{i}_{k}}|^{\tau}$. Then, the above inequality will become
\begin{equation*}
    2^{(k+1)(d-\alpha)} a_{k+1} \leq c_{1} 2^{k(d-\alpha)} a_{k} + C  [u]^{\tau}_{W^{s,p}(A_{k} \cup A_{k+1})}  .
\end{equation*}
Summing the above inequality from  $k=m \in \mathbb{Z}^{-}$ to  $-2$, we get
\begin{equation*}
   \sum_{k=m}^{-2} 2^{(k+1)(d-\alpha)} a_{k+1} \leq c_{1} \sum_{k=m}^{-2}  2^{k(d-\alpha)} a_{k} + C  \sum_{k=m}^{-2} [u]^{\tau}_{W^{s,p}(A_{k} \cup A_{k+1})}  .
\end{equation*}
By changing sides, rearranging, and re-indexing, we get
\begin{equation*}
 (1-c_{1}) \sum_{k=m+1}^{-1} 2^{k(d-\alpha)} a_{k}  \leq  2^{m(d-\alpha)} a_{m} +  C \sum_{k=m}^{-2} [u]^{\tau}_{W^{s,p}(A_{k} \cup A_{k+1})}  .
\end{equation*}
Putting the value of  $a_{k}$ in the above inequality, we obtain
\begin{equation*}
 (1-c_{1})   \sum_{k=m+1}^{-1} 2^{k(d-\alpha)} \sum_{i=1}^{\sigma_{k}} |(u)_{A^{i}_{k}}|^{\tau} \leq 2^{m(d- \alpha)} \sum_{j=1}^{\sigma_{m}} |(u)_{A^{j}_{m}}|^{\tau} + C \sum_{k=m}^{-2} [u]^{\tau}_{W^{s,p}(A_{k} \cup A_{k+1})}  .
\end{equation*}
Choose  $m$ large enough such that  $|(u)_{A^{j}_{m}}|=0$ for all  $j \in \{ 1, \dots, \sigma_{m} \}$. Therefore, we have
\begin{equation}\label{eqnn}
    (1-c_{1})   \sum_{k=m}^{-1} 2^{k(d-\alpha)} \sum_{i=1}^{\sigma_{k}} |(u)_{A^{i}_{k}}|^{\tau} \leq C \sum_{k=m}^{-2} [u]^{\tau}_{W^{s,p}(A_{k} \cup A_{k+1})}  .
\end{equation}
Combining  \eqref{eqnn1} with $sp>1$ and  \eqref{eqnn} yields
\begin{equation*}
    \sum_{k=m}^{-1} \int_{A_{k}} \frac{|u(x)|^{\tau}}{x^{\alpha}_{d} }  \, dx \leq   C \sum_{k=m}^{-2} [u]^{\tau}_{W^{s,p}(A_{k} \cup A_{k+1})}  .
\end{equation*}
Therefore, using Subsection  \ref{some estimate} (see Appendix  \ref{appen}), we have
 \begin{equation*}
      \left( \int_{\Omega_{n}} \frac{|u(x)|^{\tau}}{x^{\alpha}_{d} }  \, dx \right)^{\frac{1}{\tau}} \leq C [u]_{W^{s,p}(\Omega_{n})}.
 \end{equation*}
 This proves the lemma.
\end{proof}

\smallskip

The following lemma will be proved by an analogous modification of the proof of Lemma  \ref{case 1 flat case}. This lemma will be helpful in proving (A$3$) of Theorem  \ref{mainresult}.

\begin{lemma}\label{spd}
    Let  $d \geq1, ~  p>1, ~ s \in (0,1)$ be such that  $sp=d$. Then  for any  $\tau \geq p$ and  $u \in C^{1}_{c}(B_{R}(0)^{c})$ for some  $R>0$, we have
    \begin{equation}
    \left( \bigintsss_{B_{R}(0)^c} \frac{|u(x)|^{\tau}}{|x|^{d} \ln^{\tau} \left( \frac{2|x|}{R} \right) }  \, dx \right)^{\frac{1}{\tau}} \leq  C \|u\|_{W^{s,p}(B_{R}(0)^c)}  ,
    \end{equation}
     where  $C= C(d,p, \tau, R)>0$.  
\end{lemma}
\begin{proof}
    For each  $k \geq 0$, set
    \begin{equation*}
        A_{k} = \{ x \in \mathbb{R}^d : 2^{k} R < |x| \leq 2^{k+1}R \} \ .
    \end{equation*}
    Then, we have
    \begin{equation*}
        B_{R}(0)^c  = \bigcup_{k=0}^{\infty} A_{k} \ .
    \end{equation*}
    Applying Lemma  \ref{sobolev} with  $\Omega = \{ x \in \mathbb{R}^d : ~  R < |x|< 2R \}, ~ \lambda= 2^{k}$ and  $sp=d$, we have
    \begin{equation*}
         \left( \fint_{A_{k}} |u(x)-(u)_{A_{k}}|^{\tau } \, dx \right)^{\frac{1}{\tau}}  \, dx \leq C  [u]_{W^{s,p}(A_{k})}   , 
    \end{equation*}
  where  $C=C(d,p,\tau, R)$ is a constant. Let  $x \in A_{k}$, which implies  $ \frac{1}{|x|} < \frac{1}{2^{k}R}$. Therefore, we have
    \begin{align*}
         \int_{A_{k}} \frac{|u(x)|^{\tau}}{|x|^{d}}  \, dx & \leq \frac{C}{2^{k d}} \int_{A_{k}} |u(x)-(u)_{A_{k}} + (u)_{A_{k}}|^{\tau}  \, dx 
  \\  & \leq \frac{C}{2^{k d}} \int_{A_{k}} |u(x)-(u)_{A_{k}}|^{\tau}  \, dx  + \frac{C}{2^{k d}} \int_{A_{k}} |(u)_{A_{k}}|^{\tau}  \, dx \\ &
    = C \frac{|A_{k}|}{2^{k d}} \fint_{A_{k}} |u(x)-(u)_{A_{k}}|^{\tau}  \, dx + 
   C \frac{|A_{k}|}{2^{k d}}  |(u)_{A_{k}}|^{\tau} \\ &
    \leq C [u]^{\tau}_{W^{s,p}(A_{k})} + C  |(u)_{A_{k}}|^{\tau}  .
    \end{align*}
Here, the constant  $C$ depends on  $d, ~ p, ~ \tau$ and  $R$. For each  $x \in A_{k}$, we have  $ \frac{2|x|}{R}  >  2^{k+1}$. This implies that  $\ln \left( \frac{2|x|}{R} \right) > (k+1) \ln2$. Therefore, using the fact $\frac{1}{(k+1)^{\tau}} \leq 1$, we have
\begin{equation*}
    \bigintsss_{A_{k}} \frac{|u(x)|^{\tau}}{|x|^{d} \ln^{\tau}\left( \frac{2|x|}{R} \right)}  \, dx \leq C \frac{[u]^{\tau}_{W^{s,p}(A_{k})} }{(k+1)^{\tau}} + C \frac{|(u)_{A_{k}}|^{\tau}}{(k+1)^{\tau}} 
     \leq  C [u]^{\tau}_{W^{s,p}(A_{k})} + C \frac{|(u)_{A_{k}}|^{\tau}}{(k+1)^{\tau}}   .\end{equation*}
Summing the above inequality from  $k=0$ to  $ m+1 \in \mathbb{Z}^{+}$, we get
\begin{equation}\label{spd1}
    \bigintsss_{\{ R < |x| \leq 2^{m+2}R \} }  \frac{|u(x)|^{\tau}}{|x|^{d} \ln^{\tau}\left( \frac{2|x|}{R} \right)}  \, dx \leq C \sum_{k=0}^{m+1} [u]^{\tau}_{W^{s,p}(A_{k})} + C   \sum_{k=0}^{m+1} \frac{|(u)_{A_{k}}|^{\tau}}{(k+1)^{\tau}}  .
\end{equation} 
Using the Lemma  \ref{avg} with $E= A_{k}$ and $F=A_{k+1}$, and Lemma  \ref{sobolev} with  $\Omega = \{ x \in \mathbb{R}^d : R < |x| < 2^{2}R \}$ and  $\lambda=2^{k}$, we obtain, for some constant  $C$ (independent of  $k$),
\begin{eqnarray*}
   |(u)_{A_{k}} - (u)_{A_{k+1}} |^{\tau} &\leq& C \fint_{A_{k} \cup A_{k+1} }  |u(x) - (u)_{A_{k} \cup A_{k+1}}|^{\tau}  \, dx 
   \leq C [u]^{\tau}_{W^{s,p}(A_{k} \cup A_{k+1})} .
\end{eqnarray*}
Independently, using triangle inequality, we have 
\begin{equation*}
    |(u)_{A_{k+1}}|^\tau \leq (|(u)_{A_{k}}| + |(u)_{A_{k}} - (u)_{A_{k+1}}|)^\tau.
\end{equation*}
For  $k  \geq 0$, consider the number  $c := (\frac{k+2}{k+3/2})^{\tau-1}>1$. Applying Lemma  \ref{estimate} with  $c$, we obtain
\begin{equation*}
    |(u)_{A_{k+1}}|^{\tau} \leq \left( \frac{k+2}{k+ 3/2} \right)^{\tau -1} |(u)_{A_{k}}|^{\tau} + C (k+2)^{\tau-1} [u]^{\tau}_{W^{s,p}(A_{k} \cup A_{k+1})}.
\end{equation*}
Therefore, we have
\begin{equation*}
   \frac{|(u)_{A_{k+1}}|^{\tau}}{(k+2)^{\tau-1}}  \leq  \frac{|(u)_{A_{k}}|^{\tau}}{(k+ 3/2)^{\tau-1}}  + C  [u]^{\tau}_{W^{s,p}(A_{k} \cup A_{k+1})}.
\end{equation*}
Summing the above inequality from  $k=0$ to  $m \in \mathbb{Z}^{+}$, we get
\begin{equation*}
     \sum_{k=0}^{m} \frac{|(u)_{A_{k+1}}|^{\tau}}{(k+2)^{\tau-1}}  \leq \sum_{k=0}^{m}  \frac{|(u)_{A_{k}}|^{\tau}}{(k+ 3/2)^{\tau-1}}  + C \sum_{k=0}^{m}  [u]^{\tau}_{W^{s,p}(A_{k} \cup A_{k+1})}.
\end{equation*}
By changing sides, rearranging, and re-indexing, we get
\begin{equation*}
    \sum_{k=1}^{m+1} \left\{  \frac{1}{(k+1)^{\tau-1}} -  \frac{1}{(k+ 3/2)^{\tau-1}}  \right\} |(u)_{A_{k}}|^{\tau} \leq C|(u)_{A_{0}}|^{\tau} + C \sum_{k=0}^{m}  [u]^{\tau}_{W^{s,p}(A_{k} \cup A_{k+1})}.
\end{equation*}
Now, for all  $k$, using the lower bound in Lemma \ref{large n ineq}, we arrive at
\begin{equation*}
    \sum_{k=0}^{m+1} \frac{|(u)_{A_{k}}|^{\tau}}{(k+1)^{\tau}} \leq C |(u)_{A_{0}}|^{\tau} + C \sum_{k=0}^{m+1} [u]^{\tau}_{W^{s,p}(A_{k} \cup A_{k+1})}.
\end{equation*}
Also, 
\begin{equation*}
     |(u)_{A_{0}}|^{\tau} =   \Big|  \frac{1}{|A_{0}|} \int_{A_{0}} u(x) \, dx  \Big|^{\tau} \leq C \Big| \int_{A_{0}} |u(x)|^{p} \Big|^{\frac{\tau}{p}} \\ \leq C \|u\|^{\tau}_{L^{p}(B_{R}(0)^{c})}.
\end{equation*} 
Therefore, we have
\begin{equation}\label{spd2}
     \sum_{k=0}^{m+1} \frac{|(u)_{A_{k}}|^{\tau}}{(k+1)^{\tau}} \leq  C \sum_{k=0}^{m+1} [u]^{\tau}_{W^{s,p}(A_{k} \cup A_{k+1})} + C \|u\|^{\tau}_{L^{p}(B_{R}(0)^{c})}.
\end{equation}
Combining  \eqref{spd1}  and  \eqref{spd2}, we obtain
\begin{equation*}
      \bigintsss_{\{ R < |x| \leq 2^{m+2}R \} }  \frac{|u(x)|^{\tau}}{|x|^{d} \ln^{\tau}\left( \frac{2|x|}{R} \right)}  \, dx \leq C \sum_{k=0}^{m+1} [u]^{\tau}_{W^{s,p}(A_{k} \cup A_{k+1})} +  C \|u\|^{\tau}_{L^{p}(B_{R}(0)^c)} .
\end{equation*}
Hence,  using the fact that $\tau \geq p$, we obtain
\begin{align*}
   \left( \bigintsss_{B_{R}(0)^c} \frac{|u(x)|^{\tau}}{|x|^{d} \ln^{\tau} \left( \frac{2|x|}{R} \right) }  \, dx \right)^{\frac{1}{\tau}} & \leq C \left( [u]^{\tau}_{W^{s,p}(B_{R}(0)^{c})} + \|u\|^{\tau}_{L^{p}(B_{R}(0)^{c})} \right)^{\frac{1}{\tau}}  \\ & \leq  C \left( [u]^{p}_{W^{s,p}(B_{R}(0)^{c})} + \|u\|^{p}_{L^{p}(B_{R}(0)^{c})} \right)^{\frac{1}{p}}  .  
\end{align*}
This proves the lemma.
\end{proof}

%------------PROOF OF THE MAIN RESULTS--------------------------

\section{Proof of the main results}\label{proof of main result}
In this section, we will prove Theorem  \ref{mainresult} without the optimality of weight function depending on $\delta_{\Omega}$, Theorem  \ref{theorem sp<1} and Theorem  \ref{theorem sp>1}. 
\smallskip

Let  $\Omega$ be a bounded Lipschitz domain. Consider the definition of bounded Lipschitz domain defined in Section  \ref{preliminaries}. For simplicity, let  $T_{x}$ be the identity map. Then 
\begin{equation*}
  \Omega \cap B_{r_{x}}(x) = \{ \xi = (\xi',\xi_{d})  : \xi_{d} > \phi_{x}(\xi') \} \cap B_{r_{x}}(x)  
\end{equation*}
and  $\partial \Omega \subset \cup_{x \in \partial \Omega} B_{r_{x}}(x)$. Since $\partial \Omega$ is compact, there exists $x_{1}, \dots, x_{n} \in \partial \Omega$ such that
\begin{equation*}
    \partial \Omega  \subset \bigcup_{i=1}^{n} B_{r_{i}}(x_{i}) ,
\end{equation*}
where $r_{i}=r_{x_{i}}$. From \cite[Theorem 6.78]{leonibook}, we have $W^{s,p}(\Omega) = W^{s,p}_{0}(\Omega)$ for the case $sp\leq1$. Therefore, it is sufficient to prove our main results for $u \in C^{1}_{c}(\Omega)$.

\subsection{Proof of Theorem \ref{mainresult} without optimality of the weight function} Let $u \in C^{1}_{c}(\Omega)$. Assume (A$1$). First, we will prove for the case $\tau \in [p,p^{*}_{s}]$ for $d>1$, and then we consider the case $d=1$. Let $\Omega \subset \cup_{i=0}^{n} \Omega_{i}$, where $\overline{\Omega}_{0} \subset \Omega$ and  $\Omega_{i} = B_{r_{i}}(x_{i})$ for all  $ 1 \leq i \leq n $. Let  $ \{ \eta_{i} \}_{i=0}^{n}$ be the associated partition of unity. Then
\begin{equation*}
    u = \sum_{i=0}^{n} u_{i}, \quad \text{where} \  u_{i} = \eta_{i} u.
\end{equation*}
From Lemma  \ref{testfunc}, we have
\begin{equation*}
 \|u_{i}\|_{W^{s,p}(\Omega)} \leq C \|u\|_{W^{s,p}(\Omega)}, \quad \forall \ 0 \leq i \leq n .
\end{equation*}
Therefore, it is sufficient to prove Theorem  \ref{mainresult} for all  $u_{i}, ~  0 \leq i \leq n$. Since  $\operatorname{supp} u_{0} \subset \Omega_{0}$, and for all  $x \in \Omega_{0}$, 
 \begin{equation*}
     C_{1} \leq \delta_{\Omega}(x) \leq C_{2} \quad \text{for some} \  C_{1}, C_{2} >0 ,
 \end{equation*}
 it follows from the fractional Sobolev inequality \eqref{Sobineqq} that
\begin{equation*} 
 \left( \bigintsss_{\Omega_{0}} \frac{|u_{0}(x)|^{\tau}}{\delta_{\Omega}^{\alpha}(x) \ln^{\beta}\left( \frac{2R}{ \delta_{\Omega}}\right)}  \, dx \right)^{\frac{1}{\tau}} \leq C \left( \int_{\Omega_{0}} |u_{0}(x)|^{\tau} \, dx \right)^{\frac{1}{\tau}} \leq C \|u_{0}\|_{W^{s,p}(\Omega_{0})} .
\end{equation*}
For  $1 \leq i \leq n$, we have $ \operatorname{supp} u_{i} \subset \Omega \cap \Omega_{i} $. Consider the transformation  $F: \mathbb{R}^d \to \mathbb{R}^d$ defined by  $F(x',x_{d}) = (x',x_{d} - \phi_{x_{i}}(x'))$, and let $G = F^{-1}$  (see Subsection  \ref{lipschitzdomain}, Appendix \ref{appen}). Then 
 \begin{equation*}
      \delta_{\Omega}(x) \sim \xi_{d}, \quad \forall \  x \in \Omega \cap \Omega_{i},
 \end{equation*}
where  $F(x) = (\xi_{1}, \dots, \xi_{d})$. Therefore, from Lemma  \ref{optlemma}, we have
 \begin{align*}
    \left( \bigintsss_{\Omega \cap \Omega_{i}} \frac{|u_{i}(x)|^{\tau}}{\delta_{\Omega}^{\alpha}(x) \ln^{\beta}\left( \frac{2R}{\delta_{\Omega}} \right)}  \, dx \right)^{\frac{1}{\tau}} & \sim  \left( \bigintsss_{F(\Omega \cap \Omega_{i})} \frac{|u_{i} \circ G(\xi)|^{\tau}}{\xi_{d}^{\alpha} \ln^{\beta}\left( \frac{2R}{\xi_{d}} \right)} \, d \xi \right)^{\frac{1}{\tau}} \\ & \leq C \|u_{i} \circ G\|_{W^{s,p}(F(\Omega \cap \Omega_{i}))} = C \|u_{i}\|_{W^{s,p}(\Omega \cap \Omega_{i})}.
 \end{align*}
Now consider  $d=1$, in this case we have  $\tau \geq p$. Following the similar approach illustrated above for the case  $d>1$ and using Lemma  \ref{case 1 flat case}, we obtain
\begin{equation*}
      \left( \bigintsss_{\Omega} \frac{|u(x)|^{\tau}}{\delta_{\Omega}(x) \ln^{\tau} \left( \frac{2R}{\delta_{\Omega}(x)} \right)}  \, dx \right)^{\frac{1}{\tau}} \leq C \|u\|_{W^{s,p}(\Omega)}.
\end{equation*}
Now, assume  $\tau < p$. Then by using H$\ddot{\text{o}}$lder's inequality with  $\frac{1}{(p/ \tau)} + \frac{1}{q} = 1 $ and define  $d \nu := \frac{\, dx}{\delta_{\Omega}(x) \ln^{p}\left( \frac{2R}{\delta_{\Omega}(x)} \right) }$, we have
\begin{align*}
    \bigintsss_{\Omega} \frac{|u(x)|^{\tau}}{\delta_{\Omega}(x) \ln ^{p} \left( \frac{2R}{\delta_{\Omega}(x) } \right)} \, dx & = \int_{\Omega} |u(x)|^{\tau} d \nu  \leq \left( \int_{\Omega} |u(x)|^{p} d \nu \right)^{\frac{\tau}{p}} \left( \int_{\Omega} d \nu \right)^{1/q} \\ & = \left( \int_{\Omega} |u(x)|^{p} d \nu \right)^{\frac{\tau}{p}}  \left( \bigintsss_{\Omega} \frac{\, dx}{\delta_{\Omega}(x) \ln^{p}\left( \frac{2R}{\delta_{\Omega}(x)} \right) } \right)^{1/q} \\ &  \leq  C \left( \int_{\Omega} |u(x)|^{p} d \nu \right)^{\frac{\tau}{p}} \\ & 
    =  C \left( \bigintsss_{\Omega} \frac{|u(x)|^{p}}{\delta_{\Omega}(x) \ln^{p} \left( \frac{2R}{\delta_{\Omega}(x)} \right) }  \, dx  \right)^{\frac{\tau}{p}} \leq C \|u\|^{\tau}_{W^{s,p}(\Omega)}.
\end{align*}
Here, we have used that 
\begin{equation*}
   \bigintsss_{\Omega} \frac{\, dx}{\delta_{\Omega}(x) \ln^{p}\left( \frac{2R}{\delta_{\Omega}(x)} \right) }   < \infty .
\end{equation*}
Therefore, we have
\begin{equation*}
 \left(    \bigintsss_{\Omega} \frac{|u(x)|^{\tau}}{\delta_{\Omega}(x) \ln ^{p} \left( \frac{2R}{\delta_{\Omega}(x) } \right)} \, dx \right)^{\frac{1}{\tau}} \leq C \|u\|_{W^{s,p}(\Omega)}.
\end{equation*}
\bigskip

Assume (A$2$). Choose  $0<R_{1} < R_{2}<1$ such that  $R_{2}<R$. Let  $G_{1} = \Omega \backslash  \overline{\Omega}_{R_{2}}$ and  $G_{2}= \Omega_{R_{1}} \backslash \overline{\Omega}_{R}$. Let  $\chi \in C^{\infty}(\Omega)$ such that  $\chi(x)=1$ for all  $x \in \Omega \backslash \Omega_{R_{1}}, ~  0 \leq \chi \leq 1$ and  $\chi(x) = 0$ for all  $x \in  \overline{\Omega}_{R_{2}}$. Then
\begin{equation*}
    u = \chi u + (1- \chi)u = u_{1} + u_{2},
\end{equation*}
 which implies that  $\operatorname{supp} u_{1} \subset \overline{G}_{1}$ and  $\operatorname{supp} u_{2} \subset \overline{G}_{2}$. From Lemma  \ref{testfunc}, we have
\begin{equation*}
    \|u_{i}\|_{W^{s,p}(\Omega)} \leq C \|u\|_{W^{s,p}(\Omega)}, \quad \text{for} \ i=1,2.
\end{equation*}
Therefore, it is sufficient to prove Theorem  \ref{mainresult} for  $u_{1}$ and  $u_{2}$. Consider the transformation  $F: \mathbb{R}^d \to \mathbb{R}^d$ defined by  $F(x',x_{d}) = (x',x_{d} - \phi_{x_{i}}(x'))$, and let $G = F^{-1}$  (see Subsection  \ref{lipschitzdomain}, Appendix \ref{appen}), then  $F(G_{1})$ is a subset of  $\mathbb{R}^{d-1} \times (0,1)$. Also, there exists  $n \in \mathbb{N}$ such that  $\operatorname{supp} u_{1} \circ G \subset (-n,n)^{d-1} \times (0,1)$ on  $F(G_{1})$. From Subsection  \ref{lipschitzdomain}, we have
 \begin{equation*}
     \delta_{\Omega}(x) \sim \xi_{d},
 \end{equation*}
 for all  $x \in G_{1}$, where  $F(x) = (\xi', \xi_{d})$. Therefore, from Lemma  \ref{optlemma}, we have
 \begin{align*}
    \left( \bigintsss_{G_{1}} \frac{|u_{1}(x)|^{\tau}}{\delta_{\Omega}^{\alpha}(x) \ln^{\beta}\left( \frac{2R}{\delta_{\Omega}(x)} \right)}  \, dx \right)^{\frac{1}{\tau}} & \sim  \left( \bigintsss_{F(G_{1})} \frac{|u_{1} \circ G(\xi)|^{\tau}}{\xi_{d}^{\alpha} \ln^{\beta}\left( \frac{2R}{\xi_{d}} \right)} \,  d \xi \right)^{\frac{1}{\tau}} \\ & \leq C \|u_{1} \circ G \|_{W^{s,p}(F(G_{1}))} = C \|u_{1}\|_{W^{s,p}(G_{1})}.
 \end{align*}
Since for all  $x \in \Omega_{2}$, we have
\begin{equation*}
    C_{1} \leq \delta_{\Omega}(x) \leq C_{2}, \quad \text{for some} \  C_{1}, C_{2} >0,
\end{equation*}
it follows from the Sobolev inequality \eqref{Sobineqq} that
\begin{equation*} 
 \left( \bigintsss_{G_{2}} \frac{|u_{2}(x)|^{\tau}}{\delta_{\Omega}^{\alpha}(x) \ln^{\beta}\left( \frac{2R}{ \delta_{\Omega}(x)}\right)}  \, dx \right)^{\frac{1}{\tau}} \leq C \left( \int_{G_{2}} |u_{2}(x)|^{\tau} \, dx \right)^{\frac{1}{\tau}} \leq C \|u_{2}\|_{W^{s,p}(G_{2})} .
\end{equation*}
If  $\tau <p$, following a similar approach as illustrated in the proof of (A$1$) in Theorem  \ref{mainresult}, we obtain
\begin{equation*}
    \left(    \bigintsss_{\Omega} \frac{|u(x)|^{\tau}}{\delta_{\Omega}(x) \ln ^{p} \left( \frac{2R}{\delta_{\Omega}(x) } \right)} \, dx \right)^{\frac{1}{\tau}} \leq C \|u\|_{W^{s,p}(\Omega)}.
\end{equation*}
 \bigskip

Assume (A$3$). Let  $D = \mathbb{R}^d \backslash \Omega$. We may assume that  $0 \in D$ and $\overline{D} \subset B_{R}(0)$ for some  $R>0$. Since  $\partial D = \partial \Omega$, it follows that  $\delta_{D}(x) = \delta_{\Omega}(x)$ for all  $x \in \Omega$. Let  $\chi \in C^{\infty}_{c}(B_{2R}(0))$ be such that  $\chi(x)=1$ for all  $x \in B_{\frac{3R}{2}}(0)$ and  $0 \leq \chi \leq 1$. Then
\begin{equation*}
    u = \chi u + (1-\chi) u = u_{1}+u_{2}  ,
\end{equation*}
which implies that $\operatorname{supp} u_{1} \subset B_{2R}(0) \cap \Omega$ and   $\operatorname{supp} u_{2} \subset B_{R}(0)^c$. 
From Lemma  \ref{testfunc}, we have
\begin{equation*}
    \|u_{i}\|_{W^{s,p}(\Omega)} \leq C \|u\|_{W^{s,p}(\Omega)}, \quad \text{for} \ i=1,2 .
\end{equation*}
Therefore, it is sufficient to prove Theorem \ref{mainresult} for  $u_{1}$ and  $u_{2}$. Since  $\operatorname{supp} u_{2} \subset B_{R}(0)^{c}$ and  $\delta_{\Omega}(x) \sim |x|$ for all  $x \in B_{R}(0)^c$, it follows from Lemma \ref{spd} that
 \begin{align*}
    \left( \bigintsss_{B_{R}(0)^c \cap \{ \delta_{\Omega}(x) >R \}} \frac{|u_{2}(x)|^{p}}{\delta_{\Omega}^{d} (x) \ln^{p} \left( \frac{2\delta_{\Omega}(x)}{R} \right) }  \, dx \right)^{\frac{1}{p}} & \sim \left( \bigintsss_{B_{R}(0)^c} \frac{|u_{2}(x)|^{p}}{|x|^{d} \ln^{p} \left( \frac{2|x|}{R} \right) }  \, dx \right)^{\frac{1}{p}} \\ & \leq  C \|u_{2}\|_{W^{s,p}(B_{R}(0)^{c})} .
 \end{align*}
Also, for all $x \in B_{R}(0)^{c} \cap \{ \delta_{\Omega}(x) < R \}$, we have $C_{1} \leq \ln \left( \frac{2R}{\delta_{\Omega}(x)} \right) \leq C_{2}  $. Therefore,
\begin{equation*}
    \left( \bigintsss_{B_{R}(0)^c \cap \{ \delta_{\Omega}(x) <R \}} \frac{|u_{2}(x)|^{p}}{\delta_{\Omega}^{d} (x) \ln^{p} \left( \frac{2R}{\delta_{\Omega}(x)} \right) }  \, dx \right)^{\frac{1}{p}}  \leq C \left( \int_{B_{R}(0)^c \cap \{ \delta_{\Omega}(x) <R \}} |u_{2}(x)|^{p} \, dx \right)^{\frac{1}{p}}.
\end{equation*}
For all  $x \in B_{2R}(0) \cap \Omega$, we have
\begin{equation*}
    \frac{1}{\delta^{d}_{\Omega}(x) \ln^{p} \left( \frac{2R}{ \delta_{\Omega}(x)} \right)} \leq \frac{1}{\delta^{d}_{\Omega}(x)} \leq \frac{1}{\delta^{d}_{B_{2R}(0) \cap \Omega}(x)}.
\end{equation*}
Therefore, from (T$1$) of Theorem  \ref{dyda} for  $sp=d$ and  $d>1$, we have
\begin{equation*}
    \bigintsss_{B_{2R}(0) \cap \Omega}  \frac{|u_{1}(x)|^{p}}{\delta^{d}_{\Omega} (x) \ln^{p} \left( \frac{2R}{ \delta_{\Omega}(x)} \right)} \, dx \leq \bigintsss_{B_{2R}(0) \cap \Omega} \frac{|u_{1}(x)|^{p}}{\delta^{d}_{B_{2R}(0) \cap \Omega} (x) } \, dx \leq C [u_{1}]_{W^{s,p}(B_{2R}(0) \cap \Omega)}.
\end{equation*}
This finishes the proof of Theorem \ref{mainresult} without the optimality of the weight function.

\subsection{Proof of Theorem  \ref{theorem sp<1}}
Let  $u \in C^{1}_{c}(\Omega)$ and assume $sp<1$. Let  $\Omega \subset \cup_{i=0}^{n} \Omega_{i}$, where  $\overline{\Omega}_{0} \subset \Omega$ and  $\Omega_{i} = B_{r_{i}}(x_{i})$ for all  $ 1 \leq i \leq n $. Let  $ \{ \eta_{i} \}_{i=0}^{n}$ be the associated partition of unity. Then
\begin{equation*}
    u = \sum_{i=0}^{n} u_{i}, \quad \text{where} \  u_{i} = \eta_{i} u.
\end{equation*}
From Lemma  \ref{testfunc}, we have
\begin{equation*}
    \|u_{i}\|_{W^{s,p}(\Omega)} \leq C \|u\|_{W^{s,p}(\Omega)}, \quad \forall \ 0 \leq i \leq n .
\end{equation*}
Therefore, it is sufficient to prove Theorem  \ref{theorem sp<1} for all  $u_{i}, ~ 0 \leq i \leq n$. Since  $\operatorname{supp} u_{0} \subset \Omega_{0}$, and for all  $x \in \Omega_{0}$, 
 \begin{equation*}
     C_{1} \leq \delta_{\Omega}(x) \leq C_{2} \quad \text{for some} \  C_{1}, C_{2} >0 ,
 \end{equation*}
 it follows from the fractional Sobolev inequality \eqref{Sobineqq} that
\begin{equation*} 
 \left( \int_{\Omega_{0}} \frac{|u_{0}(x)|^{\tau}}{\delta_{\Omega}^{\alpha}(x) }  \, dx \right)^{\frac{1}{\tau}} \leq C \left( \int_{\Omega_{0}} |u_{0}(x)|^{\tau} \, dx \right)^{\frac{1}{\tau}} \leq C \|u_{0}\|_{W^{s,p}(\Omega_{0})} .
\end{equation*}
For  $1 \leq i \leq n$, we have $ \operatorname{supp} u_{i} \subset \Omega \cap \Omega_{i} $. Consider the transformation  $F: \mathbb{R}^d \to \mathbb{R}^d$ defined by  $F(x',x_{d}) = (x',x_{d} - \phi_{x_{i}}(x'))$, and let $G = F^{-1}$  (see Subsection  \ref{lipschitzdomain}, Appendix \ref{appen}). Then 
 \begin{equation*}
      \delta_{\Omega}(x) \sim \xi_{d}, \quad \forall \  x \in \Omega \cap \Omega_{i},
 \end{equation*}
where  $F(x) = (\xi_{1}, \dots, \xi_{d})$. Therefore, from Lemma \ref{flat case sp<1}, we have
 \begin{align*}
    \left( \int_{\Omega \cap \Omega_{i}} \frac{|u_{i}(x)|^{\tau}}{\delta_{\Omega}^{\alpha}(x) }  \, dx \right)^{\frac{1}{\tau}} & \sim  \left( \int_{F(\Omega \cap \Omega_{i})} \frac{|u_{i} \circ G(\xi)|^{\tau}}{\xi_{d}^{\alpha}} \,  d \xi \right)^{\frac{1}{\tau}} \\  & \leq C \|u_{i} \circ G\|_{W^{s,p}(F(\Omega \cap \Omega_{i}))} = C \|u_{i}\|_{W^{s,p}(\Omega \cap \Omega_{i})}.
 \end{align*}
Now assume that $\tau < p$. Let $d \nu = \frac{dx}{\delta^{sp}_{\Omega}(x)}$. Following a similar approach to that used in the proof of A($1$) of Theorem \ref{mainresult} for the case $\tau<p$, we obtain
\begin{equation*}
    \int_{\Omega} \frac{|u(x)|^{\tau}}{\delta^{sp}_{\Omega}(x) } \, dx = \int_{\Omega} |u(x)|^{\tau} d \nu   \leq C \|u\|^{\tau}_{W^{s,p}(\Omega)}.
\end{equation*}
This completes the proof of Theorem \ref{theorem sp<1}.

\subsection{Proof of Theorem  \ref{theorem sp>1}}\label{proof of theorem sp>1}
Let  $u \in C^{1}_{c}(\Omega)$ and assume $sp>1$. To prove Theorem  \ref{theorem sp>1}, we will utilize  Lemma  \ref{flat case sp>1 a}. For this, first we consider the case  $\tau=p$ when  $sp<d$, and  $\tau \geq p$ when  $sp=d$. Therefore, following a similar approach illustrated in the proof of Theorem \ref{theorem sp<1} and using Lemma \ref{flat case sp>1 a}, one can obtain the following inequality,
\begin{equation}\label{proof of theorem sp>1 full norm}
\left(    \int_{\Omega  } \frac{|u(x)|^{\tau}}{\delta^{sp}_{\Omega}(x) } \, dx \right)^{\frac{1}{\tau}} \leq C \|u\|_{W^{s,p}(\Omega)},
\end{equation}
for  $ \tau=p $ when  $sp<d$, and  $\tau \geq p$ when  $sp=d$. To prove Theorem  \ref{theorem sp>1}, it is sufficient to prove for some constant  $C>0$
\begin{equation}\label{sp>1 fractional sobolev}
    \int_{\Omega} |u(x)|^{p} \, dx \leq C [u]^{p}_{W^{s,p}(\Omega)}, \quad \ \forall \ u \in W^{s,p}_{0}(\Omega),
\end{equation}
when  $sp>1$. Assume this is not true. Suppose we have a sequence of functions  $\{ u_{n} \}_{n \in \mathbb{N}}$ defined on the domain  $\Omega$, such that  $\int_{\Omega} |u_{n}(x)|^{p} \, dx =1 $ for all  $n \in \mathbb{N}$ and  $[u_{n}]_{W^{s,p}(\Omega)} \to 0$ as  $n \to \infty$. Let  $\mathcal{F}= \{ u_{n}: ~ n \in \mathbb{N} \}$, then  $\mathcal{F}$ satisfies the conditions of Lemma  \ref{compactness} and hence  $\mathcal{F}$ is pre-compact in  $L^{p}(\Omega)$. Therefore, there exists a subsequence (which we will denote as  $\{ u_{n} \}_{n \in \mathbb{N}}$ again) such that   $u_{n} \to u$ pointwise almost everywhere as  $n \to \infty$ and  $\int_{\Omega} |u(x)|^{p} \, dx = 1$. Also, by Fatou's lemma, we have
\begin{equation*}
    [u]_{W^{s,p}(\Omega)} \leq \liminf_{n \to \infty} [u_{n}]_{W^{s,p}(\Omega)} = 0 .
\end{equation*}
Therefore,  $u$ must be a constant function. As  $\int_{\Omega} |u(x)|^{p} \, dx =1 $, we have  $u(x) = |\Omega|^{\frac{-1}{p}}$ almost everywhere. From  \eqref{proof of theorem sp>1 full norm} and  $sp>1$, we have
\begin{equation*}
    \infty = |\Omega|^{\frac{-1}{p}} \left( \int_{\Omega} \frac{1}{\delta^{sp}_{\Omega}(x)} \, dx \right)^{\frac{1}{\tau}} \leq C\|u\|_{W^{s,p}(\Omega)} = C, 
\end{equation*}
which is a contradiction. Therefore, from  \eqref{proof of theorem sp>1 full norm} and  \eqref{sp>1 fractional sobolev}, we have
\begin{equation}\label{ineq frac Hardy sp>k a}
    \left(    \int_{\Omega} \frac{|u(x)|^{\tau}}{\delta^{sp}_{\Omega}(x) } \, dx \right)^{\frac{1}{\tau}} \leq C [u]_{W^{s,p}(\Omega)},
\end{equation}
for  $ \tau=p $ when  $sp<d$, and  $\tau \geq p$ when  $sp=d$. Now assume $sp<d$ and $\tau \in (p,p^{*}_{s})$, there exists $\theta \in (0,1)$ such that $\tau = \theta p + (1-\theta) p^{*}_{s}$. Using H$\ddot{\text{o}}$lder's inequality and $a_{1}>0$, we obtain
\begin{equation*}
    \int_{\Omega} \frac{|u(x)|^{\tau}}{\delta^{a_{1}}_{\Omega}(x)} \, dx =  \int_{\Omega} \frac{|u(x)|^{\theta p + (1- \theta)p^{*}_{s}}}{\delta^{a_{1}}_{\Omega}(x)} \, dx \leq \left( \int_{\Omega} \frac{|u(x)|^{p}}{\delta^{\frac{a_{1}}{\theta}}_{\Omega}(x)} \, dx \right)^{\theta} \left( \int_{\Omega} |u(x)|^{p^{*}_{s}} \, dx \right)^{1- \theta}.
\end{equation*}
Choose $a_{1}>0$ such that $\frac{a_{1}}{\theta} = sp$. Then using  \eqref{ineq frac Hardy sp>k a} with $\tau=p, ~  sp<d$ and fractional Sobolev inequality  \eqref{Sobineqq} with  \eqref{sp>1 fractional sobolev}, we have
\begin{equation*}
    \int_{\Omega} \frac{|u(x)|^{\tau}}{\delta^{a_{1}}_{\Omega}(x)} \, dx \leq C [u]^{\theta p}_{W^{s,p}(\Omega)} [u]^{(1-\theta)p^{*}_{s}}_{W^{s,p}(\Omega)}= C[u]^{\tau}_{W^{s,p}(\Omega)}. 
\end{equation*}
From the definition of $\theta$, we obtain $a_{1} = d+ (sp-d) \frac{\tau}{p} = \alpha$. Therefore, combining all the above inequalities, we get
\begin{equation*}
    \left( \int_{\Omega} \frac{|u(x)|^{\tau}}{\delta^{\alpha}_{\Omega}(x)} \, dx \right)^{\frac{1}{\tau}}  \leq C [u]_{W^{s,p}(\Omega)},
\end{equation*}
for $\tau \in \left[p, p^{*}_{s} \right)$ when $sp<d$, and $\tau \geq p$ when $sp=d$. For $sp<d$ and $\tau= p^{*}_{s}$, we have $\alpha=0$. Therefore, the above inequality follows by applying \eqref{sp>1 fractional sobolev} to the fractional Sobolev inequality \eqref{Sobineqq}.  Now assume $\tau=p$. Using fractional Poincar\'e\ inequality  \eqref{poincare}, one can obtain that Lemma  \ref{sumineqlemma} and Lemma  \ref{est2} holds true for any $s \in (0,1)$ and $\tau=p$. Therefore, Lemma  \ref{flat case sp>1 a} holds true for any $s \in (0,1)$ and $\tau=p$ satisfying $sp>1$. Therefore, following a similar approach as illustrated above for the case $sp \leq d$, one can obtain the following inequality for $sp>1$,
\begin{equation*}
    \left(    \int_{\Omega} \frac{|u(x)|^{p}}{\delta^{sp}_{\Omega}(x) } \, dx \right)^{\frac{1}{p}} \leq C [u]_{W^{s,p}(\Omega)}.
\end{equation*}

%---------OPTIMALITY OF THE WEIGHT FUNCTION--------------------

\section{Optimality of the weight function}\label{optimality of weight function}
In this section, we prove the optimality of the weight function  in Theorem \ref{mainresult}, among the class of functions depending on $\delta_{\Omega}$. The next lemma establishes an inequality for the Gagliardo seminorm of a radial function defined on an annulus in $\mathbb{R}^{d}$ (which reduces to a ball when the inner radius is zero). This result will be essential for proving the convergence of a sequence of functions in higher dimensions, which will be carried out in the next theorem.

\begin{lemma}\label{d dim to 1 dim}
Let $d \geq 1$ and $p \geq 1$. For $R_{2} > R_{1} \geq 0$, define $A(R_{1},R_{2}) := B_{R_{2}}(0)\setminus \overline{B_{R_{1}}(0)} \subset \mathbb{R}^{d}$,  if $R_{1}=0$, set $A(R_{1},R_{2}) = B_{R_{2}}(0)$. Let $u : A(R_{1},R_{2}) \to \mathbb{R}$ be a radial function, and for any
$r \in (R_{1},R_{2})$ with $r=|x|$, $x \in A(R_{1},R_{2})$, define $\widetilde{u}(r) := u(|x|)$. Then there exists a constant $C = C(d,p,s,R_{2}) > 0$ such that
    \begin{equation*}
        [u]_{W^{s,p}(A(R_{1}, R_{2}))} \leq C [\widetilde{u}]_{W^{s,p}((R_{1},R_{2}))}.
    \end{equation*}
\end{lemma}
\begin{proof}
Let us begin with the expression of $[u]^{p}_{W^{s,p}(A(R_{1},R_{2}))}$. Using the change of variables $x= r_{1} \omega_{1}$ and $y= r_{2} \omega_{2}$, we have
 \begin{multline}\label{polar coordinate change}
     \int_{A(R_{1},R_{2})} \int_{A(R_{1},R_{2})} \frac{|u(x)-u(y)|^{p}}{ |x-y|^{d+sp}} \, dx \, dy \\ = \int_{\mathbb{S}^{d-1}} \int_{\mathbb{S}^{d-1}} \int_{R_{1}}^{R_{2}} \int_{R_{1}}^{R_{2}} \frac{|\widetilde{u}(r_{1})-\widetilde{u}(r_{2})|^{p}}{ |r_{1} \omega_{1} - r_{2} \omega_{2}|^{d+sp} } r^{d-1}_{1} r^{d-1}_{2} \, dr_{1} \, dr_{2} \, d \omega_{1} \, d \omega_{2}.  
 \end{multline}
 Without loss of generality, we may assume $r_{1} > r_{2}$. Let us estimate the following expression:
 \begin{multline*}
   r^{d-1}_{1}  r^{d-1}_{2}  \int_{\mathbb{S}^{d-1}} \int_{\mathbb{S}^{d-1}} \frac{1}{|r_{1} \omega_{1} - r_{2} \omega_{2}|^{d+sp}} \,   d \omega_{1} \, d \omega_{2} \\ =  r^{d-1}_{1}  r^{d-1}_{2}  \int_{\mathbb{S}^{d-1}} \Bigg( \int_{\mathbb{S}^{d-1}} \frac{dw_{2}}{|r_{1} \omega_{1} - r_{2} \omega_{2}|^{d+sp}}   
 \Bigg) \, d \omega_{1}.
 \end{multline*}
 By rotational invariance, we have
 \begin{equation}\label{eqn001}
     \int_{\mathbb{S}^{d-1}} \Bigg( \int_{\mathbb{S}^{d-1}} \frac{dw_{2}}{|r_{1} \omega_{1} - r_{2} \omega_{2}|^{d+sp}}   
 \Bigg) \, d \omega_{1} = |\mathbb{S}^{d-1}|  \int_{\mathbb{S}^{d-1}}  \frac{dw_{2}}{|r_{1} e_{d} - r_{2} \omega_{2}|^{d+sp}},
 \end{equation}
 where $e_{d} = (0, \dots, 0,1)$. By changing the variable $\omega_{2} = (z, t)$ with
 \begin{equation*}
     t = \pm \sqrt{1-|z|^{2}}, \hspace{.3cm} z \in B^{'}_{1}(0) \subset \mathbb{R}^{d-1},
 \end{equation*}
 and using Subsection \ref{changeofvariable} (see Appendix \ref{appen}), we therefore get
 \begin{align*}
     \int_{\mathbb{S}^{d-1}}  \frac{dw_{2}}{|r_{1} e_{d} - r_{2} \omega_{2}|^{d+sp}} & = 
2 \int_{B^{'}_{1}(0)} \frac{1}{((r_{1}- tr_{2})^{2} + r^{2}_{2} |z|^{2} )^{\frac{d+sp}{2}}} \frac{dz}{\sqrt{1-|z|^{2}}} \\ & \leq  2 \int_{B^{'}_{1}(0)} \frac{1}{((r_{1}- r_{2})^{2} + r^{2}_{2} |z|^{2} )^{\frac{d+sp}{2}}} \frac{dz}{\sqrt{1-|z|^{2}}} \\ & = \frac{2}{|r_{1}-r_{2}|^{d+sp}} \bigintsss_{B^{'}_{1}(0)} \frac{1}{\left( 1+ \left| \frac{r_{2}}{r_{1}-r_{2}}z \right|^{2} \right)^{\frac{d+sp}{2}}}   \frac{dz}{\sqrt{1-|z|^{2}}}.
\end{align*}
In the above estimate, we have used $(r_{1} - r_{2}) \leq (r_{1} - tr_{2})$. For simplicity, let $\lambda = \frac{r_{2}}{r_{1}-r_{2}}$. Using the change of variable $\lambda z = \eta$, we have
\begin{equation}\label{eqn002}
     \int_{\mathbb{S}^{d-1}}  \frac{dw_{2}}{|r_{1} e_{n} - r_{2} \omega_{2}|^{d+sp}} \leq \frac{C}{| r_{1}-r_{2}|^{1+sp}r^{d-1}_{2}} \bigintsss_{B^{'}_{\lambda}(0)} \frac{1}{(1+ |\eta|^{2})^{\frac{d+sp}{2}}}   \frac{d \eta}{\sqrt{1-|\frac{\eta}{\lambda}|^{2}}}.
\end{equation}
Combining the inequalities \eqref{eqn001} and \eqref{eqn002}, we have for some constant $C=C(d, R_{2})>0$, 
\begin{multline}\label{estimate on surface integral}
     r^{d-1}_{1}  r^{d-1}_{2}  \int_{\mathbb{S}^{d-1}} \int_{\mathbb{S}^{d-1}} \frac{1}{|r_{1} \omega_{1} - r_{2} \omega_{2}|^{d+sp}} \,  d \omega_{1} \, d \omega_{2} \\ \leq \frac{C}{|r_{1}-r_{2}|^{1+sp}} \int_{B^{'}_{\lambda}(0)} \frac{1}{(1+ |\eta|^{2})^{\frac{d+sp}{2}}}   \frac{d \eta}{\sqrt{1-|\frac{\eta}{\lambda}|^{2}}}.
\end{multline}
It is sufficient to prove that there exists a constant $C>0$, independent of $\lambda$, such that
\begin{equation*}
     \bigintsss_{B^{'}_{\lambda}(0)} \frac{1}{(1+ |\eta|^{2})^{\frac{d+sp}{2}}}   \frac{d \eta}{\sqrt{1-|\frac{\eta}{\lambda}|^{2}}} \leq C.
\end{equation*}
If $\lambda \leq C$ for some positive constant $C$, then the aforementioned inequality is obvious. As $\lambda \to \infty$, using a change of variable in polar coordinates, we obtain
\begin{align*}
     \bigintsss_{B^{'}_{\lambda}(0)} \frac{1}{(1+ |\eta|^{2})^{\frac{d+sp}{2}}}   \frac{d \eta}{\sqrt{1-|\frac{\eta}{\lambda}|^{2}}} & = C \bigintsss_{0}^{\lambda} \frac{t^{d-2}}{(1+ t^{2})^{\frac{d+sp}{2}}}   \frac{d t}{\sqrt{1-\frac{t^{2}}{\lambda^{2}}}} \\ & \leq C  \bigintsss_{0}^{\lambda} \frac{t^{d-2}}{(1+ t^{2})^{\frac{d+sp}{2}}}   \frac{d t}{\sqrt{1-\frac{t}{\lambda}}}.
\end{align*}
Let $0 <a<1$, we have
\begin{align*}
    \bigintsss_{B^{'}_{\lambda}(0)} \frac{1}{(1+ |\eta|^{2})^{\frac{d+sp}{2}}}   \frac{d \eta}{\sqrt{1-|\frac{\eta}{\lambda}|^{2}}} & \leq C  \bigintsss_{0}^{\lambda^{a}} \frac{t^{d-2}}{(1+ t^{2})^{\frac{d+sp}{2}}}   \frac{d t}{\sqrt{1-\frac{t}{\lambda}}} \\ & \quad +  C  \bigintsss_{\lambda^{a}}^{\lambda} \frac{t^{d-2}}{(1+ t^{2})^{\frac{d+sp}{2}}}   \frac{d t}{\sqrt{1-\frac{t}{\lambda}}} =: I_{1}+I_{2}.
\end{align*}
For $I_{1}$, we have
\begin{align*}
  I_{1}  = C \bigintsss_{0}^{\lambda^{a}} \frac{t^{d-2}}{(1+ t^{2})^{\frac{d+sp}{2}}}   \frac{d t}{\sqrt{1-\frac{t}{\lambda}}} & \leq \frac{C}{\sqrt{1-\lambda^{a-1}}} \int_{0}^{\lambda^{a}} \frac{t^{d-2}}{(1+ t^{2})^{\frac{d+sp}{2}}} dt \\ &
  \leq  \frac{C}{\sqrt{1-\lambda^{a-1}}} \int_{0}^{\infty} \frac{t^{d-2}}{(1+ t^{2})^{\frac{d+sp}{2}}} dt = \frac{C}{\sqrt{1-\lambda^{a-1}}}.
\end{align*}
For $I_{2}$, we have
\begin{equation*}
    I_{2} = C \bigintsss_{\lambda^{a}}^{\lambda} \frac{t^{d-2}}{(1+ t^{2})^{\frac{d+sp}{2}}}   \frac{d t}{\sqrt{1-\frac{t}{\lambda}}} \leq \frac{C}{\lambda^{a(2+sp)}} \bigintsss_{\lambda^{a}}^{\lambda} \frac{d t}{\sqrt{1-\frac{t}{\lambda}}} = C \frac{\sqrt{1-\lambda^{a-1}}}{2\lambda^{a(2+sp)-1} }.
\end{equation*}
Choose $0<a<1$ such that $a(2+sp)-1 >0$. Then for large $\lambda$, we have
\begin{equation*}
     \bigintsss_{B^{'}_{\lambda}(0)} \frac{1}{(1+ |\eta|^{2})^{\frac{d+sp}{2}}}   \frac{d \eta}{\sqrt{1-|\frac{\eta}{\lambda}|^{2}}} \leq I_{1}+ I_{2} \leq C.
\end{equation*}
Combining \eqref{polar coordinate change} and \eqref{estimate on surface integral}, we have
\begin{equation*}
     \int_{A(R_{1}, R_{2})} \int_{A(R_{1}, R_{2})} \frac{|u(x)-u(y)|^{p}}{ |x-y|^{d+sp}} \, dx \, dy \leq C \int_{R_{1}}^{R_{2}} \int_{R_{1}}^{R_{2}} \frac{|\widetilde{u}(r_{1}) - \widetilde{u}(r_{2})|^{p}}{|r_{1}-r_{2}|^{1+sp}} \, dr_{1} \, dr_{2}.
\end{equation*}
This proves the lemma.
\end{proof}

\smallskip

\begin{remark}
    \normalfont We thank the referee for noting that the above lemma can alternatively be proved using Catalan’s formula \cite[Formula 1.59]{Samko} in \eqref{eqn001} and proceeding in a similar manner.  A weaker statement of the above lemma can be proved for any $u \in W^{s,p}(A(R_{1}, R_{2}))$, for some $R_{2}>R_{1} \geq 0$, by using interpolation theory \cite{bergh2012interpolation}. For any smooth open set $\Omega \subset \mathbb{R}^d, ~ W^{s,p}(\Omega)$ is an interpolation space between $W^{1,p}(\Omega)$ and  $L^{p}(\Omega)$. Let  $T: W^{1,p}((R_{1},R_{2})) \to W^{1,p}(A(R_{1}, R_{2}))$ be the operator defined by
\begin{equation*}
        T(\widetilde{u}) = u, \quad \text{where } u \text{ is the radial function given by } u(x) = \widetilde{u}(|x|).
\end{equation*}
    Then  $|\nabla u|^{p} = |\widetilde{u}'(|x|)|^{p}$, which implies that  $T$ is a bounded linear operator. Moreover,  $T: L^{p}((R_{1},R_{2})) \to L^{p}(A(R_{1}, R_{2}))$ is also a bounded linear operator. Therefore, by interpolation theory,  $T: W^{s,p}((R_{1},R_{2})) \to W^{s,p}(A(R_{1}, R_{2}))$ is a bounded linear operator. Hence, there exists a constant  $C=C(d,p,s,R_{1}, R_{2})>0$ such that
    \begin{equation*}
        \|u\|_{W^{s,p}(A(R_{1}, R_{2}))} \leq C \|\widetilde{u}\|_{W^{s,p}((R_{1},R_{2}))}.
    \end{equation*}
\end{remark}

\smallskip

The next theorem establishes the existence of a sequence of functions  $\{ u_{\epsilon} \}_{\epsilon>0}$ that converges to the constant function  $1$ in the $W^{s,p}$ norm in the case  $sp=1$. This sequence of functions will help us in proving the optimality of the weight function in the next subsection.

\begin{thm}\label{density theorem}
    Let  $sp=1$. Let  $B_{R}(0) \subset \mathbb{R}^d$, for some  $R>0$,  $ d \geq 1 $, and  $0 <  \epsilon < \frac{R}{2}$. Define  $u_{\epsilon} : B_{R}(0) \to \mathbb{R}$ by
\begin{equation*}
    u_{\epsilon}(x) = \begin{cases}
        1, & |x| \leq R - \epsilon  \\
        \frac{\ln{\left( \frac{2}{\epsilon} \right)}}{\ln{\left( \frac{2}{R-|x|} \right)}}, & R- \epsilon < |x| < R.
    \end{cases}
    \end{equation*}
    Then  $u_{\epsilon}$ converges to the constant function  $1$ in  $W^{s,p}$ norm as  $\epsilon \to 0$.
\end{thm}
\begin{proof}
     It is sufficient to prove for  $R=1$. We will first establish the case  $d=1$. For higher dimensions, the statement follows from Lemma \ref{d dim to 1 dim} with $R_{1}=0$ and $R_{2}=1$. For  $d=1$ and  $R=1$, the above sequence of functions becomes
     \begin{equation*}
    u_{\epsilon}(x) = \begin{cases}
        \frac{\ln{\left( \frac{2}{\epsilon} \right)}}{\ln{\left( \frac{2}{1+x} \right)}}, & x \in (-1,-1+ \epsilon)  \\
        1, & x \in [-1+ \epsilon , 1- \epsilon]  \\
        \frac{\ln{\left( \frac{2}{\epsilon} \right)}}{\ln{\left( \frac{2}{1-x} \right)}}, & x \in ( 1- \epsilon,  1).
    \end{cases}
    \end{equation*}
     We will show
     \begin{equation*}
         \|u_{\epsilon} -1\|_{W^{s,p}((-1,1))} \to 0 \quad  \text{as} \ \epsilon \to 0.
     \end{equation*}
     By the Lebesgue dominated convergence theorem,  $u_{\epsilon} \to 1$ in the $L^{p}$ norm as $\epsilon \to 0$. Since $[u_{\epsilon}-1]_{W^{s,p}((-1,1))} = [u_{\epsilon}]_{W^{s,p}((-1,1))}$, it is sufficient to show
     \begin{equation*}
         [u_{\epsilon}]_{W^{s,p}((-1,1))} \to 0  \quad \text{as} \ \epsilon \to 0.
     \end{equation*}
     Using the change of variables, definition and symmetry of  $u_{\epsilon}$, we have
 \begin{equation*}
     [u_{\epsilon}]^{p}_{W^{s,p}((-1,1))} = 2 [u_{\epsilon}]^{p}_{W^{s,p}((-1, -1+\epsilon))} + 2 \int_{-1}^{-1+\epsilon} \int_{-1+\epsilon}^{1-\epsilon} \frac{|u_{\epsilon}(x) - u_{\epsilon}(y)|^{p}}{|x - y|^{2}} \, dx \, dy =: J_{1} + J_{2}.
 \end{equation*}
 For  $J_{1}$, we have
 \begin{align*}
     J_{1} & = 2 \int_{-1}^{-1+ \epsilon}  \int_{-1}^{-1+ \epsilon} \frac{|u_{\epsilon}(x) - u_{\epsilon}(y)|^{p}}{|x - y|^{2}} \, dx \, dy \\ & = 2  \int_{0}^{\epsilon} \int_{0}^{\epsilon} \frac{|u_{\epsilon}(x-1) - u_{\epsilon}(y-1)|^{p}}{|x - y|^{2}} \, dx \, dy \\ & = 2 \ln^{p} \left( \frac{2}{\epsilon} \right) \bigintsss_{0}^{\epsilon} \bigintsss_{0}^{\epsilon} \frac{|\ln \left( \frac{2}{x} \right) - \ln \left( \frac{2}{y} \right)|^{p}}{ \ln^{p} \left( \frac{2}{x} \right) \ln^{p} \left( \frac{2}{y} \right) |x - y|^{2} } \, dx \, dy.
 \end{align*}
Again, using the change of variables  $\ln \left( \frac{2}{x} \right) = \xi$ and  $\ln \left( \frac{2}{y} \right) = \eta$ in  $J_{1}$, we have
\begin{align*}
    J_{1} & =  C \ln^{p} \left( \frac{2}{\epsilon} \right) \int_{\ln \left( \frac{2}{\epsilon} \right)}^{\infty}\int_{\ln \left( \frac{2}{\epsilon} \right)}^{\infty} \frac{|\xi - \eta|^{p}}{\xi^{p} \eta^{p} |e^{- \xi} - e^{- \eta}|^{2} } e^{-\xi} e^{-\eta} \, d \xi \, d \eta \\ & =  C \ln^{p} \left( \frac{2}{\epsilon} \right) \int_{\ln \left( \frac{2}{\epsilon} \right)}^{\infty}\int_{\ln \left( \frac{2}{\epsilon} \right)}^{\infty} \frac{|\xi - \eta|^{p}}{\xi^{p} \eta^{p}(e^{-2 \xi} + e^{-2\eta} - 2 e^{-(\xi+ \eta)})} e^{-(\xi+\eta)} \, d \xi \, d \eta  \\ & = C  \ln^{p} \left( \frac{2}{\epsilon} \right) \int_{\ln \left( \frac{2}{\epsilon} \right)}^{\infty} \frac{1}{\eta^{p}} \int_{\ln \left( \frac{2}{\epsilon} \right)}^{\infty} \frac{|\xi - \eta|^{p}}{\xi^{p} (e^{\xi - \eta} + e^{\eta - \xi} - 2 )} \, d \xi \, d \eta .
\end{align*}
Since  $\xi \geq \ln \left( \frac{2}{\epsilon} \right) $, we have
\begin{align*}
    J_{1} & \leq C \int_{\ln \left( \frac{2}{\epsilon} \right)}^{\infty} \frac{1}{\eta^{p}} \int_{\ln \left( \frac{2}{\epsilon} \right)}^{\infty} \frac{|\xi - \eta|^{p}}{ (e^{\xi - \eta} + e^{\eta - \xi} - 2 )} \, d \xi \, d \eta \leq C \int_{\ln \left( \frac{2}{\epsilon} \right)}^{\infty} \frac{d \eta}{\eta^{p}} \int_{- \infty}^{\infty} Q(t) \, dt \\ & \leq \frac{C}{\ln^{p-1} \left( \frac{2}{\epsilon} \right)} \int_{- \infty}^{\infty} Q(t) \, dt,
\end{align*}
where  $Q(t) = \frac{|t|^{p}}{e^{t}+ e^{-t}-2} $, which is even function as  $p>1$. For  $|t| \leq 1$, we have  $  Q(t) \leq C \frac{|t|^{p}}{|t|^{2}} = C |t|^{p-2}$, and for  $|t|>1$, we have  $Q(t) \leq C |t|^{p} e^{-|t|}$. Therefore, 
\begin{equation*}
    \int_{-\infty}^{\infty} Q(t) \, dt = \int_{|t| \leq 1} Q(t) \, dt + \int_{|t|>1} Q(t) \, dt < \infty.
\end{equation*}
Hence,
\begin{equation*}
    J_{1} \leq \frac{C}{\ln^{p-1} \left( \frac{2}{\epsilon} \right)}.
\end{equation*}
Now for  $J_{2}$, we have
\begin{align*}
    J_{2} & = 2 \bigintsss_{-1}^{-1+\epsilon} \Bigg|\frac{\ln \left( \frac{2}{\epsilon} \right)} { \ln \left( \frac{2}{1+x} \right) }  -1  \Bigg|^{p} \left( \int_{-1+\epsilon}^{1- \epsilon} \frac{dy}{|x-y|^{2}} \right) \, dx \\ & =  \bigintsss_{-1}^{-1+\epsilon} \Bigg|\frac{\ln \left( \frac{2}{\epsilon} \right)} { \ln \left( \frac{2}{1+x} \right) } -1   \Bigg|^{p} \left( \frac{1}{-1+\epsilon - x} - \frac{1}{1-\epsilon - x} \right) \, dx \\ & \leq C \bigintsss_{-1}^{-1+\epsilon} \frac{| \ln \left( \frac{2}{1+x} \right)-\ln \left( \frac{2}{\epsilon} \right)|^{p}} {  \ln^{p} \left( \frac{2}{1+x} \right) (-1+ \epsilon - x)} \, dx  = C \bigintsss_{0}^{\epsilon} \frac{| \ln \left( \frac{2}{x} \right)-\ln \left( \frac{2}{\epsilon} \right)|^{p}} {  \ln^{p} \left( \frac{2}{x} \right) ( \epsilon - x)} \, dx .
\end{align*}
Using the change of variable  $t = \frac{\ln \left( \frac{2}{x} \right)}{ \ln \left( \frac{2}{\epsilon} \right) } - 1$ in the above inequality, we obtain
\begin{equation*}
    J_{2} \leq  C \ln \left( \frac{2}{\epsilon} \right) \bigintsss_{0}^{\infty} \frac{t^{p}}{(t+1)^{p} \left( \frac{\epsilon e^{(t+1)  \ln \left( \frac{2}{\epsilon} \right)}}{2} -1 \right) } \, dt \leq C \ln \left( \frac{2}{\epsilon} \right) \bigintsss_{0}^{\infty} \frac{t^{p}}{ e^{t  \ln \left( \frac{2}{\epsilon} \right)}  -1} \, dt.
\end{equation*}
Using the further change of variable  $\eta= t \ln \left( \frac{2}{\epsilon} \right)$, we obtain
\begin{equation*}
    J_{2} \leq  \frac{C}{\ln^{p} \left( \frac{2}{\epsilon} \right)} \int_{0}^{\infty} \frac{\eta^{p}}{e^{\eta}-1}  \, d\eta \leq \frac{C}{\ln^{p} \left( \frac{2}{\epsilon} \right)} .
\end{equation*}
Here, we have used the fact that  $\frac{\eta^{p}}{e^{\eta}-1} \in L^{1}(0, \infty)$. Combining the estimates for  $J_{1}$ and  $J_{2}$, we obtain
\begin{equation}\label{limit of sequence for d=1}
     [u_{\epsilon}]^{p}_{W^{s,p}((-1,1))}  \leq  \frac{C}{\ln^{p-1} \left( \frac{2}{\epsilon} \right)} + \frac{C}{\ln^{p} \left( \frac{2}{\epsilon} \right)} \leq  \frac{C}{\ln^{p-1} \left( \frac{2}{\epsilon} \right)}.
\end{equation}
Therefore,  $u_{\epsilon} $ converges to the constant function  $1$ in  $W^{s,p}$ norm as  $\epsilon \to 0$ for  $d=1$. For  $d>1$, using Lemma  \ref{d dim to 1 dim} with $R_{1}=0$ and  $R_{2}=1$,  together with the radial symmetry of  $u_{\epsilon}$, we obtain
\begin{equation}\label{limit of sequence for d>1}
    [u_{\epsilon}]^{p}_{W^{s,p}(B_{1}(0))} \leq C  [u_{\epsilon}]^{p}_{W^{s,p}((-1,1))} \leq  \frac{C}{\ln^{p-1} \left( \frac{2}{\epsilon} \right)},
\end{equation}
which implies  $u_{\epsilon} $ converges to the constant function  $1$ in  $W^{s,p}$ norm as  $\epsilon \to 0$.
\end{proof}

\smallskip

\subsection{Proof of the optimality of the weight function in Theorem \ref{mainresult} for general bounded Lipschitz domains}\label{opt proof for gen domain} To prove the optimality of the weight function in Theorem \ref{mainresult} among the class of functions depending on $\delta_{\Omega}$ for general bounded Lipschitz domains, it is sufficient to establish this for flat boundary case. Therefore, we assume the domain is $\Omega= (0,1)^{d}$ and prove the optimality of weight function in Lemma  \ref{case 1 flat case}. We define a function  $v_{\epsilon}$ on  $(0,1)$ as
\begin{equation}\label{eqn u epsilon}
    v_{\epsilon}(x) = \begin{cases}
    \frac{\ln\left( \frac{2}{\epsilon} \right)}{\ln \left( \frac{2}{x} \right)}, & 0< x < \epsilon \\
        1, & \epsilon \leq x \leq 1-\epsilon \\
        \frac{\ln\left( \frac{2}{\epsilon} \right)}{\ln \left( \frac{2}{1-x} \right)}, & 1-\epsilon< x < 1.
    \end{cases}
\end{equation} 
Clearly, $v_{\epsilon}$ is a translation of the function $u_{\epsilon}$ with $R= \frac{1}{2}$, which was defined in the previous subsection for the case $d=1$. Therefore, from Theorem  \ref{density theorem} with  $sp=1$, it is easy to deduce that 
\begin{equation}\label{opt eqn 1}
    [v_{\epsilon}]^{p}_{W^{s,p}((0,1))} \leq \frac{C}{\ln^{p-1} \left( \frac{2}{\epsilon} \right)}.
\end{equation}
We will define a function  $\phi_{\mu}$ on $(0,1)$ for some  $\mu > 0$ that has compact support in  $(0,1)$ and belongs to  $W^{s,p}_{0}((0,1))$. Therefore, Theorem  \ref{mainresult} will hold for  $\phi_{\mu}$. The following lemma is essential in establishing the optimality of the weight function in Theorem  \ref{mainresult} for a general bounded Lipschitz domain.

\begin{lemma}\label{lemma on opt}
    Let $sp=1$, $\tau>1$, and let $\phi_{\mu} \in W^{s,p}_{0}((0,1)), ~ 0<\mu<1$  be the function defined by 
    \begin{equation}
     \phi_{\mu}(x) = \begin{cases}
    v_{\epsilon} \left( \frac{x- \mu}{1-2 \mu} \right), & \mu< x < 1- \mu \\
        0, & x \in (0, \mu) \cup (1- \mu ,1). 
    \end{cases}
\end{equation}
Then
\begin{equation*}
    \lim_{\mu \to 0} [\phi_{\mu}]^{p}_{W^{s,p}((0,1))} = [v_{\epsilon}]^{p}_{W^{s,p}((0,1))},
\end{equation*}
and
\begin{equation*}
    \lim_{\mu \to 0} \int_{0}^{1} \frac{|\phi_{\mu}(x)- (\phi_{\mu})_{(0,1)}|^{\tau}}{x \ln^{\tau} \left( \frac{2}{x} \right)} \, dx = \int_{0}^{1} \frac{|v_{\epsilon}(x) - (v_{\epsilon})_{(0,1)}|^{\tau}}{x \ln^{\tau} \left( \frac{2}{x} \right)} \, dx.
\end{equation*}
\end{lemma}
\begin{proof}
    Let us calculate the Gagliardo seminorm of  $\phi_{\mu}$,
\begin{multline}\label{eqn phi mu}
   [\phi_{\mu}]^{p}_{W^{s,p}((0,1))} = [\phi_{\mu}]^{p}_{W^{s,p}((\mu, 1- \mu))} + 2  \int_{\mu}^{1-\mu} |\phi_{\mu}(x)|^{p} \int_{0}^{\mu} \frac{1}{|x-y|^{2}} \, dx \, dy \\ + 2  \int_{\mu}^{1-\mu} |\phi_{\mu}(x)|^{p} \int_{1- \mu}^{1} \frac{1}{|x-y|^{2}} \, dx \, dy =: [\phi_{\mu}]^{p}_{W^{s,p}((\mu, 1- \mu))}  + L_{1}+ L_{2} .
\end{multline}
For  $L_{1}$, using the change of variable  $\frac{x- \mu}{1- 2 \mu}=z$, we obtain
\begin{align*}
    L_{1} & = 2  \int_{\mu}^{1-\mu} |\phi_{\mu}(x)|^{p} \int_{0}^{\mu} \frac{1}{|x-y|^{2}} \, dx \, dy \\ & = \int_{\mu}^{1- \mu} \left| v_{\epsilon} \left( \frac{x- \mu}{1-2 \mu} \right) \right|^{p} \left\{ \frac{1}{x- \mu} - \frac{1}{x}   \right\} \, dx \\ & = 2(1- 2 \mu)  \int_{0}^{1} |v_{\epsilon}(z)|^{p} \left\{ \frac{1}{(1-2\mu)z} - \frac{1}{(1-2 \mu)z+ \mu}  \right\} \, dz.
\end{align*} 
Since  $(1- 2 \mu)|v_{\epsilon}(z)|^{p} \left\{ \frac{1}{(1-2\mu)z} - \frac{1}{(1-2 \mu)z+ \mu}  \right\} \leq \frac{|v_{\epsilon}(z)|^{p}}{z} $, and 
\begin{align*}
    \int_{0}^{1} \frac{|v_{\epsilon}(z)|^{p}}{z} \, dz & = \ln^{p} \left( \frac{2}{\epsilon} \right) \int_{0}^{\epsilon} \frac{1}{z \ln^{p} \left( \frac{2}{z} \right)} \, dz + \int_{\epsilon}^{1-\epsilon} \frac{1}{z} dz + \ln^{p} \left( \frac{2}{\epsilon} \right) \int_{1- \epsilon}^{1} \frac{1}{z \ln^{p} \left( \frac{2}{1-z} \right)} \, dz \\ & \leq  \frac{\ln \left( \frac{2}{\epsilon} \right)}{p-1} + \ln \left( \frac{1- \epsilon}{\epsilon} \right) + C \ln^{p} \left( \frac{2}{\epsilon} \right) \int_{1- \epsilon}^{1} \frac{1}{(1-z)\ln^{p} \left( \frac{2}{1-z} \right)} \, dz \\ & \leq C \frac{\ln \left( \frac{2}{\epsilon} \right)}{p-1} + \ln \left( \frac{1- \epsilon}{\epsilon} \right) < \infty,
\end{align*}
this implies that  $\frac{|v_{\epsilon}(z)|^{p}}{z} \in L^{1}((0,1))$. Therefore, using dominated convergence theorem, we have  $L_{1} \to 0$ as  $\mu \to 0$. Similarly,  $L_{2} \to 0$ as  $\mu \to 0$. Moreover, using the change of variables $\frac{x- \mu}{1- 2 \mu}=z_{1}$ and $\frac{y- \mu}{1- 2 \mu}=z_{2}$, we obtain
\begin{equation*}
[\phi_{\mu}]^{p}_{W^{s,p}((\mu, 1- \mu))} =  (1-2 \mu)^{1-sp} [v_{\epsilon}]^{p}_{W^{s,p}((0,1))}.
\end{equation*}
Therefore, taking the limit  $\mu \to 0$ in \eqref{eqn phi mu}, we have
\begin{equation*}
     \lim_{\mu \to 0} [\phi_{\mu}]^{p}_{W^{s,p}((0,1))} = [v_{\epsilon}]^{p}_{W^{s,p}((0,1))}.
\end{equation*}
We observe that
\begin{equation*}
(\phi_{\mu})_{(0,1)} = \int_{\mu}^{1-\mu}  v_{\epsilon} \left( \frac{x- \mu}{1-2 \mu} \right) \,  dx = (1-2 \mu) \int_{0}^{1} v_{\epsilon}(x) \, dx = (1- 2 \mu)(v_{\epsilon})_{(0,1)} .
\end{equation*}
Since  $\frac{1}{x \ln^{\tau} \left( \frac{2}{x} \right) } \in L^{1}((0,1))$, the dominated convergence theorem yields
\begin{equation*}
    \lim_{\mu \to 0} \int_{0}^{1} \frac{|\phi_{\mu}(x)- (\phi_{\mu})_{(0,1)}|^{\tau}}{x \ln^{\tau} \left( \frac{2}{x} \right)} \, dx = \int_{0}^{1} \frac{|v_{\epsilon}(x) - (v_{\epsilon})_{(0,1)}|^{\tau}}{x \ln^{\tau} \left( \frac{2}{x} \right)} \, dx.
\end{equation*}
This completes the proof.
\end{proof}

The function $\phi_{\mu}$ defined in the above lemma has compact support in $(0,1)$ and a finite Gagliardo seminorm, which implies that  $\phi_{\mu} \in W^{s,p}_{0}((0,1))$.  Moreover, $\phi_{\mu} \to 1$ in $L^{p}((0,1))$ as $\epsilon, \, \mu \to 0$. Using Lemma~\ref{lemma on opt} and inequality~\eqref{opt eqn 1} with $sp = 1$, we have
\begin{equation*}
    \lim_{\epsilon \to 0}   \lim_{\mu \to 0} [\phi_{\mu}]^{p}_{W^{s,p}((0,1))} = \lim_{\epsilon \to 0} [v_{\epsilon}]^{p}_{W^{s,p}((0,1))} = 0.
\end{equation*}
Since $\phi_{\mu} \in W^{s,p}_{0}((0,1))$ for  $\mu, \, \epsilon >0$ and   $W^{s,p}_{0}((0,1))$ is a closed subspace of  $W^{s,p}((0,1))$, the constant function  $1 \in W^{s,p}_{0}((0,1))$. Dyda and Kijaczko in  \cite[Lemma 13]{dyda2022} proved that  $W^{s,p}(\Omega)= W^{s,p}_{0}(\Omega)$ if and only if the constant function  $1 \in W^{s,p}_{0}(\Omega)$. In fact, for the sake of completeness of this article, we present a proof of this result in Lemma  \ref{density lemma appen} (See Appendix  \ref{appen}). Therefore,  $W^{s,p}((0,1)) = W^{s,p}_{0}((0,1))$ in the case  $sp=1$. The next lemma extends this result to a general bounded Lipschitz domain for the case  $sp=1$. 

\smallskip

\begin{lemma}\label{lemma on Wsp}
    Let $\Omega$ be a bounded Lipschitz domain in $\mathbb{R}^{d}$, and let   $p>1, ~ s \in (0,1)$ be such that  $sp=1$. Then  $W^{s,p}_{0}(\Omega) = W^{s,p}(\Omega)$.
\end{lemma}
\begin{proof}
    It is sufficient to prove the result for the flat boundary case; for a general bounded Lipschitz domain, it follows from the usual patching argument. Define a function  $\Phi$ on  $(0,1)^{d}$ by $\Phi(x) = \phi_{\mu}(x_{1}) \cdots \phi_{\mu}(x_{d}) $. Then  $\Phi \in W^{s,p}_{0}((0,1)^{d})$. Moreover, using Subsection \ref{slicing lemma} (See Appendix  \ref{appen}), we obtain
 \begin{equation*}
     [\Phi]^{p}_{W^{s,p}((0,1)^{d})} \leq C [\phi_{\mu}]^{p}_{W^{s,p}((0,1))}.
 \end{equation*}
 From Lemma \ref{lemma on opt} and inequality \eqref{opt eqn 1}, we conclude that  $[\Phi]^{p}_{W^{s,p}((0,1)^{d})} \to 0$ as  $\epsilon,  \mu \to 0$ when  $sp=1$. Furthermore,  $\Phi \to 1$ in $L^{p}((0,1)^{d})$ as  $\epsilon,  \mu \to 0$. Therefore, the constant function  $1$ belongs to $W^{s,p}_{0}((0,1)^{d})$ in the case  $sp=1$. By Lemma \ref{density lemma appen} (See Appendix  \ref{appen}), we conclude that $W^{s,p}_{0}((0,1)^{d}) = W^{s,p}((0,1)^{d})$ when  $sp=1$. This completes the proof of the lemma.
\end{proof}

The following lemma establishes the optimality of the weight function in Theorem \ref{mainresult} for a general bounded Lipschitz domain in the case $\tau \geq p$ when $d=1$, and $\tau =p$ with $d>1$, among the class of functions depending on $\delta_{\Omega}$. The previous two lemmas, Lemma \ref{lemma on opt} and Lemma \ref{lemma on Wsp}, serve as the key ingredients for proving the next lemma.

\begin{lemma}\label{lemma proof of opt gen domain}
Let $\Omega$ be a bounded Lipschitz domain in $\mathbb{R}^{d}$. Then the weight function in Theorem \ref{mainresult} is optimal, among the class of functions depending on $\delta_{\Omega}$, for $\tau \geq p$ when $d=1$ and for $\tau =p$ when $d>1$.
\end{lemma}
\begin{proof}
To prove the optimality of the weight function in Theorem \ref{mainresult} among the class of functions depending on $\delta_{\Omega}$, it suffices to establish this for the case of a flat boundary, as presented in Lemma \ref{case 1 flat case}. The conclusion for a general bounded Lipschitz domain then follows from a patching argument. Consider the function $\widetilde{\Phi}_{\epsilon, \mu}$ defined on $(0,1)^{d}$ by
    \begin{equation}\label{Defn Phi (epsilon, mu)}
        \widetilde{\Phi}_{\epsilon, \mu}(x',x_{d}) = \phi_{\mu}(x_{d}),
    \end{equation}
where $\phi_{\mu}$ is as defined in Lemma \ref{lemma on opt}.    Clearly, $\widetilde{\Phi}_{\epsilon, \mu} \in W^{s,p}((0,1)^{d})$. By applying Lemma \ref{lemma on Wsp}, we obtain $ \widetilde{\Phi}_{\epsilon, \mu} \in W^{s,p}_{0}((0,1)^{d})$ in the case $sp=1$. Therefore, Lemma \ref{case 1 flat case} applies to $\widetilde{\Phi}_{\epsilon,\mu}$, and together with \eqref{poincare}, for $\tau \geq p$ when $d=1$ and $\tau=p$ when $d>1$, we obtain
\begin{align*}
    \left(  \bigintsss_{(0,1)^{d}} \frac{ |\widetilde{\Phi}_{\epsilon, \mu}(x) - (\widetilde{\Phi}_{\epsilon, \mu})_{(0,1)^{d}}|^{\tau}}{x_{d} \ln^{\tau} \left(\frac{2}{x_{d}} \right)} \, dx \right)^{\frac{1}{\tau}}  \leq  C  [\widetilde{\Phi}_{\epsilon, \mu}]_{W^{s,p}((0,1)^{d})}.
\end{align*}
Now consider the function $v_{\epsilon}$ defined in \eqref{eqn u epsilon}. We compute
\begin{equation*}
(v_{\epsilon})_{(0,1)} = \int_{0}^{1}  v_{\epsilon}(y) \,  dy = 2 \ln \left( \frac{2}{\epsilon} \right) \bigintssss_{0}^{\epsilon} \frac{1}{\ln \left( \frac{2}{y} \right)} \, dy + (1- 2\epsilon) = 1+ O \left( \epsilon \right).
\end{equation*}
Therefore, for $\tau>1$,
\begin{align*}
    \bigintssss_{0}^{1} \frac{ |v_{\epsilon}(x_{d})- (v_{\epsilon})_{(0,1)}|^{\tau}}{x_{d} \ln^{\tau} \left( \frac{2}{x_{d}} \right)}  \, dx_{d} &  =  \bigintssss_{0}^{1} \frac{|v_{\epsilon}(x_{d})- 1+ O \left( \epsilon  \right)|^{\tau}}{x_{d} \ln^{\tau} \left( \frac{2}{x_{d}} \right)} \, dx_{d} \\ &  = C  \bigintssss_{0}^{1} \frac{|v_{\epsilon}(x_{d})-1|^{\tau}}{x_{d} \ln^{\tau} \left( \frac{2}{x_{d}} \right)} \, dx_{d} + O \left( \epsilon  \right) \\ &  =  C   \bigintssss_{0}^{\epsilon} \frac{ \left| \ln \left( \frac{2}{\epsilon} \right)- \ln \left( \frac{2}{x_{d}} \right) \right|^{\tau}}{x_{d} \ln^{2 \tau} \left( \frac{2}{x_{d}} \right)} \, dx_{d} + O \left( \epsilon \right).
\end{align*}
Using the change of variable $\frac{\ln \left( \frac{2}{x_{d}} \right)}{ \ln \left( \frac{2}{\epsilon} \right)} = t$ and the fact that $\int_{1}^{\infty} \frac{(t-1)^{\tau}}{t^{2 \tau}} dt < \infty$, we obtain
\begin{align}\label{eqn opt 1}
 \bigintsss_{0}^{1} \frac{ |v_{\epsilon}(x_{d})- (v_{\epsilon})_{(0,1)}|^{\tau}}{x_{d} \ln^{\tau} \left( \frac{2}{x_{d}} \right)} \, dx_{d}  & =   \frac{C}{\ln^{\tau -1} \left( \frac{2}{\epsilon} \right)} \int_{1}^{\infty} \frac{(t-1)^{\tau}}{t^{2 \tau}} \, dt + O \left( \epsilon \right) \nonumber \\ & =   \frac{C}{\ln^{\tau-1} \left( \frac{2}{\epsilon} \right)} + O \left( \epsilon \right).
\end{align}
Therefore, from Lemma \ref{lemma on opt}, we obtain
\begin{align*}
  \lim_{\mu \to 0}   \bigintssss_{(0,1)^{d}} \frac{ |\widetilde{\Phi}_{\epsilon, \mu}(x) - (\widetilde{\Phi}_{\epsilon, \mu})_{(0,1)^{d}}|^{\tau}}{x_{d} \ln^{\tau} \left(\frac{2}{x_{d}} \right)} \, dx &  =   \lim_{\mu \to 0} \bigintssss_{0}^{1} \frac{ |\phi_{\mu}(x_{d}) - (\phi_{\mu})_{(0,1)}|^{\tau}}{x_{d} \ln^{\tau} \left(\frac{2}{x_{d}} \right)} \, dx_{d} \\ & =  \bigintssss_{0}^{1} \frac{ |v_{\epsilon}(x_{d})- (v_{\epsilon})_{(0,1)}|^{\tau}}{x_{d} \ln^{\tau} \left( \frac{2}{x_{d}} \right)} \, dx_{d} = \frac{C}{\ln^{\tau-1} \left( \frac{2}{\epsilon} \right)} + O \left( \epsilon \right).
\end{align*}
Now suppose that there exists a function  $f$ with the property that  $|f(x_{d})| \to \infty$ as $x_{d} \to 0$. Additionally, assume that there is a further improvement of the inequality in Lemma \ref{case 1 flat case} for $d=1$ and $\tau \geq p$, and for $d>1$ with $\tau = p$. Consequently, the same improvement also applies to \eqref{Theorem 2 Remark} with $sp=1$, $\Omega= (0,1)^{d}$, and $\delta_{\Omega}(x) = x_{d}$. More precisely, we assume that the following inequality holds.
\begin{equation}\label{optcond1}
      \left(  \bigintsss_{(0,1)^{d}} \frac{|f(x_{d})| |u(x)-(u)_{(0,1)^{d}}|^{\tau}}{x_{d} \ln^{\tau} \left(\frac{2}{x_{d}} \right)} \, dx \right)^{\frac{1}{\tau}} \leq  C  [u]_{W^{s,p}((0,1)^{d})}, \quad \forall \ u \in W^{s,p}_{0}((0,1)^{d})  .
\end{equation}
Consider the above sequence of functions $\widetilde{\Phi}_{\epsilon,\mu}$. Using a similar computation as above, we have
\begin{equation}\label{weightconv1}
  \lim_{\mu \to 0}  \bigintsss_{(0,1)^{d})} \frac{|f(x_{d})| | \widetilde{\Phi}_{\epsilon, \mu}(x)- (\widetilde{\Phi}_{\epsilon, \mu})_{(0,1)^{d}}|^{\tau}}{x_{d} \ln^{\tau} \left(\frac{2}{x_{d}} \right)} \, dx \geq C \frac{|f(x_{d,\epsilon})|}{\ln^{\tau - 1}\left( \frac{2}{\epsilon} \right) },
\end{equation}
where  $x_{d,\epsilon} = \epsilon$. First assume that $d=1$ and $\tau \geq p$. Note that in this case $\widetilde{\Phi}_{\epsilon, \mu}(x_{d})= \phi_{\mu}(x_{d})$. Let  $\tau_{1} > \tau > p$. Then there exists $\theta \in (0,1)$ such that  $\tau= \theta p + (1-\theta)\tau_{1}$. Using H$\ddot{\text{o}}$lder's inequality with  $\frac{1}{1/ \theta}+ \frac{1}{1/(1-\theta)}=1$, we obtain
\begin{align*}
   & \bigintsss_{0}^{1}  \frac{ |f(x_{d})| |\phi_{\mu}(x_{d})-(\phi_{\mu})_{(0,1)}|^{\tau}}{x_{d} \ln^{\tau} \left( \frac{2}{x_{d}} \right) } \, dx_{d} =  \bigintsss_{0}^{1} \frac{|f(x_{d})|^{\theta  + (1-\theta)} |\phi_{\mu}(x_{d})-(\phi_{\mu})_{(0,1)}|^{\theta p + (1-\theta)\tau_{1}}}{x^{\theta  + (1- \theta)}_{d} \ln^{\theta p + (1-\theta)\tau_{1}} \left( \frac{2}{x_{d}} \right) } \, dx_{d} \\  & \leq \left( \bigintsss_{0}^{1} \frac{ |f(x_{d})| |\phi_{\mu}(x_{d})-(\phi_{\mu})_{(0,1)}|^{p}}{x_{d} \ln^{p} \left( \frac{2}{x_{d}} \right) } \, dx_{d} \right)^{\theta} \left( \bigintsss_{0}^{1} \frac{ |f(x_{d})| |\phi_{\mu}(x_{d})-(\phi_{\mu})_{(0,1)}|^{\tau_{1}}}{x_{d} \ln^{\tau_{1}} \left( \frac{2}{x_{d}} \right) } \, dx_{d} \right)^{1-\theta}.
\end{align*}
Therefore, from the assumed inequality \eqref{optcond1}, we have
\begin{equation*}
     \bigintsss_{0}^{1} \frac{ |f(x_{d})| |\phi_{\mu}(x_{d})- (\phi_{\mu})_{(0,1)}|^{\tau}}{x_{d} \ln^{\tau} \left( \frac{2}{x_{d}} \right) } \, dx_{d} \leq C[\phi_{\mu}]^{\theta p}_{W^{s,p}((0,1))} [\phi_{\mu}]^{(1-\theta) \tau_{1}}_{W^{s_{1},\tau_{1}}((0,1))} ,
\end{equation*}
where $s_{1} \in (0,1)$ such that $s_{1} \tau_{1} = 1$. Applying Lemma \ref{lemma on opt}, and using \eqref{opt eqn 1} and \eqref{weightconv1}, we get
\begin{align*}
     \frac{C |f(x_{d,\epsilon})| }{\ln^{\tau-1} \left( \frac{2}{\epsilon} \right) } & \leq \lim_{\mu \to 0} \bigintsss_{0}^{1} \frac{|f(x)||\phi_{\mu}(x_{d})-(\phi_{\mu})_{(0,1)}|^{\tau}}{x_{d} \ln^{\tau} \left( \frac{2}{x_{d}} \right) } \, dx_{d} \\ & \leq C[v_{\epsilon}]^{\theta p}_{W^{s,p}((0,1))} [v_{\epsilon}]^{(1-\theta) \tau_{1}}_{W^{s_{1},\tau_{1}}((0,1))} \\ & \leq C \left( \frac{1}{\ln^{(p-1) \theta} \left( \frac{2}{\epsilon} \right) } \right) \left( \frac{1}{\ln^{(\tau_{1}-1)(1- \theta)} \left( \frac{2}{\epsilon} \right) } \right).
\end{align*}
Hence, 
\begin{equation*}
    \frac{\ln^{(\tau_{1}-1)(1- \theta)} \left( \frac{2}{\epsilon} \right)  \ln^{(p-1) \theta} \left( \frac{2}{\epsilon}\right) |f(x_{d,\epsilon})| }{\ln^{\tau-1} \left( \frac{2}{\epsilon} \right)} \leq C.
\end{equation*}
Using the fact that $\tau=\theta p + (1-\theta) \tau_{1}$, we obtain 
\begin{equation*}
    (\tau_{1}-1)(1-\theta) +(p-1)\theta -(\tau-1)=0.
\end{equation*}
Therefore, 
\begin{equation*}
    |f(x_{d,\epsilon})| \leq C.
\end{equation*}
This cannot be true as  $|f(x_{d,\epsilon})| \to \infty$ as  $\epsilon \to 0$. This proves the optimality of the weight function for  $d=1$ and  $\tau \geq p$ in Lemma \ref{case 1 flat case}. Now assume that $d > 1$ and $\tau = p$. From \eqref{opt eqn 1} and \eqref{weightconv1}, Lemma \ref{lemma on opt}, together with Subsection \ref{slicing lemma} (See Appendix  \ref{appen}), we have
\begin{align*}
  \frac{C |f(x_{d,\epsilon})| }{\ln^{p-1} \left( \frac{2}{\epsilon} \right) } & \leq \lim_{\mu \to 0} \bigintsss_{(0,1)^{d}} \frac{|f(x_{d})| |\widetilde{\Phi}_{\epsilon, \mu}(x)-(\widetilde{\Phi}_{\epsilon, \mu})_{(0,1)^{d}}|^{p}}{x_{d} \ln^{p} \left(\frac{2}{x_{d}} \right)} \,  dx  \\ & = \lim_{\mu \to 0} \bigintsss_{0}^{1}  \frac{ |f(x_{d})| |\phi_{\mu}(x_{d})-(\phi_{\mu})_{(0,1)}|^{p}}{x_{d} \ln^{p} \left( \frac{2}{x_{d}} \right) } \, dx_{d} \\ & \leq C \lim_{\mu \to 0} [\widetilde{\Phi}_{\epsilon, \mu}]^{p}_{W^{s,p}((0,1)^{d})}  \leq C \lim_{\mu \to 0} [\phi_{\mu}]^{p}_{W^{s,p}((0,1))} \leq    \frac{C}{\ln^{p-1} \left( \frac{2}{\epsilon} \right)}.
\end{align*}
Therefore, we have
\begin{equation*}
   | f(x_{d,\epsilon})| \leq C.
\end{equation*}
This is not possible as  $|f(x_{d,\epsilon})| \to \infty$ as  $\epsilon \to 0$. This proves the optimality of the weight function for $d>1$ and  $\tau=p$ in Lemma \ref{case 1 flat case}. Therefore, by the patching argument, we deduce the optimality of the weight function in Theorem \ref{mainresult} among the class of functions depending on $\delta_{\Omega}$ for a general bounded Lipschitz domain in the case $\tau \geq p$ when $d=1$ and $\tau =p$ when $d>1$.
\end{proof}

\smallskip

\subsection{Optimality of the weight function in Theorem \ref{mainresult}: 
an illustrative example of the unit ball \texorpdfstring{$\Omega= B_{1}(0)$}{ball of radius 1}}\label{opt proof for a ball}

%\subsection{An illustrative example: the unit ball \texorpdfstring{$\Omega= B_{1}(0)$}{ball of radius 1}}\label{opt proof for a ball}
%\textcolor{red}{To establish the optimality of the weight function in Theorem \ref{mainresult}, among the class of functions depending on $\delta_{\Omega}$ when $\Omega= B_{1}(0)$, it is sufficient to establish the optimality of the weight function in the fractional Hardy-type inequality in \eqref{Theorem 2 Remark}.}

In this subsection, we illustrate the optimality of the weight function in Theorem~\ref{mainresult}
among the class of functions depending on $\delta_{\Omega}$ by considering the unit ball
$\Omega = B_{1}(0)$. Although the optimality has already been established for general bounded
Lipschitz domains in the previous subsection, we present this example in order to understand
the optimality more precisely in the case of the unit ball.

%In this subsection, we will provide a comprehensive analysis to establish the optimality of weight function in Theorem \ref{mainresult}, among the class of function depending on $\delta_{\Omega}$ by giving on a unit ball $\Omega = B_{1}(0)$. While we already proved the optimality of weight function in Theorem \ref{mainresult} for any bounded Lipschitz domains in the previous subsection. Here, we considered a particular domain, a unit ball and provide 

\smallskip

Let us consider the sequence of functions $\{ u_{\epsilon} \}_{\epsilon>0}$ defined in Theorem~\ref{density theorem} with $R=1$. Using the Bernoulli inequality, we obtain
\begin{align*}
    (u_{\epsilon})_{B_{1}(0)} = \frac{1}{|B_{1}(0)|} \int_{B_{1}(0)} u_{\epsilon}(x) \,  dx & = \frac{1}{|B_{1}(0)|} \left(  \int_{|x| \leq 1 - \epsilon} \, dx + \int_{1- \epsilon<|x|<1}   \frac{\ln{\left( \frac{2}{\epsilon} \right)}}{\ln{\left( \frac{2}{1-|x|} \right)}}  \, dx   \right) \\ & \geq (1-\epsilon)^{d} \geq 1 - d \epsilon .
\end{align*} 
On the other hand,  $(u_{\varepsilon})_{B_{1}(0)}\le 1$. Therefore,
\begin{equation*}
     (u_{\epsilon})_{B_{1}(0)} = 1 + O(\epsilon).
\end{equation*}
Here,  $O$ is the notation for big  $O$. For  $\Omega= B_{1}(0)$, we have  $\delta_{\Omega}(x) = 1-|x|$. Therefore, we have, for $\tau>1$,
\begin{align*}
    \bigintsss_{B_{1}(0)} \frac{|u_{\epsilon}(x)-(u_{\epsilon})_{B_{1}(0)}|^{\tau}}{ (1-|x|) \ln^{\tau} \left( \frac{2}{1-|x|} \right) } \, dx &  =   \bigintsss_{B_{1}(0)} \frac{|u_{\epsilon}(x)- 1+ O(\epsilon )|^{\tau}}{ (1-|x|) \ln^{\tau}\left( \frac{2}{1-|x|} \right)} \, dx  \\ & = C  \bigintsss_{B_{1}(0)} \frac{|u_{\epsilon}(x)- 1 |^{\tau}}{ (1-|x|) \ln^{\tau} \left( \frac{2}{1-|x|} \right) } \, dx + O(\epsilon ) \\ &  = C \bigintsss_{1- \epsilon<|x|<1} \frac{\left| \ln \left(  \frac{2}{\epsilon} \right) - \ln \left( \frac{2}{1-|x|} \right) \right|^{\tau}}{ (1-|x|) \ln^{2 \tau} \left( \frac{2}{1-|x|} \right) } \, dx  +  O(\epsilon ).
\end{align*} 
Using the change of variable in polar coordinates, we get
\begin{align*}
    \bigintsss_{B_{1}(0)} \frac{|u_{\epsilon}(x)-(u_{\epsilon})_{B_{1}(0)}|^{\tau}}{ (1-|x|) \ln^{\tau} \left( \frac{2}{1-|x|} \right) } \, dx &   = C \int_{1- \epsilon}^{1} \frac{\left| \ln \left(  \frac{2}{\epsilon} \right) - \ln \left( \frac{2}{1-r} \right) \right|^{\tau}}{ (1-r) \ln^{2 \tau} \left( \frac{2}{1-r} \right) } r^{d-1} \, dr  +  O(\epsilon)  \\ & = C \int_{0}^{\epsilon} \frac{\left| \ln \left(  \frac{2}{\epsilon} \right) - \ln \left( \frac{2}{t} \right) \right|^{\tau}}{ t \ln^{2 \tau} \left( \frac{2}{t} \right) } (1-t)^{d-1} \, dt  +  O(\epsilon) \\ & \geq C  \int_{0}^{\epsilon} \frac{\left| \ln \left(  \frac{2}{\epsilon} \right) - \ln \left( \frac{2}{t} \right) \right|^{\tau}}{ t \ln^{2 \tau} \left( \frac{2}{t} \right) } \, dt  +  O (\epsilon) .
\end{align*}
Using the change of variable  $ \frac{\ln \left( \frac{2}{t} \right)}{ \ln \left( \frac{2}{\epsilon} \right)}  = \xi$, we obtain
\begin{equation}\label{optimality lhs}
    \bigintsss_{B_{1}(0)} \frac{\left|u_{\epsilon}(x)-(u_{\epsilon})_{B_{1}(0)} \right|^{\tau}}{ (1-|x|) \ln^{\tau} \left( \frac{2}{1-|x|} \right) } \, dx \geq \frac{C}{\ln^{\tau -1} \left( \frac{2}{\epsilon} \right)} \int_{1}^{\infty} \frac{|\xi - 1|^{\tau}}{\xi^{2 \tau}} \, d \xi +   O(\epsilon) = \frac{C}{\ln^{\tau -1} \left( \frac{2}{\epsilon} \right)} + O(\epsilon) .
\end{equation}
Also, from \eqref{limit of sequence for d>1}, we have
\begin{equation*}
    [u_{\epsilon}]^{p}_{W^{s,p}(B_{1}(0))} \leq  \frac{C}{\ln^{p-1} \left( \frac{2}{\epsilon} \right)}.
\end{equation*}
Therefore, by following steps similar to those in the previous subsection for general bounded Lipschitz domains in the cases $\tau \geq p$ when $d=1$, and $\tau = p$ with $d > 1$, and using the above two inequalities, one can deduce the optimality of the weight function in Theorem \ref{mainresult} for $\tau \ge p$ when $d = 1$ and for $\tau = p$ when $d > 1$, when the domain is $\Omega = B_{1}(0)$.

\section{Weighted fractional boundary Hardy-type inequality}\label{weighted fractional boundary Hardy-type}

In this section, we establish a weighted fractional boundary Hardy-type inequality in 
$\mathbb{R}^{d}_{+}$ for the critical case $1 + \beta_{1} + \beta_{2} = sp$ under suitable 
conditions on $\beta_{1}$ and $\beta_{2}$. This result serves as a fundamental building block 
for proving weighted fractional boundary Hardy-type inequalities for the case 
$1 + \beta_{1} + \beta_{2} = sp$ in various domains. In particular, we will prove the following 
theorem:

\begin{thm}\label{upperhalfplane}
Let $d \geq 1$, $p>1$, and $\tau>1$, and define $ \alpha = d + (1-d)\frac{\tau}{p}$. Assume that $\beta_{1}, \beta_{2} \in \mathbb{R}$ satisfy $\beta_{1},\; \beta_{2},\; \beta_{1}+\beta_{2} \in (-1,0]$, and  let $s \in (0,1)$ be such that $ 1 + \beta_{1} + \beta_{2} = sp$. Assume moreover that $\tau \geq p$ when $sp=d=1$, and that $\tau \in [p, p^{*}_{s}]$ when $sp < d$. Then there exists a constant  $C = C(d,p,s,\tau,\beta_{1},\beta_{2},R) > 0$ such that for every $u \in C^{1}_{c}(\mathbb{R}^{d}_{+})$ satisfying 
$\operatorname{supp} u \subset \mathbb{R}^{d-1} \times (0,R)$ for some $R>0$, the following weighted Hardy-type inequality holds,
\begin{equation}\label{ineq-upperhalfspace}
    \left( 
        \int_{\mathbb{R}^{d}_{+}}
        \frac{|u(x)|^{\tau}}{x_{d}^{\alpha}\, \ln^{\,b}\!\left( \tfrac{4R}{x_{d}} \right)}
        \, dx
    \right)^{\!\frac{1}{\tau}}
    \leq
    C
    \left(
        \int_{\mathbb{R}^{d}_{+}}
        \int_{\mathbb{R}^{d}_{+}}
        \frac{|u(x)-u(y)|^{p}}{|x-y|^{d+sp}}
        \, x_{d}^{\beta_{1}} y_{d}^{\beta_{2}}
        \, dx \, dy
    \right)^{\!\frac{1}{p}},
\end{equation}
where $b= \tau -1$ when $0 \leq \alpha <1$, and $b=\tau$ when $\alpha=1$.
\end{thm}
\begin{proof}
Let  $n_{0}, ~ n_{1} \in \mathbb{Z}$ be such that  $2^{n_{0}} \leq R \leq 2^{n_{0}+1}$, and  $2^{n_{1}} \geq 2^{n_{0}+1}$ so that  $D= (-2^{n_{1}},2^{n_{1}})^{d-1} \times (0,2^{n_{0}+1})$ and  $\operatorname{supp} u \subset D$. From the definition of the sets  $A_{k}$ and  $A^{i}_{k}$ in Section \ref{The flat boundary case} with  $n=2^{n_{1}}$, we have
\begin{equation*}
    D= \bigcup_{k = - \infty}^{n_{0}} A_{k}  ,
\end{equation*}
and
\begin{equation*}
    A_{k} = \bigcup_{i = 1}^{\sigma_{k}} A^{i}_{k}  ,
\end{equation*}
where  $\sigma_{k} =  2^{(-k+1)(d-1)} 2^{n_{1}(d-1)}$. Fix such a set  $A^{i}_{k}$. Applying Lemma \ref{sobolev} with  $\Omega = (1,2)^{d}, ~ \lambda=2^{k}$, and  $sp=1+ \beta_{1} + \beta_{2}$, we obtain
\begin{equation*}
    \fint_{A^{i}_{k}} |u(x)-(u)_{A^{i}_{k}}|^{\tau} \, dx \leq C 2^{k(1+ \beta_{1} + \beta_{2} -d) \frac{\tau}{p}} [u]^{\tau}_{W^{s,p}(A^{i}_{k})}  ,
\end{equation*}
where  $C=C(d,p,s, \tau)$ is a constant. Since for $x, ~ y \in A^{i}_{k}$, we have  $ 2^{k \beta_{1}}  \leq C x_{d}^{\beta_{1}}$ and   $ 2^{k \beta_{2}} \leq C y_{d}^{\beta_{2}}$, we get
\begin{equation}\label{weighted sobolev on cube}
    \fint_{A^{i}_{k}} |u(x)-(u)_{A^{i}_{k}}|^{\tau}  dx \leq C 2^{k(1 -d) \frac{\tau}{p}} [u]^{\tau}_{W^{s,p, \beta_{1}, \beta_{2}}(A^{i}_{k})}  ,
\end{equation}
where $C=C(d, p,s, \tau,  \beta_{1},  \beta_{2})$ is a constant. For any  $x = (x',x_d) \in A^{i}_{k}$, we have  $x_d \geq 2^k$. Therefore, we obtain
\begin{equation*}
    \int_{A^{i}_{k}} \frac{|u(x)|^{\tau}}{x^{\alpha}_{d}}  \, dx 
    \leq \frac{C}{2^{k \alpha}} \int_{A^{i}_{k}} |u(x)-(u)_{A^{i}_{k}}|^{\tau} \,  dx +  \frac{C}{2^{k \alpha}} \int_{A^{i}_{k}} |(u)_{A^{i}_{k}}|^{\tau}dx.
    \end{equation*}
    Using the previous inequality together with the identity  $\alpha = d+(1-d)\frac{\tau}{p}$, we obtain
    \begin{equation*}
    \int_{A^{i}_{k}} \frac{|u(x)|^{\tau}}{x^{\alpha}_{d}} \, dx  
   \leq C [u]^{\tau}_{W^{s,p, \beta_{1}, \beta_{2}}(A^{i}_{k})} + C 2^{k(d-\alpha)} |(u)_{A^{i}_{k}}|^{\tau}  .
\end{equation*}   
For each $x \in A^{i}_{k}$, we have $\frac{1}{x_{d}} > \frac{1}{2^{k+1}}$. 
Since there exists $n_{0} \in \mathbb{Z}$ such that $2^{n_{0}} \leq R \leq 2^{\,n_{0}+1}$, 
it follows that   $ \ln \left(\frac{4R}{x_{d}} \right) > (n_{0}-k+1) \ln 2$. Therefore, we obtain
\begin{equation*}
    \bigintsss_{A^{i}_{k}} \frac{|u(x)|^{\tau}}{x^{\alpha}_{d} \ln^{b} \left(\frac{4R}{x_{d}} \right)}  \, dx 
     \leq  C [u]^{\tau}_{W^{s,p, \beta_{1}, \beta_{2}}(A^{i}_{k})} + C \frac{2^{k(d-\alpha)}}{(n_{0}-k+1)^{b}} |(u)_{A^{i}_{k}}|^{\tau}  .
\end{equation*}
Observe that this reduces to an inequality identical in form to \eqref{eqn0090} in the proof of Lemma \ref{case 1 flat case}. Therefore, by repeating the same steps used in that proof and summing over the relevant indices, we obtain
\begin{equation*}
    \sum_{k=m}^{n_{0}} \bigintsss_{A_{k}} \frac{|u(x)|^{\tau}}{x^{\alpha}_{d} \ln^{b} \left( \frac{4R}{x_{d}} \right)} \,  dx \leq  C 
     \sum_{j=1}^{\sigma_{n_{0}+1}} |(u)_{A^{j}_{n_{0}+1}}|^{\tau} +  C \sum_{k=m}^{n_{0}} [u]^{\tau}_{W^{s,p, \beta_{1}, \beta_{2}}(A_{k} \cup A_{k+1})}  .
\end{equation*}
Since  $\operatorname{supp} u \subset (-2^{n_{1}},2^{n_{1}})^{d-1} \times (0,2^{n_{0}+1})$, we have $(u)_{A^{j}_{n_{0}+1}} = 0$. Therefore,
\begin{equation*}
     \sum_{k=m}^{n_{0}} \bigintsss_{A_{k}} \frac{|u(x)|^{\tau}}{x^{\alpha}_{d} \ln^{b} \left( \frac{4R}{x_{d}} \right)}  dx \leq  C 
\sum_{k=m}^{n_{0} } [u]^{\tau}_{W^{s,p, \beta_{1}, \beta_{2}}(A_{k} \cup A_{k+1})}  .
\end{equation*}
Hence,
\begin{equation*}
   \left(  \bigintsss_{\mathbb{R}^{d}_{+}} \frac{|u(x)|^{\tau}}{x^{\alpha}_{d} \ln^{b} \left(\frac{4R}{x_{d}} \right)} \,  dx \right)^{\frac{1}{\tau}} \leq  C 
 [u]_{W^{s,p, \beta_{1}, \beta_{2}}(\mathbb{R}^{d}_{+})}.
\end{equation*}
This proves the theorem.
\end{proof}

\smallskip

Following a similar approach as illustrated in the proof of Theorem  \ref{mainresult} and using Lemma  \ref{testfunc2} (See Appendix  \ref{appen}), we can prove the following theorem for weighted fractional Sobolev space.

\begin{thm}\label{weighted fractional main theorem}
Let $d \geq 1$, $\beta_{1}, \beta_{2} \in \mathbb{R}$ satisfy $\beta_{1}, \beta_{2}, \beta_{1}+\beta_{2} \in \left(-\frac{1}{2},0 \right)$ and $p> \frac{d+2(1+\beta_{1}+\beta_{2})}{2}$, and $\tau>1$. Let $s \in (0,1)$ be such that $sp = 1 + \beta_{1}+ \beta_{2}$. Let  $\Omega$ be an open set in $\mathbb{R}^d$, then the following weighted fractional boundary Hardy-type inequality
    \begin{equation}
        \left( \int_{\Omega} \frac{|u(x)|^{\tau}}{\delta_{\Omega}^{\alpha}(x) \ln^{\beta}(\rho(x))}  dx \right)^{\frac{1}{\tau}} \leq C \|u\|_{W^{s,p, \beta_{1}, \beta_{2}}(\Omega)},  \quad \ \forall \ u \in C^{1}_{c}(\Omega),
    \end{equation}
   where $C=C(d,p,s, \tau, \beta_{1}, \beta_{2}, R, \Omega)>0$, holds true in each of the following cases:
    \begin{itemize}
        \item[(D$1$)] $\Omega$ is a bounded Lipschitz domain such that  $\delta_{\Omega}(x) < R$ for all  $x \in \Omega$, for some  $R>0$. In this case,  $sp < 1$,  $\rho(x) = \frac{4R}{\delta_{\Omega}(x)} $ and \\
        (a) if  $\tau =p$, then  $\alpha=1$ and  $ \beta=p$. \\
        (b) if  $\tau \in (p,p + \frac{1}{d}]$, then  $\alpha = d+ (1-d)\frac{\tau}{p} $ and  $\beta= \tau -1$. \\
        (c) if  $\tau \in (p+ \frac{1}{d}, p^{*}_{s} ]$, then  $\alpha = d+ (1-d)\frac{\tau}{p}$ and  $\beta= dp + (1-d) \tau$. \\
        (d) if  $\tau <p$, then  $\alpha=1$ and  $\beta=p$.
        \item[(D$2$)]  $\Omega$ is a domain above the graph of a Lipschitz function  $\gamma : \mathbb{R}^{d-1} \to \mathbb{R}, ~ d>1$, such that  $\operatorname{supp} u \subset \Omega \backslash \overline{\Omega}_{R}$, where  $\Omega_{R}$ is a domain above the graph of a Lipschitz function  $\gamma_{R}: \mathbb{R}^{d-1} \to \mathbb{R}$ such that  $\gamma_{R}(x') = \gamma(x')+R$ for some  $R>0$. In this case,  $sp < 1$,  $\rho(x) = \frac{4R}{\delta_{\Omega}(x)} $ and \\
         (a) if  $\tau =p$, then  $\alpha=1$ and  $ \beta=p$. \\
        (b) if  $\tau \in (p,p + \frac{1}{d}]$, then  $\alpha = d+ (1-d)\frac{\tau}{p} $ and  $\beta= \tau -1$. \\
        (c) if  $\tau \in (p+ \frac{1}{d}, p^{*}_{s} ]$, then  $\alpha = d+ (1-d)\frac{\tau}{p}$ and  $\beta= dp + (1-d) \tau$. \\
        (d) if  $\tau <p$, then  $\alpha=1$ and  $\beta=p$.
    \end{itemize}
\end{thm}

\section{Appendix}\label{appen}

In this section, we establish several important results that are used in the proofs of our main theorems. Although some of these results are already known in the literature, we include their proofs here to keep the article self-contained.

\subsection{FBHI \eqref{fractionalhardy} fails for (T\texorpdfstring{$2$}{2}) with  \texorpdfstring{$sp=1$}{sp=1} }\label{failure FBHI} Assume that \eqref{fractionalhardy} holds for (T$2$) with $sp=1$. Let $\Omega = B^{c}_{1}(0)$ and, for large $n \in \mathbb{N}$, define a radial function $v_{n} \in C^{\infty}_{c}(B^{c}_{1}(0))$ such that
\begin{equation*}
     v_{n}(x) = \begin{cases}
        1, & 1+ \frac{2}{n} \leq |x| \leq 2  \\
        0, & |x| < 1+ \frac{1}{n} \ \text{or} \  |x|  > 3.
    \end{cases}
\end{equation*}
Moreover, $0 \leq v_{n} \leq 1$, $|\nabla v_{n}(x)| \leq C n$ for all $1+ \dfrac{1}{n} < |x| < 1+ \dfrac{2}{n}$, and $|\nabla v_{n}(x)| \leq C$ for all $|x| > 1+ \dfrac{2}{n}$, for some constant $C>0$. We denote $A(r_{1}, r_{2}) = B_{r_{2}}(0) \setminus \overline{B_{r_{1}}(0)}$ for $r_{2}> r_{1}>0$. Let us calculate the following expression: 
\begin{align*}
    \int_{B^{c}_{1}(0)} \int_{B^{c}_{1}(0)} \frac{|v_{n}(x)- v_{n}(y)|^{p}}{|x-y|^{d+1}} \, dx \, dy  & =  \int_{A(1, 4)} \int_{A(1, 4)}  \frac{|v_{n}(x)- v_{n}(y)|^{p}}{|x-y|^{d+1}} \, dx \, dy \\ & \quad + 2  \int_{A(1,3)} |v_{n}(x)|^{p} \left( \int_{A(4, \infty)} \frac{1}{|x-y|^{d+1}} \, dy \right) \, dx \\ & =: J_{1}+J_{2}.
\end{align*}
We first estimate $J_{1}$. Since $v_{n}$ is radial, applying Lemma~\ref{d dim to 1 dim} with $R_{1}=1$ and $R_{2}=4$, we obtain
\begin{align*}
J_{1} \leq C \int_{1}^{4} \int_{1}^{4} \frac{|v_{n}(r) - v_{n}(t)|^{p}}{|r - t|^{2}} & \, dr \,  dt  = C \Bigg( \int_{1}^{1+\frac{2}{n}} \int_{1}^{1+\frac{3}{n}}  + \int_{1}^{1+\frac{2}{n}} \int_{1+\frac{3}{n}}^{4}  \\ & \quad + \int_{1+\frac{2}{n}}^{4} \int_{1}^{1+\frac{3}{n}} +  \int_{1+\frac{2}{n}}^{4} \int_{1+\frac{3}{n}}^{4} \Bigg)  \frac{|v_{n}(r) - v_{n}(t)|^{p}}{|r - t|^{2}} \, dr \,  dt \\ & =: J^{1}_{1} + J^{2}_{1} + J^{3}_{1} + J^{4}_{1} .
\end{align*}
For $J^{1}_{1}$, using the bound $|v_{n} (r) - v_{n} (t)| \leq C n|r-t|$ and $p>1$, we have
\begin{align*}
  J^{1}_{1} & \leq C n^{p}  \int_{1}^{1+\frac{2}{n}} \int_{1}^{1+\frac{3}{n}} \frac{|r-t|^{p}}{|r-t|^{2}} \, dr \,  dt \\  & =  C n^{p}  \int_{0}^{\frac{2}{n}} \int_{0}^{\frac{3}{n}} |r-t|^{p-2}\, dr \,  dt  \leq C n^{p} \int_{0}^{\frac{2}{n}} dt \int_{0}^{\frac{5}{n}} z^{p-2} dz  =  C n^{p} \left( \frac{5}{n} \right)^{p-1} \frac{2}{n}= C.
\end{align*}
For $J^{2}_{1}$, we have
\begin{align*}
    J^{2}_{1} & = C \int_{1}^{1+\frac{2}{n}} \int_{1+\frac{3}{n}}^{4} \frac{|v_{n}(r) - v_{n}(t)|^{p}}{|r - t|^{2}} \, dr \,  dt \\ &  \leq C 2^{p}  \int_{1}^{1+\frac{2}{n}} \int_{1+\frac{3}{n}}^{4}  \frac{1}{|r - t|^{2}} \, dr \, dt   = C 2^{p} \int_{0}^{\frac{2}{n}} \int_{\frac{3}{n}}^{3} |r-t|^{-2} dr \, dt \leq C 2^{p} \int_{0}^{\frac{2}{n}} dt \int_{\frac{1}{n}}^{\infty} z^{-2} \, dz = C.
\end{align*}
 For $J^{3}_{1}$, using $|v_{n}(r)-v_{n}(t)| \leq C |r-t|$ and $p>1$, we obtain
\begin{align*}
   J^{3}_{1} & =  C \int_{1+\frac{2}{n}}^{4} \int_{1}^{1+\frac{3}{n}} \frac{|v_{n}(r) - v_{n}(t)|^{p}}{|r - t|^{2}} \, dr \,  dt  \leq C^{p}  \int_{1+\frac{2}{n}}^{4} \int_{1}^{1+\frac{3}{n}} |r-t|^{p-2} \, dr \, dt \\ & =  C^{p} \int_{\frac{2}{n}}^{3} \int_{0}^{\frac{3}{n}} |r-t|^{p-2} \, dr \, dt   \leq C^{p} \int_{0}^{3} \int_{0}^{1} |r-t|^{p-2} \, dr \, dt  = C. 
\end{align*}
Finally, for $J^{4}_{1}$, using the same bound and $p>1$, we have
\begin{align*}
    J^{4}_{1} & = C \int_{1+\frac{2}{n}}^{4} \int_{1+\frac{3}{n}}^{4} \frac{|v_{n}(r) - v_{n}(t)|^{p}}{|r - t|^{2}} \, dr \,  dt \\ & \leq C^{p}  \int_{1+\frac{2}{n}}^{4} \int_{1+\frac{3}{n}}^{4} |r - t|^{p-2} \, dr \,  dt  \leq C^{p} \int_{0}^{4} \int_{0}^{4} |r - t|^{p-2} \, dr \,  dt = C.
\end{align*}
Combining the above estimates, we obtain $J^{i}_{1} \leq C$ for all $i=1,2,3,4$, and consequently, 
$J_{1} \leq C$.
Finally, for $J_{2}$, we have
\begin{align*}
 J_{2} =   \int_{A(1,3)} |v_{n}(x)|^{p} \left( \int_{A(4, \infty)} \frac{1}{|x-y|^{d+1}} \, dy \right) \, dx \leq  \int_{A(1,3)} dx \int_{|z| >1} |z|^{-d-1} dz = C.
\end{align*}
Hence, combining the estimates for $J_{1}$ and $J_{2}$, we obtain
\begin{equation}\label{counter1}
     \int_{B^{c}_{1}(0)} \int_{B^{c}_{1}(0)} \frac{|v_{n}(x)- v_{n}(y)|^{p}}{|x-y|^{d+1}} \, dx \, dy \leq C.
\end{equation}
On the other hand, the left-hand side of \eqref{fractionalhardy} satisfies
\begin{equation}\label{counter2}
   \int_{B^{c}_{1}(0)} \frac{|v_{n}(x)|^{p}}{(|x|-1)} \, dx \geq \int_{1+ \frac{2}{n} <|x| <2 } \frac{1}{(|x|-1)} \, dx \to \infty \hspace{3mm} \text{as} \ n \to \infty.
\end{equation}
Therefore, \eqref{counter1} and \eqref{counter2} yield a contradiction to \eqref{fractionalhardy}.

\subsection{Slicing}\label{slicing lemma}
We introduce a slicing lemma for the Gagliardo seminorm. While this inequality can be found in  \cite[Section 6.2]{leonibook}, for the sake of completeness and self-containment, we present its proof in this article. Let  $\Phi(x)=\phi_{\mu}(x_{1}) \cdots \phi_{\mu}(x_{d})$ defined in Subsection  \ref{opt proof for gen domain} with $\Omega= (0,1)^{d}$. Adding and subtracting  $\phi_{\mu}(y_{1})\phi_{\mu}(x_{2}) \cdots \phi_{\mu}(x_{d})$ and using  $|\phi_{\mu}(x_{i})| \leq 1$ for all  $ 1 \leq i \leq d$, we get
\begin{multline*}
    |\phi_{\mu}(x_{1})\phi_{\mu}(x_{2}) \cdots \phi_{\mu}(x_{d}) - \phi_{\mu}(y_{1}) \phi_{\mu}(y_{2}) \cdots \phi_{\mu}(y_{d})|^{p}  \\ \leq C  |\phi_{\mu}(x_{1})-\phi_{\mu}(y_{1})|^{p} + C |\phi_{\mu}(x_{2}) \cdots \phi_{\mu}(x_{d}) - \phi_{\mu}(y_{2}) \cdots \phi_{\mu}(y_{d})|^{p}.
\end{multline*}
Similarly, proceeding similarly as above for  $|\phi_{\mu}(x_{i}) \cdots \phi_{\mu}(x_{d}) - \phi_{\mu}(y_{i}) \cdots \phi_{\mu}(y_{d})|^{p}, ~ 2 \leq i \leq d-1$, we obtain
\begin{equation*}
    |\phi_{\mu}(x_{1}) \cdots \phi_{\mu}(x_{d}) - \phi_{\mu}(y_{1}) \cdots \phi_{\mu}(y_{d})|^{p}  \leq  C \sum_{j=1}^{d} |\phi_{\mu}(x_{j})- \phi_{\mu}(y_{j})|^{p}.
\end{equation*}
Therefore, we have
\begin{align*}
    [\Phi]^{p}_{W^{s,p}(\Omega)} &  \leq C \sum_{j=1}^{d} \int_{\Omega} \int_{\Omega} \frac{|\phi_{\mu}(x_{j})-\phi_{\mu}(y_{j})|^{p}}{|x-y|^{d+sp}} \, dx \, dy \\ & = C \sum_{j=1}^{d} \bigintsss_{\Omega} \bigintsss_{\Omega} \frac{|\phi_{\mu}(x_{j})- \phi_{\mu}(y_{j})|^{p}}{(x_{j}-y_{j})^{d+sp} \left( 1+ \sum_{\substack{i=1 \\ i \neq j}}^{d} \left| \frac{x_{i}-y_{i}}{x_{j}-y_{j}} \right|^{2} \right)^{\frac{d+sp}{2}}} \, dx \, dy .
\end{align*}
Using the change of variables $z_{i} = \frac{x_{i}-y_{i}}{x_{j}-y_{j}}$ and the fact 
\begin{equation*}
    \bigintssss_{(0, \infty)^{d-1}} \frac{1}{(1+ \sum_{\substack{i=1 \\ i \neq j}}^{d} |z_{i}|^{2} )^{\frac{d+sp}{2}}} \, dz_{1} \cdots dz_{j-1} \, dz_{j+1} \cdots  dz_{d} < \infty,
\end{equation*}
we obtain
\begin{equation*}
    [\Phi]^{p}_{W^{s,p}(\Omega)} \leq C \sum_{j=1}^{d} \int_{0}^{1} \int_{0}^{1} \frac{|\phi_{\mu}(x_{j})- \phi_{\mu}(y_{j})|^{p}}{|x_{j}-y_{j}|^{1+sp}} \, dx_{j} \, dy_{j}.
\end{equation*}
This proves the slicing lemma for Gagliardo seminorm. Furthermore, if we consider the function  $\widetilde{\Phi}_{\epsilon,\mu}$  defined in Subsection \ref{opt proof for gen domain} as $\widetilde{\Phi}_{\epsilon,\mu}(x',x_{d}) = \phi_{\mu}(x_{d})$ for all $(x',x_{d}) \in \Omega= (0,1)^{d}$. Following a similar approach as above, we obtain
\begin{equation*}
    [\widetilde{\Phi}_{\epsilon,\mu}]^{p}_{W^{s,p}((0,1)^{d})} \leq C [\phi_{\mu}]^{p}_{W^{s,p}((0,1))}.
\end{equation*}

\subsection{Density Result}\label{density result appen} In this subsection, we will establish the density result. In fact we prove that for an open set $\Omega \subset \mathbb{R}^{d}$, we have $W^{s,p}_{0}(\Omega) = W^{s,p}(\Omega)$ if and only if the constant function $1 \in W^{s,p}_{0}(\Omega)$. This result can be found in \cite[Lemma 13]{dyda2022}, but for the sake of completeness and self-containment, we provide its proof here. We therefore present the following lemma:
\begin{lemma}\label{density lemma appen}
    Let $\Omega$ be an open set in $\mathbb{R}^{d}$. Then $W^{s,p}(\Omega)=W^{s,p}_{0}(\Omega)$ if and only if the constant function $1 \in W^{s,p}_{0}(\Omega)$.
\end{lemma}
\begin{proof}
    If $W^{s,p}_{0}(\Omega) = W^{s,p}(\Omega)$, then the constant function $1 \in W^{s,p}_{0}(\Omega)$. Assume $1 \in W^{s,p}_{0}(\Omega)$, there exists a sequence of functions $g_{m} \in C^{\infty}_{c}(\Omega)$ such that $g_{m} \to 1$ in $W^{s,p}_{0}(\Omega)$. Let $f \in W^{s,p}(\Omega)$ and we may also assume $f \in L^{\infty}(\Omega)$. Define a sequence of functions $\{h_{m}=f g_{m} \}_{m \in \mathbb{N}}$. Then $h_{m} \in W^{s,p}_{0}(\Omega)$ and clearly $h_{m} \to f$ in $L^{p}(\Omega)$. Now let us calculate the Gagliardo seminorm of $h_{m}-f$,
    \begin{align*}
        [h_{m} - f]^{p}_{W^{s,p}(\Omega)} & = \int_{\Omega} \int_{\Omega} \frac{|f(x)(1-g_{m}(x))-f(y)(1- g_{m}(y))|^{p}}{|x-y|^{d+sp}} \, dx \, dy \\ & \leq 2^{p-1} \int_{\Omega} \int_{\Omega}  \frac{|f(x)|^{p}|g_{m}(x)-g_{m}(y)|^{p}}{|x-y|^{d+sp}} \, dx \, dy \\ & \quad + 2^{p-1} \int_{\Omega} \int_{\Omega}  \frac{|1-g_{m}(x)|^{p}|f(x)-f(y)|^{p}}{|x-y|^{d+sp}} \, dx \, dy \\ & \leq  2^{p-1}\|f\|^{p}_{\infty} [g_{m}]^{p}_{W^{s,p}(\Omega)} + 2^{p-1} \int_{\Omega} \int_{\Omega}  \frac{|1-g_{m}(x)|^{p}|f(x)-f(y)|^{p}}{|x-y|^{d+sp}}  \, dx \, dy.
    \end{align*} 
    Since $g_{m} \to 1$ in $L^{p}(\Omega)$ as $m \to \infty$,  there exists a subsequence (which we again denote by  $\{ g_{m} \}_{m \in \mathbb{N}}$ again) such that $g_{m} \to 1$ almost everywhere as $m \to \infty$. Therefore, by the dominated convergence theorem and the fact that $[g_{m}]^{p}_{W^{s,p}(\Omega)} \to 0$ as $m \to \infty$, we obtain
    \begin{equation*}
        [h_{m} - f]^{p}_{W^{s,p}(\Omega)} \to 0 \quad \text{as} \quad m \to \infty. 
    \end{equation*}
    Hence, $f \in W^{s,p}_{0}(\Omega)$. This completes the proof of the lemma.
\end{proof}

\subsection{Weighted fractional Sobolev space} In this subsection, we establish an analogue of Lemma \ref{testfunc} for the weighted fractional Sobolev space (see \eqref{weighted fractional sobolev norm} for the definition). This lemma plays an important role in the proof of Theorem \ref{weighted fractional main theorem}. In particular, we will prove the following lemma.

\begin{lemma}\label{testfunc2}
Let $\Omega$ be a bounded Lipschitz domain in $\mathbb{R}^{d}$. Assume that $\beta_{1}, \beta_{2} \in \mathbb{R}$ satisfy $\beta_{1}, \beta_{2}, \beta_{1}+\beta_{2} \in \left(-\frac{1}{2},0 \right)$ and $p> \frac{d+2(1+\beta_{1}+\beta_{2})}{2}$. Let $s \in (0,1)$ be such that $sp = 1 + \beta_{1}+ \beta_{2}$. Consider $u \in W^{s,p, \beta_{1}, \beta_{2}}(\Omega)$ and $\xi \in C^{0,1}(\Omega)$ with $0 \leq \xi \leq 1$. Then $\xi u \in W^{s,p, \beta_{1}, \beta_{2}}(\Omega)$, and for some constant $C=C(d,p,s, \beta_{1}, \beta_{2},\Omega)>0$,
    \begin{equation*}
        \| \xi u\|_{W^{s,p, \beta_{1}, \beta_{2}}(\Omega)} \leq C \|u\|_{W^{s,p, \beta_{1}, \beta_{2}}(\Omega)}.
    \end{equation*}
\end{lemma}
\begin{proof}
Let $sp=1+ \beta_{1}+ \beta_{2}$ and $\beta_{1},  \beta_{2},  \beta_{1}+ \beta_{2} \in \left( - \frac{1}{2}, 0  \right)$. Since $\xi \leq 1$, it is clear that $\|\xi u\|_{L^{p}(\Omega)} \leq \|u\|_{L^{p}(\Omega)}$. Now, adding and subtracting the term $\xi(y) u(x)$ and using $|\xi(x)-\xi(y)| \leq C |x-y|$, for some $C>0$ and for all $x,y \in \Omega$, we obtain
    \begin{align*}
       & [\xi u]^{p}_{W^{ s,p, \beta_{1}, \beta_{2}}(\Omega)}  \leq C [u]^{p}_{W^{s,p, \beta_{1}, \beta_{2}}(\Omega)} + C \int_{\Omega} |u(x)|^{p} \int_{\Omega}  \frac{|\xi(x)-\xi(y)|^{p}}{ |x-y|^{d+sp}}\delta^{\beta_{1}}_{\Omega}(x) \delta^{\beta_{2}}_{\Omega}(y) \, dy \, dx \\ & \leq  C [u]^{p}_{W^{s,p, \beta_{1}, \beta_{2}}(\Omega)} + C \int_{\Omega} |u(x)|^{p} \int_{\Omega}   |x-y|^{p-d-sp} \delta^{\beta_{1}}_{\Omega}(x) \delta^{\beta_{2}}_{\Omega}(y) \, dy \, dx  \\ & = C  [u]^{p}_{W^{s,p, \beta_{1}, \beta_{2}}(\Omega)} + C \int_{\Omega} |u(x)|^{p} \int_{\Omega \cap \{ \delta_{\Omega}(x) \leq \delta_{\Omega}(y) \} }  |x-y|^{p-d-sp} \delta^{\beta_{1}}_{\Omega}(x) \delta^{\beta_{2}}_{\Omega}(y) \, dy \, dx \\ &\quad + C \int_{\Omega} |u(x)|^{p} \int_{\Omega \cap \{ \delta_{\Omega}(y) < \delta_{\Omega}(x) \} }  |x-y|^{p-d-sp} \delta^{\beta_{1}}_{\Omega}(x) \delta^{\beta_{2}}_{\Omega}(y) \, dy \, dx \\ & =: C  [u]^{p}_{W^{s,p, \beta_{1}, \beta_{2}}(\Omega)} + I_{1}+ I_{2}.
    \end{align*}
For $I_{1}$, we have $\delta^{\beta_{2}}_{\Omega}(y) \leq \delta^{\beta_{2}}_{\Omega}(x)$ as $\beta_{2} <0$. Therefore, using the estimate $\int_{\Omega} |x-y|^{p-d-sp} \, dy \leq C$, we obtain
\begin{align*}
    I_{1} & \leq C \int_{\Omega} |u(x)|^{p} \delta^{\beta_{1}+ \beta_{2}}_{\Omega}(x) \int_{\Omega }  |x-y|^{p-d-sp} \, dy \, dx \\ & \leq C \int_{\Omega} |u(x)|^{p} \delta^{\beta_{1}+ \beta_{2}}_{\Omega}(x) \, dx \\ & = C \int_{\Omega} \frac{|u(x)|^{p}}{\delta^{sp}_{\Omega}(x)}  \delta^{1+2(\beta_{1}+ \beta_{2})}_{\Omega}(x) \, dx \leq C \int_{\Omega} \frac{|u(x)|^{p}}{\delta^{sp}_{\Omega}(x)} \, dx \leq C \left( [u]^{p}_{W^{s,p}(\Omega)} + \|u\|^{p}_{L^{p}(\Omega)} \right).
\end{align*}
Here, we have used fractional boundary Hardy inequality for the case $sp<1$ (See Theorem \ref{theorem sp<1}), and the fact $1+ 2(\beta_{1}+ \beta_{2}) >0$. Using $[u]^{p}_{W^{s,p}(\Omega)} \leq C [u]^{p}_{W^{s,p, \beta_{1}, \beta_{2}}(\Omega)} $, we get
\begin{equation*}
    I_{1} \leq C \left([u]^{p}_{W^{s,p, \beta_{1}, \beta_{2}}(\Omega)} + \|u\|^{p}_{L^{p}(\Omega)} \right).
\end{equation*}
Now, for $I_{2}$, we have $\delta^{\beta_{1}}_{\Omega}(x) < \delta^{\beta_{1}}_{\Omega}(y)$ as $\beta_{1} <0$. Therefore, using H$\ddot{\text{o}}$lder's inequality and the fact that $ -1<2(\beta_{1}+ \beta_{2}) <0$, we obtain
\begin{align*}
    I_{2} & \leq C \int_{\Omega} |u(x)|^{p} \int_{\Omega} |x-y|^{p-d-sp} \delta^{\beta_{1}+ \beta_{2}}_{\Omega}(y) \, dy \, dx \\ & \leq  C \int_{\Omega} |u(x)|^{p} \left( \int_{\Omega} |x-y|^{2p-2d-2sp} \, dy \right)^{\frac{1}{2}} \left( \int_{\Omega}  \delta^{2(\beta_{1}+ \beta_{2})}_{\Omega}(y) \, dy \right)^{\frac{1}{2}} \, dx \\ & \leq C \int_{\Omega} |u(x)|^{p} \left( \int_{\Omega} |x-y|^{2p-2d-2sp} \, dy \right)^{\frac{1}{2}} \, dx.
\end{align*}
Since
\begin{equation*}
    \int_{\Omega} |x-y|^{2p-2d-2sp} dy < \infty,
\end{equation*}
whenever $2p-d-2sp>0$, that is $p> \frac{d+2sp}{2}= \frac{d+2(1+\beta_{1}+\beta_{2})}{2}$. Hence, 
\begin{equation*}
    I_{2} \leq C \|u\|^{p}_{L^{p}(\Omega)}.
\end{equation*}
This proves the lemma.
\end{proof}

\subsection{Some estimate}\label{some estimate}
In this subsection, we derive an inequality that is useful for establishing our main results in the flat boundary case. Let $\Omega_{n} = (-n,n)^{d-1} \times (0,1)$, and let $A_{k}$ be as defined in Section \ref{The flat boundary case}. We aim to show that
\begin{equation}\label{se1}
     \sum_{k= m}^{-2} [u]^{p}_{W^{s,p}(A_{k} \cup A_{k+1})} \leq 2 [u]^{p}_{W^{s,p}(\Omega_{n})}.
\end{equation}
Consider two families of sets:
\begin{equation*}
  \mathcal{E}:=  \left\{ A_{k} \cup A_{k+1} : -k \ \text{is even and} \ k \leq -1  \right\}
\end{equation*}
and
\begin{equation*}
  \mathcal{O}:=  \left\{ A_{k} \cup A_{k+1} : -k \ \text{is odd and} \ k \leq -1  \right\}.
\end{equation*}
Then $\mathcal{E}$ and $\mathcal{O}$ are collection of mutually disjoint sets respectively. Define
\begin{equation*}
    \mathcal{F}_{e} := \bigcup_{\substack{k =m \\ -k \ \text{is even}}}^{-2} A_{k} \cup A_{k+1} \quad \text{and} \quad \mathcal{F}_{o} := \bigcup_{\substack{k =m \\ -k \ \text{is odd}}}^{-2} A_{k} \cup A_{k+1}.
\end{equation*} 
From the definition of $A_{k}$, we have $\mathcal{F}_{e} \subset \Omega_{n}$ and $\mathcal{F}_{o} \subset \Omega_{n}$. Therefore, we have
\begin{align}\label{sum Ak A(k+1)}
    \sum_{k= m}^{-2} [u]^{p}_{W^{s,p}(A_{k} \cup A_{k+1})} & = \sum_{\substack{k =m \\ -k \ \text{is even}}}^{-2} [u]^{p}_{W^{s,p}(A_{k} \cup A_{k+1})} + \sum_{\substack{k =m \\ -k \ \text{is odd}}}^{-2} [u]^{p}_{W^{s,p}(A_{k} \cup A_{k+1})} \nonumber \\ &  \leq [u]^{p}_{W^{s,p}(\mathcal{F}_{e})} + [u]^{p}_{W^{s,p}(\mathcal{F}_{o})} \leq 2 [u]^{p}_{W^{s,p}(\Omega_{n})}.
\end{align}
This establishes the desired inequality.

\subsection{Domain above the graph of a Lipschitz function}\label{lipschitzdomain}
In this section, we suppose that if  $\Omega$ is a domain lying above the graph of a Lipschitz function  $\gamma: \mathbb{R}^{d-1} \to \mathbb{R}$, and let  $F : \mathbb{R}^d \to \mathbb{R}^d$ be the map defined by  $F(x)= (\xi', \xi_{d})$, where  $\xi ' = x'$ and  $ \xi_{d} = x_{d} - \gamma(x')$. We will prove that
\begin{equation*}
    \delta_{\Omega}(x) \sim \xi_{d},
\end{equation*}
for all  $x \in \Omega$, i.e., there exist  $C_{1}, ~ C_{2} >0$ such that
\begin{equation*}
    C_{1} \xi_{d} \leq \delta_{\Omega}(x) \leq C_{2} \xi_{d} . 
\end{equation*}
Let  $\gamma: \mathbb{R}^{d-1} \to \mathbb{R} $ be a Lipschitz function. Then there exists a constant $M >0$ such that for all  $x', ~ y' \in \mathbb{R}^{d-1}$, we have
\begin{equation*}
    |\gamma(x') - \gamma(y')| \leq M |x'-y' | .
\end{equation*}
Let  $F: \mathbb{R}^d \to \mathbb{R}^d$ be defined by  $F(x)= (F_{1}(x), \dots, F_{d}(x))= (x', x_{d}- \gamma(x'))$, where  $x'=(x_{1}, \dots , x_{d-1})$. Then, using the inequality $(a+b)^{2} \leq 2a^{2}+ 2 b^{2}$, we obtain
\begin{align*}
     |F(x)-F(y)|^{2} & = |x'-y'|^{2} + |x_{d}-y_{d}-\gamma(x')+ \gamma(y')|^{2} \\ &
    \leq |x'-y'|^{2} + 2|x_{d}-y_{d}|^{2} + 2|\gamma(x') - \gamma(y')|^{2} \\ &
    \leq |x'-y'|^{2}  + 2 |x_{d}-y_{d}|^{2} + 2M^2 |x'-y'|^{2} +  
    \leq (2M^{2}+2) |x-y|^{2} .
\end{align*}
Let  $C= (2M^{2}+2)^{1/2}$. Then  $|F(x)-F(y)| \leq C |x-y|$. Define  $G(\xi) = F^{-1}(\xi) = (\xi',\xi_{d}+ \gamma(\xi'))$. Then  $G$ is also Lipschitz, and  $| G(\xi)-G(\eta)| \leq (2M^{2}+2)^{1/2} |\xi-\eta|$. Hence, there exists a constant  $C>0$ such that
\begin{equation*}
    \frac{1}{C} |x-y| \leq |F(x)-F(y)| \leq C |x-y| .
\end{equation*}
Let  $\Omega = \{ x \in \mathbb{R}^d : x_{d} > \gamma(x') \}$ and  $\partial \Omega = \{ x \in \mathbb{R}^d : x_{d} = \gamma(x') \}$. Then  $F(\Omega) = \mathbb{R}^{d}_{+}$ and  $F(\partial \Omega) = \partial \mathbb{R}^d_{+}$. Let  $x \in \Omega$ and  $y \in \partial \Omega$ be such that
\begin{equation*}
    \delta_{\Omega}(x) = |x-y| = \text{inf} \{ |x- \eta| : \eta \in \partial \Omega \} .
\end{equation*}
Then  $\delta_{\Omega}(x) = |x-y| \leq |x- \eta|$ for all  $\eta \in \partial \Omega$. Therefore,
\begin{equation*}
    \delta_{\Omega}(x) = |x-y| \leq C |F(x)- F(\eta)| 
    \leq C |F(x)- \xi | 
\end{equation*}
for all  $\xi \in \partial \mathbb{R}^d_{+}$. So,  $\delta_{\Omega}(x) \leq C \underset{\xi \in \partial \mathbb{R}^{d}_{+}}{\inf}  \{ |F(x) - \xi| \} = C F_{d}(x)$. Let  $F(x) = (\xi' , \xi_{d})$. Then, we obtain  $\delta_{\Omega}(x) \leq C \xi_{d}$. Similarly, by considering the inverse map  $G$, we obtain  $C_{1} \xi_{d} \leq \delta_{\Omega}(x)$. Therefore,
\begin{equation*}
    C_{1} \xi_{d} \leq \delta_{\Omega}(x) \leq C \xi_{d} .
\end{equation*} 

\subsection{Change of variables formula }\label{changeofvariable} (see \cite{gariepy2015measure})
Let $T: \mathbb{R}^{d}\to \mathbb{R}^d$ be a Lipschitz function and $DT(x)$ be the gradient matrix of $T$. Then if $u : \mathbb{R}^d \to \mathbb{R}$ is Lebesgue measurable, then $u \circ T : \mathbb{R}^d \to \mathbb{R}$ is a Lebesgue measurable. If $u$ is Lebesgue measurable on $\mathbb{R}^d$ and if $u \in L^{1}$, then
\begin{equation*}
    \int_{\mathbb{R}^d} u(x) JT(x) \, dx = \bigintsss_{\mathbb{R}^d} \left( \sum_{x \in T^{-1}(y)} u(x) \right) \, dy
\end{equation*}
where $JT(x)$ denotes the Jacobian of $T$ and is defined by $JT(x) = |det DT(x)|$. \smallskip

Let $F: \mathbb{R}^d \to \mathbb{R}^d$ and $G: \mathbb{R}^d \to \mathbb{R}^d$ be Lipschitz functions defined as described earlier. It follows that $F^{-1}(\xi) = G(\xi)$ for all $\xi \in \mathbb{R}^d$, which implies that $F$ is bilipschitz. As a result, $F$ is injective. Also, $JF(x) = JG(x)=1$. If $\operatorname{supp} u \subset \Omega$, then
\begin{equation*}
    \int_{\Omega} u(x) \, dx = \int_{F(\Omega)} u \circ G(x) \, dx \ .
\end{equation*}\smallskip

Let $\mathbb{S}^{d-1}_{+} = \{ x=(x', x_{d}) \in \mathbb{R}^d : |x|=1 \ \text{and} \ x_{d} >0 \}$. By using the change of variables $ z = (x', \sqrt{1-|x'|^{2}})$, we have
\begin{equation*}
    \int_{\mathbb{S}^{d-1}_{+}} f(z) \, dz = \int_{|x'|<1} f(x', \sqrt{1-|x'|^{2}}) \, \frac{dx'}{\sqrt{1-|x'|^{2}}}.
\end{equation*}

\bigskip 

\textbf{Acknowledgement:} We express our gratitude to the Department of Mathematics and Statistics at the Indian Institute of Technology Kanpur for providing a conducive research environment. For this work, Adimurthi acknowledges support from IIT Kanpur, while P. Roy is supported by the Core Research Grant (CRG/2022/007867) of SERB and the ARG MATRICS Grant (ANRF/ARGM/2025/002269/MTR) of ANRF. V. Sahu is grateful for the support received through MHRD, Government of India (GATE fellowship). P. Roy would like to thank Prof. S. Varadhan for discussions on the subject. We would like to thank the anonymous referee, whose comments helped us significantly improve the results of this article.

%\bibliographystyle{amsplain}
%\bibliography{references}

\end{document}